\documentclass[11pt]{article}
\usepackage[dvipsnames, table]{xcolor}

  \usepackage[pdftex]{graphicx}

\usepackage{epstopdf}

\usepackage[sort,nocompress]{cite}
\usepackage{microtype}
\usepackage{algorithm}

\usepackage[caption=false, format=hang]{subfig}
\makeatletter
\renewcommand{\p@subfigure}{\thefigure-}
\makeatother

\usepackage{multirow}   		
\usepackage{booktabs}   		
\usepackage{array}
\newcolumntype{+}{>{\global\let\currentrowstyle\relax}}
\newcolumntype{^}{>{\currentrowstyle}}
\newcommand{\rowstyle}[1]{\gdef\currentrowstyle{#1}#1\ignorespaces}
\newcommand{\brow}{\rowstyle{\bfseries}}

\expandafter\let\csname equation*\endcsname\relax
\expandafter\let\csname endequation*\endcsname\relax
\usepackage[cmex10]{amsmath}
\usepackage{amsfonts, amssymb, amsthm, mathtools}
\usepackage{empheq}
\newcommand{\NN}{\ensuremath{\mathbb N}}
\newcommand{\ZZ}{\ensuremath{\mathbb Z}}
\newcommand{\RR}{\ensuremath{\mathbb R}}

\newcommand{\HH}{\ensuremath{{\mathcal H}}}
\newcommand{\RP}{\ensuremath{\left[0,+\infty\right[}}
\newcommand{\RPPbis}{\ensuremath{\left]0,+\infty\right[}}
\newcommand{\RX}{\ensuremath{\left]-\infty,+\infty\right]}}
\newcommand{\DD}{\ensuremath{\mathbb D}}

\DeclarePairedDelimiter{\parens}{(}{)}
\DeclarePairedDelimiter{\bracks}{[}{]}
\DeclarePairedDelimiter{\braces}{\{}{\}}

\DeclarePairedDelimiter{\abs}{\lvert}{\rvert}
\DeclarePairedDelimiter{\norm}{\lVert}{\rVert}

\newcommand{\minimize}[2]{\ensuremath{\underset{\substack{{#1}}}{\operatorname{minimize}}\;\;#2}}
\newcommand{\scal}[2]{{\left\langle{{#1}\mid{#2}}\right\rangle}}
\newcommand{\menge}[2]{\big\{{#1}~\big |~{#2}\big\}} 
\newcommand{\emp}{\ensuremath{{\varnothing}}}
\newcommand{\lev}[1]{{\ensuremath{{{{\operatorname{lev}}}_{\leq #1}}\,}}}
\newcommand{\pinf}{\ensuremath{{+\infty}}}
\DeclareMathOperator{\subto}{s.t.}
\newcommand{\bell}{\boldsymbol{\ell}}
\newcommand{\epi}{\operatorname{epi}}
\newcommand{\Id}{\ensuremath{\operatorname{Id}}\,}
\newcommand{\dom}{\ensuremath{\operatorname{dom}}}
\newcommand{\prox}{\ensuremath{\operatorname{prox}}}

\newcommand{\reli}{\ensuremath{\operatorname{ri}}}
\newtheorem{theorem}{Theorem}[section]

\newtheorem{corollary}[theorem]{Corollary}
\newtheorem{proposition}[theorem]{Proposition}

\theoremstyle{definition}
\newtheorem{problem}[theorem]{Problem}			
\begin{document}

\title{Epigraphical splitting for solving constrained convex formulations of inverse problems with proximal tools}

\author{G.~Chierchia\thanks{Institut Mines-T\'el\'ecom, T\'el\'ecom ParisTech, CNRS LTCI, 75014 Paris, France}, N.~Pustelnik\thanks{ENS Lyon, Laboratoire de Physique, UMR CNRS 5672, F69007 Lyon, France}, J.-C. Pesquet\thanks{Universit\'e Paris-Est, LIGM, UMR CNRS 8049, 77454 Marne-la-Vall\'ee, France}, and B.~Pesquet-Popescu\footnotemark[1]}

\maketitle

\begin{abstract}
We propose a proximal approach to deal with a class of convex variational problems involving nonlinear constraints. A large family of constraints, proven to be effective in the solution of inverse problems, can be expressed as the lower level set of a sum of convex functions evaluated over different, but possibly overlapping, blocks of the signal. For such constraints, the associated projection operator generally does not have a simple form. We circumvent this difficulty by splitting the lower level set into as many epigraphs as functions involved in the sum. A closed half-space constraint is also enforced, in order to limit the sum of the introduced epigraphical variables to the upper bound of the original lower level set. In this paper, we focus on a family of constraints involving linear transforms of distance functions to a convex set or $\bell_{1,p}$ norms with $p\in \{1,2,\pinf\}$. In these cases, the projection onto the epigraph of the involved function has a closed form expression.

The proposed approach is validated in the context of image restoration with missing samples, by making use of constraints based on Non-Local Total Variation. Experiments show that our method leads to significant improvements in term of convergence speed over existing algorithms for solving similar constrained problems. A second application to a pulse shape design problem is provided in order to illustrate the flexibility of the proposed approach.
\end{abstract}


\section{Introduction}\label{sec:into}

As an offspring of the wide interest in frame representations and sparsity promoting techniques for data recovery, proximal methods have become popular for solving large-size non-smooth convex optimization problems \cite{Rockafellar_R_1976_j-siam-jco, Combettes_P_2010_inbook_proximal_smsp, Bach_F_2012_j-ftml_opt_sip}. The efficiency of these methods in the solution of inverse problems has been widely studied in the recent signal and image processing literature (see for instance \cite{Guerquin_M_2011_j-ieee-tmi_fas_wbr, Dupe_FX_2008_ip_proximal_ifdpniusr, Aujol_JF_2006_ijcv_structure_tidmaps, Briceno_L_2011_j-math-imaging-vis_pro_ami, Theodoridis_S_2011_j-ieee-spm_adaptive_lwp, Chaux_C_2012_j-miv_parallel_psm} and references therein). Even if proximal algorithms and the associated convergence properties have been deeply investigated 
\cite{Combettes_PL_2005_j-siam-mms_Signal_rbpfbs, Chaux_C_2007_j-ip_variational_ffbip, Combettes_PL_2007_istsp_Douglas_rsatncvsr, Combettes_PL_2008_j-ip_proximal_apdmfscvip}, 
some questions persist in their use for solving inverse problems. A first question is: how can we set the parameters serving to enforce the regularity of the solution in an automatic way? Various strategies were proposed in order to address this question \cite{Galatsanos_N_1992_tip_methods_crpenvirtr, Hansen_P_1993_j-siam-sci-comp_use_lcr, Pizurica_A_2006_j-ieee-tip_est_pps, Ramani_S_2008_tip_monte_csbborpgda, Chaari_L_2010_j-tsp_hierarchical_bmfr}, but the computational cost of these methods is often high, especially when several regularization parameters have to be set. Alternatively, it has been recognized for a long time that incorporating constraints directly 
on the solutions \cite{Youla_DC_1982_tmi_POCS_irbtmopocs, Trussell_H_1984_tassp_feasible_ssp, Combettes_PL_1994_tsp_Inconsistent_sfplssiaps, Kose_K_2012_p-icassp_fil_vmd, Teuber_T_2012_j-inv-prob_minimization_pes, Stuck_R_2012_j-ip_iteratively_rgn, Ono_S_2013_p-icassp_poisson_irl}, instead of considering regularized functions, may often facilitate the choice of the involved parameters. Indeed, in a constrained formulation, the constraint bounds are usually related to some physical properties of the target solution or some knowledge of the degradation process, e.g. the noise statistical properties. Note also that there exist some conceptual Lagrangian equivalences between regularized solutions to inverse problems and constrained ones, although some caution should be taken when the regularization functions are nonsmooth (see \cite{Steidl_G_2012_j-math-imaging-vis_hom_pcc} where the case of a single regularization parameter is investigated).

Another question is related to the selection of the most appropriate algorithm within the class of proximal methods according to a given application. This also raises the question of the computation of the proximity operators associated with the different functions involved in the criterion. In this context, the objective of this paper is to propose an efficient splitting technique for solving some constrained convex optimization problems of the form:
\begin{problem} \label{p:gen}
\begin{equation}
\underset{x\in \HH}{\mathrm{minimize}}\; \sum_{r=1}^R g_r(T_rx)\quad\subto\quad 
\begin{cases}
H_1x\in C_1,\\
\dots\\
H_Sx\in C_S,
\end{cases}
\end{equation}
where $\HH$ is a real Hilbert space, and
\begin{enumerate}
\item for every $r\in\{1,\ldots,R\}$, $T_r$ is a bounded linear operator from $\HH$ to $\RR^{N_r}$, 
\item for every $r\in\{1,\ldots,R\}$, $g_r\colon \RR^{N_r}\mapsto \RX$ is a proper lower-semicontinuous convex function,
\item for every $s\in \{1,\ldots,S\}$, $H_s$ is a bounded linear operator from $\HH$ to $\RR^{M_s}$,
\item for every $s \in \{1,\ldots,S\}$, $C_s$ is a nonempty closed convex subset of $\RR^{M_s}$.
\end{enumerate}
\end{problem}

More precisely, the present work aims at designing a method to address Problem~\ref{p:gen} when some convex constraints are expressed as follows: for some $s\in \{1,\ldots,S\}$,
\begin{equation}\label{e:levconstintro}
(\forall x\in \HH)\qquad H_s x \in C_s \quad\Leftrightarrow\quad h_s(H_s x) \le \eta_s,
\end{equation}
where $\eta_s \in \RR$ and $h_s$ is a proper lower-semicontinuous convex function from $\RR^{M_s}$ to $\RX$. Indeed, the projection onto the convex set $C_s$ defined in \eqref{e:levconstintro} often does not have a closed form expression. In the present work, we will show that: 
\begin{enumerate}
\item when the function $h_s$ in \eqref{e:levconstintro} corresponds to a \emph{decomposable loss}, i.e.\ it can be expressed as the sum of functions evaluated over different blocks of the vector $H_sx$, the problem of computing the projection onto the associated convex set $C_s$ can be addressed by resorting to a splitting approach that decomposes the set $C_s$ into a collection of epigraphs and a half-space;
\item the projection operator associated with an epigraph (namely the \emph{epigraphical projection}) has a closed form for some functions of practical interest, such as the absolute value raised to a power $q\in [1,+\infty[$, the distance to a convex set and the $\bell_{p}$-norm with $p\in \{2,\pinf\}$;
\item in the context of image restoration, regularity constraints based on Total Variation \cite{Rudin_L_1992_tv_atvmaopiip} and Non-Local Total Variation \cite{Gilboa_G_2009_j-siam-mms_nonlocal_oai} can be efficiently handled by the proposed epigraphical splitting, which significantly speeds up the convergence (in terms of execution time) with respect to standard iterative solutions \cite{VanDenBerg_E_2008_j-siam-sci-comp_pro_pfb,Quattoni_A_2009_p-icml_efficient_plr}.
\end{enumerate}

The paper is organized as follows. In Section~\ref{sec:rec}, we review the algorithms which are applicable for solving large-size convex optimization problems, so motivating the choice of proximal methods, and we review the variable-splitting techniques commonly used with these methods. In order to deal with a constraint expressed under the form \eqref{e:levconstintro}, we propose in section \ref{sec:proposed} a novel splitting approach involving an epigraphical projection. In addition, closed form expressions for specific epigraphical projections are given. Experiments in two different contexts are presented in Section~\ref{sec:exp}. The first ones concern an image reconstruction problem, while the second ones
are related to pulse shape design for digital communications. Finally, some conclusions are drawn in Section \ref{sec:con}.\\

\noindent\textbf{Notation}: Let $\HH$ be a real Hilbert space endowed with the norm $\|\cdot\|$ and the scalar product $\scal{\cdot}{\cdot}$. $\Gamma_0(\HH)$ denotes the set of proper lower-semicontinuous convex functions from $\HH$ to $\RX$.
The epigraph of $\varphi\in \Gamma_0(\HH)$ is the nonempty closed convex subset of $\HH\times \RR$ defined as $\epi \varphi = \menge{(y,\zeta) \in \HH\times \RR}{\varphi(y)\le \zeta}$ and the lower level set of $\varphi$ at height $\zeta \in \RR$ is the nonempty closed convex subset of $\HH$ defined as $\lev{\zeta}\varphi = \menge{y\in \HH}{\varphi(y) \le \zeta}$. 
A subgradient of $\varphi$ at $y\in \HH$ is an element of its subdifferential defined as $\partial \varphi(y)
= \menge{t\in\HH}{(\forall u \in \HH)\;\;\varphi(u) \ge \varphi(y)+\scal{t}{u-y}}$. When  $\varphi$ is G\^ateaux-differentiable at $y$, $\partial \varphi(y) = \{\nabla\varphi(y)\}$ where $\nabla\varphi(y)$ is the gradient of $\varphi$ at $y$. Let $C$ be a nonempty closed convex subset $C$ of $\HH$. The relative interior of $C$ is denoted by $\reli C$. For every $y \in \HH$, the indicator function $\iota_C\in \Gamma_0(\HH)$ of $C$ is given by
\begin{equation}
\iota_C(y)
= \begin{cases}
0, & \mbox{if $y \in C$},\\
\pinf, & \mbox{otherwise,}
\end{cases}
\end{equation}
the projection onto $C$ reads $P_C(y) = \operatorname{argmin}_{u\in C} \|u-y\|$, and the distance to $C$ is given by $d_C(y) = \|y-P_C(y)\|$.

\section{Proximal tools}\label{sec:rec}
\subsection{From gradient descent to proximal algorithms}
\label{ssec:algprox}
The first methods for finding a solution to an inverse problem were restricted to the use of a differentiable cost function \cite{Tikhonov_A_1963_j-sov-mat-dok_tikhonov_ripp}, i.e. Problem~\ref{p:gen} where $S=0$ and, for every $r\in\{1,\ldots,R\}$, $g_r$ denotes a differentiable function. In this context, gradient-based algorithms appear to be the most efficients solutions when an iterative procedure is required (see \cite{Chouzenoux_E_2011_j-ieee-tip_majorize_mss} and references therein). However, in order to model additional properties, sparsity promoting penalizations ($R\ge 1$) or hard constraints ($S \ge 1$) may be introduced and the diffentiability property is not satisfied anymore. One way to circumvent this difficulty is to resort to smart approximations in order to smooth the involved non-differentiable functions \cite{Martinet_B_1970_j-iro_reg_iva, Aubert_G_1980_j-boll_minimisation_fnc, BenTal_A_1989_j-lnm_smoothing_tno, Hiriart_Urruty_1996_book_convex_amaIf}. If one wants to address the original nonsmooth problem without approximation errors, one may apply some specific algorithms \cite{Tseng_P_2001_j-ota_conv_bcd}, the convergence of which is guaranteed under restrictive assumptions. Interior point methods \cite{Wright_S_1997_book_primal_dip} can also be employed for small to medium size optimization problems.

On the other hand, in order to solve convex feasibility problems, i.e. to find a vector belonging to the intersection of convex sets (Problem~\ref{p:gen} with $R=0$), iterative projection methods were developed. The projection onto convex sets algorithm (POCS) is one of the most popular approach to solve data recovery problems \cite{Bregman_LM_1965_sm_POCS_tmospfacpocs, Gurin_LG_1967_zvmmf_Projection_mffacpocs, Youla_DC_1982_tmi_POCS_irbtmopocs, Combettes_P_1993_pieee_fou_ste}. A drawback of POCS is that it is not well-suited for parallel implementations. The Parallel Projection Method (PPM) and Method of Parallel Projections (MOPP) are variants of POCS making use of parallel projections. Moreover, these algorithms were designed to efficiently solve inconsistent feasibility problems (when the intersection of the convex set is empty). Thorough comparisons between projection methods have been performed in \cite{Combettes_P_1997_j-ieee-tip_con_sti, Censor_Y_2012_j-comp-opt-appl_effectiveness_pmc}.

Computing the projection $P_C$ onto a nonempty closed convex subset $C$ of a real Hilbert space $\HH$ requires to solve a constrained quadratic minimization problem. 
However, it turns out that a closed form expression of the solution to this problem is available in a limited number of instances. Some well-known examples are the projections onto hyperplanes, closed half-spaces and $\bell_2$-norm balls \cite{Hiriart_Urruty_1996_book_convex_amaIf, Bauschke_H_2011_book_con_amo}.
When an expression of the direct projection is not available, the convex set $C$ can be approximated by a half-space, which leads to the concept of subgradient projection. 
An efficient block iterative surrogate splitting method was proposed in \cite{Combettes_PL_2003_tsp_Block_abiscsmfqsr} in order to solve Problem~\ref{p:gen} when $S \ge 1$ and, for every $r\in\{1,\ldots,R\}$, $g_r = \|\cdot - z_r\|^2$ where $z_r \in \RR^{N_r}$. A main limitation of this method is that the global objective function must be strictly convex. For recent works about subgradient projection methods, the readers may refer to \cite{Yamada_I_2001_inbook_hybrid_sdm,Slavakis_K_2006_j-nfao_adaptive_psm, Bouboulis_P_2012_j-ieee-tnnls_adaptive_lcr}.

A way to overcome this difficulty consists of considering proximal approaches. The key tool in these methods is the proximity operator \cite{Moreau_J_1965_bsmf_Proximite_eddueh} of a function $\varphi \in \Gamma_0(\HH)$, defined as
\begin{equation}
(\forall y \in \HH)\qquad
\prox_\varphi(y)= \underset{u\in \HH}{\operatorname{argmin}} \frac12 \|u-y\|^2 +
\varphi(u).
\end{equation}
The proximity operator can be interpreted as a sort of subgradient step for the function $\varphi$, as $p = \prox_\varphi(y)$ is uniquely defined through the inclusion
\begin{equation}\label{e:subdiprox}
y - p \in \partial \varphi(p).
\end{equation}
Proximity operators enjoy many interesting properties \cite{Chaux_C_2007_j-ip_variational_ffbip}. In particular, they generalize the notion of projection onto a closed convex set $C$, in the sense that $\prox_{\iota_C} = P_C$. Hence, proximal methods provide a unifying framework that allows one to address non-smooth penalizations ($R \ge 1$) and hard constraints ($S \ge 1$).

The class of proximal methods includes primal algorithms \cite{Daubechies_I_2004_cpamath_iterative_talipsc, Chaux_C_2007_j-ip_variational_ffbip, Combettes_PL_2007_istsp_Douglas_rsatncvsr, Figueiredo_M_2007_j-ieee-sel-topics-sp_gra_psr, Beck_A_2009_j-siam-is_fast_istalip, Fornasier_M_2009_j-sjna_subspace_cmtvl1m, Steidl_G_2010_j-math-imaging-vis_removing_mndrsm, Combettes_P_2010_inbook_proximal_smsp, Pesquet_J_2012_j-pjpjoo_par_ipo} 
and primal-dual algorithms \cite{Chen_G_1994_j-mp_pro_bdm, Esser_E_2010_j-siam-is_gen_fcf, Chambolle_A_2010_first_opdacpai, Briceno_L_2011_j-siam-opt_mon_ssm, Combettes_P_2011_j-svva_pri_dsa, Vu_B_2011_j-acm_spl_adm, Condat_L_2012,Chen_2013_ip_primal_dual_fpa}. Primal algorithms generally require to inverse some linear operators (e.g.\ $\sum_{r=1}^R T_r^* T_r + \sum_{s=1}^S H_s^* H_s$), while primal-dual ones only require to compute $(T_r)_{1\le r \le R}$, $(H_s)_{1\le s \le S}$ and their adjoints. Consequently, primal-dual methods are often easier to implement than primal ones, but their convergence may be slower \cite{Pustelnik_N_2012_j-ieee-tsp_rel_tfc,Couprie_C_2012}. Note also that some of these methods are closely related to augmented Lagrangian approaches 
\cite{Setzer_S_2009_j-jvcir_deblurring_pibsbt, Figueiredo_M_2010_t-ip_restoration_piado}.

\subsection{Variable-splitting techniques}

With the development of proximal methods for solving convex optimization problems, many techniques have been developed to cope with the intrinsic limitations of these methods. In particular, the computation of the proximity operator becomes intractable 
in general when considering the sum of several functions or a function composed with a linear operator. In this context, variable-splitting constitutes a very effective way to design easily implementable algorithms. One of the most popular examples is given by the approaches inspired by the Alternating Direction Method of Multipliers (ADMM), which deal with optimization problems of the form
\begin{equation}
\label{p:split}
\minimize{x\in \HH} g_1(T_1 x) + g_2(x).
\end{equation}
By introducing an auxiliary variable $v\in \RR^{N_1}$, the problem is reformulated as
\begin{equation}
\minimize{(x,v) \in \HH \times \RR^{N_1}} g_1(v) + g_2(x) \quad\subto\quad T_1x = v.
\end{equation}
 This kind of splitting technique has been often used in image restoration \cite{Afonso_M_2009_j-tip_augmented_lacofiip, Setzer_S_2009_j-jvcir_deblurring_pibsbt} and more recently for distributed optimization problems \cite{Boyd2011_j-found-tml_distributed_osl_admm}. A similar form of splitting has been considered in \cite{Briceno_L_2011_j-math-imaging-vis_pro_ami, Peyre_G_2011_p-eusipco_gro_sop}, where the constraint $T_1 x = v$ is handled by computing the projection onto the nullspace of the linear operator  $[T_1 \;\;-\Id]$, which has a closed-form expression for some specific choices of $T_1$, such as circulant matrices involved in image restoration.

The solution that we will propose in this work also introduces auxiliary variables. However,
our objective is not to deal with linear transformations of the data but with a projection which does not have a closed-form expression. Consequently, the proposed solution departs from the usual splitting methods, in the sense that our approach leads to a collection of epigraphs and a half-space constraint sets, while the usual splitting techniques yield linear constraints.

\section{Proposed method}\label{sec:proposed}\label{sec:spl}
We now turn our attention to convex sets for which the associated projection does not have a closed form and we show 
that, under some appropriate assumptions, it is possible to circumvent this difficulty. In Problem \ref{p:gen}, assume that $C_1$ denotes such a constraint and that it can be modelled as: for every $y \in \RR^{M_1}$,
\begin{equation}\label{e:levconst}
y \in C_1 \quad\Leftrightarrow\quad
h_1(y) = \sum_{\ell = 1}^{L_1} h_1^{(\ell)}({\sf y}^{(\ell)}) \le \eta_1
\end{equation}
where $\eta_1 \in \RR$. Hereabove, the generic vector $y$ has been decomposed into blocks of coordinates as follows
\begin{equation}\label{e:decblocy}
y = [\underbrace{({\sf y}^{(1)})^\top}_{\text{size}\,M_1^{(1)}},\ldots,\underbrace{({\sf y}^{(L)})^\top}_{\text{size}\,M_1^{(L_1)}}]^\top
\end{equation}
and, for every $\ell \in \{1,\ldots,L_1\}$, ${\sf y}^{(\ell)} \in \RR^{M_1^{(\ell)}}$
and $h_1^{(\ell)}$ is a function in $\Gamma_0(\RR^{M_1^{(\ell)}})$ such that $\reli(\dom h_1^{(\ell)}) \neq \emp$. 

\subsection{Epigraphical splitting}\label{sec:epi_split}
The idea underlying our approach consists of introducing an auxiliary vector ${\zeta_1} = \big(\zeta_1^{(\ell)}\big)_{1 \le \ell \le L_1} \in \RR^{L_1}$, so that Constraint~\eqref{e:levconst} can be equivalently rewritten as$^*$ \footnote[0]{$^*$Note that the inequality in \eqref{e:const1} can be also replaced by an equality, even though it makes little difference in our approach.}
\begin{empheq}[left=\left\{,right=\right.]{align}
&\sum_{\ell=1}^{L_1} \zeta_1^{(\ell)} \le \eta_1,\label{e:const1}\\
& (\forall \ell \in \{1,\ldots,L_1\})\qquad h_1^{(\ell)}({\sf y}^{(\ell)}) \le \zeta_1^{(\ell)}\label{e:const2}.
\end{empheq}
Let us now introduce the closed half-space of $\RR^{L_1}$ defined as 
\begin{equation}\label{e:defV}
V_1 = \menge{\zeta\in \RR^{L_1}}{{\sf 1}_{L_1}^\top\zeta\le \eta_1},
\end{equation}
with ${\sf 1}_{L_1} = (1,\ldots,1)^\top \in \RR^{L_1}$, and the closed convex set
\begin{equation}\label{e:defE}
E_1 = \big\{ {(y,\zeta)\in \RR^{M_1}\times \RR^{L_1}}~\big|~{(\forall \ell \in \{1,\dots,L_1\})\;} 
({\sf y}^{(\ell)},\zeta^{(\ell)})\in \epi h_1^{(\ell)}
\big\}.
\end{equation}
Then, Constraint~\eqref{e:const1} means that $\zeta_1 \in V_1$, while Constraint~\eqref{e:const2} is equivalent to $(y,\zeta_1)\in E_1$. In other words, the constraint $C_1$ can be split into the two constraints $V_1$ and $E_1$, provided that an additional vector $\zeta_1 \in \RR^{L_1}$ is introduced in Problem~\ref{p:gen}. The resulting criterion takes the form:
\begin{problem}\label{p:epi}
\begin{equation}
\underset{(x,\zeta_1)\in \HH \times V_1}{\mathrm{minimize}}\; \sum_{r=1}^R g_r(T_rx)\quad\subto\quad 
\begin{cases}
(H_1x,\zeta_1)\in E_1\\
H_2 x\in C_2\\
\quad \vdots\\
H_Sx\in C_S.
\end{cases}
\end{equation}
\end{problem}
Note that the additional constraints can be easily handled by proximal algorithms as far as the projections onto the associated constraint sets can be computed. In the present case, the projection onto $V_1$ is well-known \cite{Hiriart_Urruty_1996_book_convex_amaIf}, whereas the projection onto $E_1$ is given by 
\begin{equation}
(\forall (y,\zeta)\in \RR^{M_1}\times \RR^{L_1})\qquad
P_{E_1}(y,\zeta) = (p,\theta) 
\end{equation}
where $\theta = (\theta^{(\ell)})_{1\le\ell\le L_1}$, vector $p\in \RR^{M_1}$ is blockwise decomposed as $p = [({\sf p}^{(1)})^\top,\ldots, ({\sf p}^{(L_1)})^\top]^\top$ like in \eqref{e:decblocy},  and
\begin{equation}\label{eq:P_epi}
(\forall \ell \in \{1,\ldots,L_1\})\quad
({\sf p}^{(\ell)},\theta^{(\ell)}) = P_{\epi h_1^{(\ell)}}({\sf y}^{(\ell)},\zeta^{(\ell)}).
\end{equation}
Hence, the problem reduces to the lower-dimensional problem of the determination of the projection onto the convex subset $\epi h_1^{(\ell)}$ of $\RR^{M_1^{(\ell)}} \times \RR$ for each $\ell \in \{1,\ldots,L_1\}$. 
An example of an algorithm that converges to a solution to Problem~\ref{p:epi} will be provided in Section~\ref{s:algo}.

\subsection{Proximity operators: new closed forms}
\label{ss:newprox}
The key point in the proposed splitting is the introduction of some epigraphs in the minimization process. In the following, we provide some results concerning the projection onto the epigraph of a convex function. 

\begin{proposition}\label{p:newprox}
Let $\HH$ be a real Hilbert space and let $\HH \times \RR$ be equipped with the standard product space norm. Let $\varphi$ be a function in $\Gamma_0(\HH)$ such that $\dom\varphi$ is open. The projector $P_{\epi \varphi}$ onto the epigraph of $\varphi$ is given by: 
\begin{equation}
\big(\forall (y,\zeta)\in \HH\times \RR\big) \qquad P_{\epi \varphi}(y,\zeta) = (p,\theta)
\end{equation}
 where
\begin{equation}\label{e:proxprojS}
\begin{cases}
p &= \prox_{\frac12 (\max\{\varphi-\zeta,0\})^2}(y),\\ 
\theta &= \max\{\varphi(p),\zeta\}.
\end{cases}
\end{equation}
\end{proposition}

\begin{proof}
For every $(y,\zeta)\in \HH\times \RR$, let $(p,\theta)=P_{\epi \varphi}(y,\zeta)$. 
If $\varphi(y) \le \zeta$, then $p = y$ and $\theta = \zeta =
\max\{\varphi(p),\zeta\}$. In addition,
\begin{align}
(\forall u \in \HH)\quad
0 &=\frac12 \|p-y\|^2 + \frac12 \big(\max\{\varphi(p)-\zeta,0\}\big)^2\nonumber\\
&\le \frac12 \|u-y\|^2 + \frac12 \big(\max\{\varphi(u)-\zeta,0\}\big)^2,
\end{align}
which shows that \eqref{e:proxprojS} holds. 
Let us now consider the case when $\varphi(y) > \zeta$. From the definition of the projection, we get
\begin{equation}\label{e:projnew}
(p,\theta) = \underset{(u,\xi)\in \epi \varphi}{\operatorname{argmin}}\; \|u-y\|^2+(\xi-\zeta)^2.
\end{equation}
From the Karush-Kuhn-Tucker theorem \cite[Theorem~5.2]{Ekeland_I_1999_book_Convex_aavp},$^*$ \footnote[0]{$^*$By considering $u_0
  \in \dom \varphi$ and $\xi_0 >\varphi(u_0)$, the required
  qualification condition is obviously satisfied.}
there exists $\alpha \in\RP$
such that
\begin{equation}\label{e:Lagrange}
(p,\theta) = \underset{(u,\xi)\in \HH\times \RR}{\operatorname{argmin}}\;
\frac12 \|u-y\|^2+ \frac12 (\xi-\zeta)^2+ \alpha(\varphi(u)-\xi)
\end{equation}
where the Lagrange multiplier $\alpha$ is such that $\alpha(\varphi(p)-\theta)=0$. 
Since the value $\alpha = 0$ is not allowable (since it would lead to $p=y$ and $\theta = \zeta$), it can be deduced from
the above equality that 
$\varphi(p) = \theta$.
In addition, differentiating the Lagrange functional in \eqref{e:Lagrange}
w.r.t. $\xi$ yields
\begin{equation}
\varphi(p) = \theta = \zeta + \alpha \geq \zeta.
\end{equation}
Hence, $(p,\theta)$ given by \eqref{e:projnew} is such that
\begin{align}
p &= \underset{\substack{u\in \HH\\\varphi(u)\ge \zeta}}{\operatorname{argmin}}
\|u-y\|^2+(\varphi(u)-\zeta)^2\label{e:texp1}\\
\theta &= \varphi(p) = \max\{\varphi(p),\zeta\}.
\end{align}
Furthermore, as $\varphi(y) > \zeta$, we have
\begin{equation}\label{e:texp0}
\underset{\substack{u\in \HH\\\varphi(u)\le \zeta}}{\inf} \|u-y\|^2 
= \|P_{\lev{\zeta}\varphi }(y)-y\|^2
= \underset{\substack{u\in \HH\\\varphi(u)= \zeta}}{\inf} \|u-y\|^2
\end{equation}
where 
we have used the fact that
$P_{\lev{\zeta}\varphi}(y)$ belongs to the
boundary of $\lev{\zeta}\varphi$ which is equal
to $\menge{u \in\HH}{\varphi(u)= \zeta}$ since $\varphi$ is
lower-semicontinuous and $\dom \varphi$ is open \cite[Corollary 8.38]{Bauschke_H_2011_book_con_amo}.
We have then
\begin{align}\label{e:texp2}
\underset{\substack{u\in \HH\\\varphi(u)\le \zeta}}{\inf} \|u-y\|^2 
&= \underset{\substack{u\in \HH\\\varphi(u)= \zeta}}{\inf} \|u-y\|^2\nonumber\\
&\ge \underset{\substack{u\in \HH\\\varphi(u)\ge \zeta}}{\inf}
\|u-y\|^2+(\varphi(u)-\zeta)^2.
\end{align}
Altogether, \eqref{e:texp1} and \eqref{e:texp2} lead to
\begin{equation}
p = \underset{u\in \HH}{\operatorname{argmin}}\;
\frac12\|u-y\|^2+\frac12\big(\max\{\varphi(u)-\zeta,0\}\big)^2
\end{equation}
which is equivalent to \eqref{e:proxprojS} since $\frac12\big(\max\{\varphi-\zeta,0\}\big)^2 \in \Gamma_0(\HH)$.
\end{proof}

Note that alternative characterizations of the epigraphical projection 
can be found in \cite[Propositions~9.17, 28.28]{Bauschke_H_2011_book_con_amo}.

From the previous proposition, we see that the proximity operator in \eqref{e:proxprojS} plays a prominent role in the calculation of the projection onto $\epi \varphi$. We now provide an example of function $\varphi$ for which this proximity operator admits a simple form.
\begin{proposition}\label{p:valbeta}
Let $\beta \in [1,\pinf[$, $\tau\in \RPPbis$. Assume that
\begin{equation}
(\forall y \in \RR) \qquad \varphi(y) = \tau |y|^{\beta}.
\end{equation}
If $\zeta \in ]-\infty,0]$, then for every $y \in \RR$
\begin{equation}
\prox_{\frac12 (\max\{\tau |\cdot|^{\beta}-\zeta,0\})^2}(y) = 
\begin{cases}
\displaystyle \frac{\operatorname{sign}(y)}{1+\tau^2}  \max\{|y|+\tau\zeta,0\}, & \mbox{if
  $\beta = 1$,}\\
\operatorname{sign}(y) \chi_0, & \mbox{if $\beta > 1$,}
\end{cases}
\end{equation}
where $\chi_0$ is the unique solution on $\RP$ of the equation
\begin{equation}\label{e:poly}
\beta \tau^2\chi^{2\beta-1} - \beta \tau \zeta \chi^{\beta-1}+\chi = |y|.
\end{equation}
If $\zeta \in \RPPbis$, then, for every $y \in \RR$,
\begin{equation}\label{e:proxbetap}
\prox_{\frac12 (\max\{\tau |\cdot|^{\beta}-\zeta,0\})^2}(y) = 
\begin{cases}
y, & \mbox{if $\tau|y|^{\beta}\le \zeta$,}\\
\operatorname{sign}(y) \chi_\zeta, & \mbox{otherwise,}
\end{cases}
\end{equation}
where 
$\chi_\zeta$ is the unique solution on {\small$[(\frac{\zeta}{\tau})^{1/\beta},\pinf[$} of \eqref{e:poly}.
\end{proposition}

\begin{proof}
Since $(\max\{\varphi-\zeta,0\})^2$ is an even function, $\prox_{\frac12 (\max\{\varphi-\zeta,0\})^2}$ is an odd
function \cite[Remark 4.1(ii)]{Chaux_C_2007_j-ip_variational_ffbip}. In the following, we thus focus
on the case when $y \in \RPPbis$. If $\zeta \in ]-\infty,0]$, then $(\max\{\varphi-\zeta,0\})^2=
(\varphi-\zeta)^2$. When $\beta = 1$, $\frac12(\max\{\varphi-\zeta,0\})^2=
(\tau^2/2)(\cdot)^2- \tau \zeta |\cdot|+\zeta^2/2$ and, from \cite[Example 4.6]{Chaux_C_2007_j-ip_variational_ffbip}, it can be deduced that
\begin{equation}
\prox_{\frac12 (\max\{\varphi-\zeta,0\})^2}(y) = 
\frac{1}{1+\tau^2} \max\{y+\tau\zeta,0\}.
\end{equation}
When $\beta > 1$, $(\varphi-\zeta)^2$ is differentiable and, according to \eqref{e:subdiprox}, 
$p = \prox_{\frac12 (\max\{\varphi-\zeta,0\})^2}(y)$ is uniquely defined as
\begin{equation}
p-y + \beta\tau p^{\beta-1} (\tau p^\beta-\zeta) = 0 \label{e:betaprox}
\end{equation}
where, according to \cite[Corollary 2.5]{Combettes_PL_2007_jopt_Proximal_tafmoob}, $p \ge 0$.
This allows us to deduce that $p = \chi_0$.

Let us now focus on the case when $\zeta \in ]0,\pinf[$. If $y\in
]0,(\zeta/\tau)^{1/\beta}[$, it can be deduced from~\cite[Corollary~2.5]{Combettes_PL_2007_jopt_Proximal_tafmoob}, that $p = \prox_{\frac12 (\max\{\varphi-\zeta,0\})^2}(y)
\in [0,(\zeta/\tau)^{1/\beta}[$. Since $(\forall v \in
[0,(\zeta/\tau)^{1/\beta}[)$ $\max\{\varphi(v)-\zeta,0\}=0$, \eqref{e:subdiprox}
yields $p=y$. On the other hand if $y > (\zeta/\tau)^{1/\beta}$, as the
proximity operator of a function from $\RR$ to $\RR$ is continuous and
increasing \cite[Proposition 2.4]{Combettes_PL_2007_jopt_Proximal_tafmoob}, 
$p = \prox_{\frac12 (\max\{\varphi-\zeta,0\})^2}(y) \ge \prox_{\frac12
  (\max\{\varphi-\zeta,0\})^2}\big((\zeta/\tau)^{1/\beta}\big) =(\zeta/\tau)^{1/\beta}$.
Since 
$(\max\{\varphi-\zeta,0\})^2$ is differentiable in this case, and
$(\forall v \ge (\zeta/\tau)^{1/\beta})$
$(\max\{\varphi(v)-\zeta,0\})^2=(\tau v^\beta-\zeta)^2$, \eqref{e:subdiprox} allows us
to deduce that $p$ is the unique value in $[(\zeta/\tau)^{1/\beta},\pinf[$ satisfying 
\eqref{e:betaprox}. It can be concluded that, when $\zeta \in
]0,\pinf[$, \eqref{e:proxbetap} holds. 
\end{proof}

Note that, when $\beta$ is a rational number, \eqref{e:poly} is equivalent to a polynomial equation for which either closed form solutions are known or standard numerical solutions exist.

\subsection{Examples of epigraphical projections}

The previous propositions allow us to establish the following results concerning the epigraphical projection in \eqref{eq:P_epi}:$^*$ \footnote[0]{$^*$We drop the subscript from $h_1$, $M_1$ and $L_1$ in order to relieve the notations.}
\begin{itemize}
\item Distance functions of the form
\begin{equation}
\big(\forall {\sf y}^{(\ell)} \in \RR^{M^{(\ell)}}\big)
\qquad h^{(\ell)}({\sf y}^{(\ell)}) = \tau^{(\ell)} d_{C^{(\ell)}}^{\beta^{(\ell)}}({\sf y}^{(\ell)})\label{e:foncdistl}
\end{equation}
where $\ell \in \{1,\ldots,L\}$, $\tau^{(\ell)} \in \RPPbis$, $\beta^{(\ell)} \in [1,\pinf]$, and $C^{(\ell)}$ is a nonempty closed convex subset of $\RR^{M^{(\ell)}}$. 
\begin{proposition}\label{ex:epidistl}
Assume that $h^{(\ell)}$ is given by \eqref{e:foncdistl}. Then, for every 
$({\sf y}^{(\ell)},\zeta^{(\ell)})\in \RR^{M^{(\ell)}}\times \RR$,
$P_{\epi h^{(\ell)}}({\sf y}^{(\ell)},\zeta^{(\ell)}) = ({\sf p}^{(\ell)},\theta^{(\ell)})$
where
\begin{equation}\label{e:prodqdS}
{\sf p}^{(\ell)}= 
\begin{cases}
{\sf y}^{(\ell)}, &\mbox{if ${\sf y}^{(\ell)} \in C^{(\ell)}$,}\\
\alpha^{(\ell)}{\sf y}^{(\ell)}+(1-\alpha^{(\ell)})P_{C^{(\ell)}}({\sf y}^{(\ell)}), & \mbox{otherwise,}
\end{cases}
\end{equation}
and $\theta^{(\ell)} = \max\{\tau^{(\ell)} d_{C^{(\ell)}}^{\beta^{(\ell)}}({\sf p}^{(\ell)}),\zeta^{(\ell)}\}$, 
where 
\begin{equation}
\alpha^{(\ell)} = \frac{\prox_{\frac12 (\max\{\tau^{(\ell)} |\cdot|^{\beta^{(\ell)}}-\zeta^{(\ell)},0\})^2}\big(d_{C^{(\ell)}}({\sf y}^{(\ell)})\big)}{d_{C^{(\ell)}}({\sf y}^{(\ell)})}
\end{equation}
and  the expression of $\prox_{\frac12 (\max\{\tau^{(\ell)} |\cdot|^{\beta^{(\ell)}}-\zeta^{(\ell)},0\})^2}$ is provided
by Proposition {\rm\ref{p:valbeta}}.
\end{proposition}
\begin{proof}

Let us notice that
$\frac12 (\max\{\tau^{(\ell)} d_{C^{(\ell)}}^{\beta^{(\ell)}}-\zeta^{(\ell)},0\})^2 = \psi^{(\ell)} \circ d_{C^{(\ell)}}$ where
$\psi^{(\ell)} = \frac12 (\max\{\tau^{(\ell)}|\cdot|^{\beta^{(\ell)}}-\zeta^{(\ell)},0\})^2$. According to
\cite[Proposition~2.7]{Combettes_PL_2008_j-ip_proximal_apdmfscvip},  for every ${\sf y}^{(\ell)}\in\RR^{M^{(\ell)}}$,
\begin{equation}\label{e:proxdSgen}
\prox_{\psi^{(\ell)} \circ d_{C^{(\ell)}}}({\sf y}^{(\ell)}) = \\
\begin{cases}
{\sf y^{(\ell)}}, & \mbox{if ${\sf y}^{(\ell)}\in C^{(\ell)}$},\\
P_{C^{(\ell)}}({\sf y}^{(\ell)}), & \mbox{if 
$d_{C^{(\ell)}}({\sf y}^{(\ell)}) \le \max \partial \psi^{(\ell)}(0)$},\\
\displaystyle \alpha^{(\ell)} {\sf y}^{(\ell)} + & \\
\quad (1-\alpha^{(\ell)}) P_{C^{(\ell)}}({\sf y}^{(\ell)}),& \mbox{if $d_{C^{(\ell)}}({\sf y}^{(\ell)}) > \max \partial \psi^{(\ell)}(0)$}
\end{cases}
\end{equation}
where 
$\alpha^{(\ell)} =  \frac{\prox_{\psi^{(\ell)}}\big(d_{C^{(\ell)}}({\sf y}^{(\ell)})\big)}{d_{C^{(\ell)}}({\sf y}^{(\ell)})}$.
In addition, we have
\begin{equation}
\partial \psi^{(\ell)}(0) =
\begin{cases}
[\tau^{(\ell)}\zeta^{(\ell)},-\tau^{(\ell)}\zeta^{(\ell)}], & \mbox{if $\zeta^{(\ell)} < 0$ and $\beta^{(\ell)} = 1$},\\
\{0\}, & \mbox{otherwise},
\end{cases}
\end{equation}
and, according to Proposition \ref{p:valbeta}, when $\zeta^{(\ell)}
< 0$, $\beta^{(\ell)}=1$ and $d_{C^{(\ell)}}({\sf y}^{(\ell)}) \le -\tau^{(\ell)}\zeta^{(\ell)}$, $\prox_{\psi^{(\ell)}}\big(d_{C^{(\ell)}}({\sf y}^{(\ell)})\big) =
0$. These show that \eqref{e:proxdSgen} reduces to \eqref{e:prodqdS}. 
\end{proof}
\item[]
\item Euclidean norms, as particular case of distance functions in \eqref{e:foncdistl} when $\beta^{(\ell)} \equiv 1$ and $C^{(\ell)} = \{0\}$:
\begin{equation}
\big(\forall {\sf y}^{(\ell)} \in \RR^{M^{(\ell)}}\big)
\qquad h^{(\ell)}({\sf y}^{(\ell)}) = \tau^{(\ell)} \|{\sf y}^{(\ell)}\|\label{e:foncnorml}
\end{equation}
where $\ell \in \{1,\ldots,L\}$ and $\tau^{(\ell)} \in \RPPbis$. The resulting epigraphical projection is given below.

\begin{corollary}\label{ex:norm_l2}
Assume that $h^{(\ell)}$ is given by \eqref{e:foncnorml}.\\
Then, for every 
$({\sf y}^{(\ell)},\zeta^{(\ell)})\in \RR^{M^{(\ell)}}\times \RR$,
\begin{equation}
P_{\epi h^{(\ell)}}({\sf y}^{(\ell)},\zeta^{(\ell)})= 
\begin{cases}
(0,0), & \mbox{if $\|{\sf y}^{(\ell)}\| < -\tau^{(\ell)}\zeta^{(\ell)}$},\\
\displaystyle ({\sf y}^{(\ell)},\zeta^{(\ell)}), & 
\mbox{if $\|{\sf y}^{(\ell)}\| < \frac{\zeta^{(\ell)}}{\tau^{(\ell)}}$},\\
\displaystyle
\alpha^{(\ell)} \, \big({\sf y}^{(\ell)},\tau^{(\ell)} \|{\sf y}^{(\ell)}\|\big), 
& \mbox{otherwise,}
\end{cases}
\end{equation}
where $\displaystyle \alpha^{(\ell)} = \frac{1}{1+(\tau^{(\ell)})^2}\Big(1+\frac{\tau^{(\ell)}\zeta^{(\ell)}}{\|{\sf y}^{(\ell)}\|}\Big)$.
\end{corollary}
The epigraph of the Euclidean norm is the so-called Lorentz convex symmetric cone \cite{Alizadeh2001_j_math_second_order_cone,Faraut1994_book_Symmetric_cones} and
the above result is actually known in the literature \cite{Pang2002_j_math_semismooth_homeomorphisms}.
As it will be shown in Section \ref{sec:exp}, this expression of the epigraphical projection is useful to deal with
multivariate sparsity constraints \cite{Wu_J_2011_j-ieee-tip_mul_csi} or total variation bounds \cite{Rudin_L_1992_tv_atvmaopiip,Aujol_JF_2009_jmiv_firstorder_atvbir}, since such constraints
typically involve a sum of functions like \eqref{e:foncnorml} composed with linear
operators corresponding to analysis transforms or gradient operators.
\item[] 
\item Infinity norms defined as: for every $\ell \in \{1,\ldots,L\}$ and ${\sf y}^{(\ell)} = ({\sf y}^{(\ell,m)})_{1 \le m \le M^{(\ell)}}\in \RR^{M^{(\ell)}}$,
\begin{equation}\label{hellinfty}
h^{(\ell)}({\sf y}^{(\ell)}) 
= \max\left\{\frac{{|\sf y}^{(\ell,m)}|}{\tau^{(\ell,m)}}\mid 1 \le m \le M^{(\ell)}\right\}.
\end{equation}
where  $(\tau^{(\ell,m)})_{1 \le m \le M^{(\ell)}} \in \RPPbis^{M^{(\ell)}}$. 
\begin{proposition}\label{ex:norm_linf}
Assume that $h^{(\ell)}$ is given by \eqref{hellinfty}
where the values $(\nu^{(\ell,m)} = |{\sf y}^{(\ell,m)}|/\tau^{(\ell,m)})_{1\le m \le M^{(\ell)}}$ 
are
in ascending order,
and set $\nu^{(\ell,0)} = -\infty$ and 
$\nu^{(\ell,M^{(\ell)}+1)} = \pinf$.
Then, for every $\zeta^{(\ell)} \in \RR$, $({\sf p}^{(\ell)},\theta^{(\ell)}) =
P_{\epi h^{(\ell)}}({\sf y}^{(\ell)},\zeta^{(\ell)})$ is such that
${\sf p}^{(\ell)} = ({\sf p}^{(\ell,m)})_{1 \le m \le M^{(\ell)}}$,
with 
\begin{equation}\label{e:projscalinffuncbis}
{\sf p}^{(\ell,m)} = \begin{cases}
{\sf y}^{(\ell,m)}, & \mbox{if\; $|{\sf y}^{(\ell,m)}|\le \tau^{(\ell,m)}\theta^{(\ell)}$},\\
\tau^{(\ell,m)}\theta^{(\ell)}, & \mbox{if\; ${\sf y}^{(\ell,m)} > \tau^{(\ell,m)}\theta^{(\ell)}$},\\
-\tau^{(\ell,m)}\theta^{(\ell)}, & \mbox{if\; ${\sf y}^{(\ell,m)} < -\tau^{(\ell,m)}\theta^{(\ell)}$},
\end{cases}
\end{equation}
\begin{equation}\label{e:projscalinffuncter}
\theta^{(\ell)} = 
\frac{\max\Big(\zeta^{(\ell)}+\sum_{m=\overline{m}^{(\ell)}}^{M^{(\ell)}}\nu^{(\ell,m)}(\tau^{(\ell,m)})^2,0\Big)}{1+\sum_{m=\overline{m}^{(\ell)}}^{M^{(\ell)}}(\tau^{(\ell,m)})^2},
\end{equation}
and $\overline{m}^{(\ell)}$ is the unique integer in {\small $\{1,\ldots,M^{(\ell)}+1\}$} such that

\begin{equation}\label{e:projscalmaxfuncterbis}
\nu^{(\ell,\overline{m}^{(\ell)}-1)}
 < \frac{\zeta^{(\ell)}+\sum_{m=\overline{m}^{(\ell)}}^{M^{(\ell)}} \nu^{(\ell,m)}(\tau^{(\ell,m)})^2}{1+\sum_{m=\overline{m}^{(\ell)}}^{M^{(\ell)}}(\tau^{(\ell,m)})^2} 
  \le \nu^{(\ell,\overline{m}^{(\ell)})}.
\end{equation}
\end{proposition}
\begin{proof}
For every $({\sf y}^{(\ell)},\zeta^{(\ell)}) \in \RR^{M^{(\ell)}}\times \RR$, in order to determine $P_{\epi h^{(\ell)}}({\sf y}^{(\ell)},\zeta^{(\ell)}) $ we have to find
\begin{equation}
\min_{\theta^{(\ell)} \in \RP}\Big((\theta^{(\ell)}-\zeta^{(\ell)})^2  +
\min_{\substack{|{\sf p}^{(\ell,1)}|\le \tau^{(\ell,1)}\,\theta^{(\ell)}\\\dots\\
|{\sf p}^{(\ell,M^{(\ell)})}|\le \tau^{(\ell,M^{(\ell)})}\,\theta^{(\ell)}}}
\|{\sf p}^{(\ell)}-{\sf y}^{(\ell)}\|^2\Big).
\end{equation}
For every $\theta^{(\ell)}\in \RP$, the inner minimization is achieved when, for every $j\in \{1,\ldots,M^{(\ell)}\}$, ${\sf p}^{(\ell,m)}$ is the projection of ${\sf y}^{(\ell,m)}$ onto $[-\tau^{(\ell,m)}\,\theta^{(\ell)},\tau^{(\ell,m)}\,\theta^{(\ell)}]$, which is given by \eqref{e:projscalinffuncbis}. Then, the problem reduces to
\begin{equation}
\underset{\theta^{(\ell)} \in \RP}{\operatorname{minimize}}\;\Big((\theta^{(\ell)}-\zeta^{(\ell)})^2 + \sum_{m=1}^{M^{(\ell)}} (\max\{|{\sf y}^{(\ell,m)}|-\tau^{(\ell,m)}\theta^{(\ell)},0\})^2\Big)
\end{equation}
which is also equivalent to calculate $\theta^{(\ell)}=\prox_{\phi^{(\ell)}+\iota_{\RP}}(\zeta^{(\ell)})$, where $\phi^{(\ell)}$ is such that, for every $v \in \RR$,
\begin{equation}\label{e:varphiinf}
\phi^{(\ell)}(v) = \frac12\sum_{m=1}^{M^{(\ell)}} 
(\max\{\tau^{(\ell,m)}(\nu^{(\ell,m)}-v),0\})^2.
\end{equation}
By using \cite[Proposition 12]{Combettes_PL_2007_istsp_Douglas_rsatncvsr}, we have $\prox_{\phi^{(\ell)}+\iota_{\RP}} = P_{\RP}\circ \prox_{\phi^{(\ell)}}$. The function $\phi^{(\ell)}$ belongs to $\Gamma_0(\RR)$ since for every $m \in \{1,\ldots,M^{(\ell)}\}$, $v \mapsto \max\{\tau^{(\ell,m)}(\nu^{(\ell,m)} - v,0\}$ is finite convex and $(\cdot)^2$ is finite convex and increasing on $\RP$. In addition, $\phi^{(\ell)}$ is differentiable and it is such that, for every $v\in \RR$ and every $k \in \{1,\ldots,M^{(\ell)}+1\}$,
\begin{equation}
\nu^{(\ell,k-1)} < v \le \nu^{(\ell,k)} \quad \Rightarrow \quad
 \phi^{(\ell)}(v) =
\frac12 \sum_{m=k}^{M^{(\ell)}} (\tau^{(\ell,m)})^2 \Big(v-\nu^{(\ell,m)} \Big)^2.
\end{equation}
For every $\zeta^{(\ell)}\in \RR$, as $\chi^{(\ell)} = \prox_{\phi^{(\ell)}}(\zeta^{(\ell)})$ is characterized by \eqref{e:subdiprox}, there exists $\overline{m}^{(\ell)} \in \{1,\ldots,M^{(\ell)}+1\}$ such that $\nu^{(\ell,\overline{m}^{(\ell)}-1)} < \chi^{(\ell)} \le
\nu^{(\ell,\overline{m}^{(\ell)})}$ and 
\begin{equation}
\zeta^{(\ell)}-\chi^{(\ell)}  = \sum_{m=\overline{m}^{(\ell)}}^{M^{(\ell)}} (\tau^{(\ell,m)})^2 (\chi^{(\ell)} -\nu^{(\ell,m)}).
\end{equation}
This yields $\theta^{(\ell)} = P_{\RP}(\chi^{(\ell)} )$, hence \eqref{e:projscalinffuncter}, and we have: \eqref{e:projscalmaxfuncterbis} $\Leftrightarrow$ $\nu^{(\ell,\overline{m}^{(\ell)}-1)} < \chi^{(\ell)}  \le
\nu^{(\ell,\overline{m}^{(\ell)})}$. The uniqueness of $\overline{m}^{(\ell)} \in \{1,\ldots,M^{(\ell)}+1\}$ satisfying this inequality follows from the uniqueness of $\prox_{\phi^{(\ell)}}(\zeta^{(\ell)})$. 
\end{proof}

When $\tau^{(\ell,m)} \equiv 1$, the function $h^{(\ell)}$ in \eqref{hellinfty} reduces to the standard infinity norm $\|\cdot\|_\infty$ for which the expression of the epigraphical projection has been recently given in \cite{Ding2012_j_math_matrix_cone}.
Note that this proposition can be employed to efficient deal with $\boldsymbol{\ell}_{1,\infty}$ regularization which has attracted
much interest recently \cite{Turlach_B_2005_j-technometrics_simultaneous_vs,Quattoni_A_2009_p-icml_efficient_plr,Chen_Y_2012_j-ieee-tsp_recursive_ll}.
\end{itemize}

\section{Experimental Results}\label{s:algo}\label{sec:exp}
In this section, we provide numerical examples to illustrate the usefulness of the proposed 
epigraphical projection method.
The first presented experiment focuses on applications in image restoration involving projections onto $\boldsymbol{\ell}_{1,p}$-balls where $p\in \{2,\pinf\}$. 
The second experiment deals with a pulse shape design problem based on Proposition~\ref{ex:epidistl}.
 
\subsection{Image restoration}
\subsubsection{Degradation model}
Set $\HH = \RR^{\overline{N}}$. Denote by $\overline{x} = \parens{\overline{x}^{(n)}}_{1 \le n \le \overline{N}} \in \RR^{\overline{N}}$ the signal of interest, and by 
$z \in \RR^N$ an observation vector such that $z = DA\overline{x} + b$. It is assumed that $A \in \RR^{\overline{N} \times \overline{N}}$ is a linear operator, 
$D \in \RR^{N \times \overline{N}}$ is a decimation operator,$^*$\footnote[0]{$^*$$D$ thus corresponds to $N \leq \overline{N}$ lines of the identity
$\overline{N}\times \overline{N}$ matrix.} and $b \in \RR^N$ is a realization of a zero-mean white Gaussian noise vector. The recovery of $\overline{x}$ from the degraded observations is performed by following a variational approach which aims at solving the following problem
\begin{equation}\label{eq:problem_old}
\minimize{x \in [\underline{\mu},\overline{\mu}]^{\overline{N}}} \norm{DAx - z}^2
\quad\subto\quad
\sum_{\ell =1}^L \norm{\Omega_\ell \, B_\ell \, F \, x}_p \le \eta,
\end{equation}
where $(\underline{\mu},\overline{\mu})\in \RR^2$ with $\underline{\mu} \le \overline{\mu}$, $\eta$ is a real positive constant, and $F \in \RR^{K \times \overline{N}}$ 
is the linear operator associated with an analysis transform. Furthermore, for every $\ell \in \{1,\ldots,L\}$, $B_\ell \in \RR^{M^{(\ell)}\times K}$ 
is a \emph{block-selection linear operator} which selects a block of $M^{(\ell)}$ data from its input vector.$^*$ \footnote[0]{$^*$This means that there exist distinct indices $m_1,\ldots,m_{M^{(\ell)}}$ in $\{1,\ldots,K\}$  such that, for every $y = (y^{(k)})_{1\leq k \leq K}\in \RR^K$, $B_{\ell} y = (y^{(m_j)})_{1 \le j \le M^{(\ell)}}$.} 
For every $\ell \in \{1,\ldots,L\}$, $\Omega_\ell$ denotes an $M^{(\ell)}\times M^{(\ell)}$ diagonal matrix of real positive weights.

The term $\norm{DA x - z}^2$ is the \emph{data fidelity} corresponding to the negative log-likelihood of $x$. The bounds $\underline{\mu}$ and $\overline{\mu}$ allow us to take into account the \emph{value range} of each component of $\overline{x}$. The second constraint involved in Problem \eqref{eq:problem_old} promotes solutions having a sparse analysis representation. 
Indeed, it reduces to the weighted $\bell_1$-norm criterion found in \cite{Candes_2008_j-four-anal-appl_enhancing_srl} when each block reduces to a singleton (i.e. $L = K$, and, for every $\ell \in \{1,\ldots,L\}$, $M^{(\ell)} = 1$ and $B_\ell\,y = y^{(\ell)}$). It captures the $\bell_{1,2}$ criteria present in \cite{Rudin_L_1992_tv_atvmaopiip, Gilboa_G_2009_j-siam-mms_nonlocal_oai, Peyre_G_2011_p-eusipco_gro_sop, Bayram2012} when $p=2$. It matches the $\bell_{1,\infty}$ criterion proposed in \cite{Quattoni_A_2009_p-icml_efficient_plr} when $p=\pinf$. 

Note that overlapping blocks in Constraint \eqref{eq:problem_old} are dealt with by increasing the dimensionality of the problem (through the linear transform $F$) and then using an usual non-overlapping block selection operator (denoted $B_\ell$). Let us define $\Lambda = \bracks{B_1^\top\Omega_1,\;\ldots, B_L^\top \Omega_L}^\top$ and
 \begin{equation}\label{eq:proj_E}
C_1 = \menge{ y \in \RR^{M}}{\sum_{\ell = 1}^L \norm{{\rm y}^{(\ell)}}_p \le \eta }
\end{equation}
where $M = M^{(1)} + \cdots + M^{(L)}$ and the same decomposition as in \eqref{e:decblocy} is performed. 
Then, it can be observed that Problem~\eqref{eq:problem_old} is a particular case of Problem~\ref{p:gen} where 
$S=2$, $R=1$, $g_1 = \norm{D\cdot - z}^2$, $T_1 = A$, $H_1=\Lambda F$, $C_1$ is the above $\bell_{1,p}$-ball, $H_2 = I$ and $C_2 = [\underline{\mu},\overline{\mu}]^{\overline{N}}$.

\subsubsection{Algorithmic solution}
As already mentioned in Section~\ref{ssec:algprox}, various proximal algorithms can be used to solve
non-smooth convex optimization problems and would potentially benefit from the proposed epigraphical splitting technique. In this work, we will 
consider the Monotone+Lipschitz Forward Backward Forward (M+LFBF) algorithm, which is a primal-dual method recently proposed in \cite{Combettes_P_2011_j-svva_pri_dsa}, 
and the Simultaneous-Direction Method of Multipliers algorithm (SDMM) \cite{Eckstein1994_j_optim_parallel_ADMM, Setzer_S_2009_j-jvcir_deblurring_pibsbt, Combettes_P_2010_inbook_proximal_smsp}, which is a parallelized version of ADMM. Their convergence is guaranteed (under weak conditions) and their structure makes them suitable for implementation on parallel architectures. 

Although both algorithms address a wide class of convex optimization problems, SDMM requires to invert the matrix $Q = A^\top A + I + F^\top\Lambda \Lambda F$. This is a well-known limitation of ADMM-like algorithms, which can be circumvented in the case when the matrix $Q$ is diagonalizable in the DFT domain (e.g.\ see \cite{Afonso_M_2010_j-tip_fast_iruvsco,Afonso_M_2009_j-tip_augmented_lacofiip}). In the present case, the matrix $Q$ is not diagonalizable since $\Lambda \neq I$. Consequently, SDMM requires to introduce some auxiliary variables into the minimization problem (see \cite{Peyre_G_2011_p-eusipco_gro_sop, Briceno_L_2011_j-math-imaging-vis_pro_ami}) in order to solve the original problem \eqref{eq:problem_old}. As we will see in Tables \ref{tab:nltv3_2}-\ref{tab:nltv3_inf}, this makes ADMM-like approaches much slower than primal-dual approaches.

The main difficulty in solving Problem~\eqref{eq:problem_old} stems from the constraint in \eqref{eq:proj_E}. The point is that proximal algorithms (as well as most of the applicable algorithms) require to compute the projection onto $C_1$. Specific numerical methods \cite{VanDenBerg_E_2008_j-siam-sci-comp_pro_pfb,Weiss_P_2009_Efficient_stvm, Quattoni_A_2009_p-icml_efficient_plr,Fadili_J_2011_tip_tv_proj_fos} have been developed for this purpose.
The aim of this section is to propose an alternative method based on the splitting principle presented in Section~\ref{sec:spl}. So doing, the resulting problem can be efficiently addressed by proximal algorithms. The two possible approaches are now detailed. 

\begin{itemize}
\item \textsl{Epigraphical method} -- The principle of this method is to decompose $C_1$ into the closed half-space defined by \eqref{e:defV} and the closed convex set defined by \eqref{e:defE} with $h^{(\ell)} = \norm{\cdot}_p$. We have thus to
\begin{equation}\label{eq:problem_new}
\minimize{\parens{x,\zeta_1} \in C_2 \times V_1} g_1(Ax)
\quad\subto\quad 
\parens{\Lambda F \, x, \, \zeta_1} \in E_1.
\end{equation}
The advantage of this decomposition is that the projections onto $V_1$ and $E_1$ have closed forms. Indeed, $P_{V_1}$ is the projection onto a half-space, while $P_{E_1}$ is given by Proposition~\ref{ex:norm_l2} for $p=2$ and Proposition~\ref{ex:norm_linf} for $p=\pinf$. We are then able to solve Problem \eqref{eq:problem_new} by means of algorithms such as M+LFBF or SDMM. The associated iterations are given in Algorithm~\ref{algo:epi}. Note that replacing the inequality in the definition of $V_1$ by an equality constraint ${\sf 1}_{L_1}^\top\zeta_1 = \eta_1$ was observed to make little difference in the numerical behaviour of SDMM.

\newcommand{\myvec}[2]{\parens[\Big]{#1\;,\;#2}}
\begin{algorithm}
\caption{M+LFBF  for solving Problem~\eqref{eq:problem_new}}\label{algo:epi}
{\footnotesize
\[
\begin{array}{l}
\mathrm{Initialization}\\
\left\lfloor
\begin{array}{l}
\parens{v^{[0]}, \nu^{[0]}} \in \RR^{M} \times \RR^L\\
\parens{x^{[0]}, \zeta^{[0]}} \in \RR^{\overline{N}} \times  \RR^L\\
\theta = 2\Vert A \Vert^2 + \max\{\norm{{\Lambda} F},1\}\\
\epsilon \in ]0, \frac{1}{\theta+1}[\\
\end{array}
\right.\\[5mm]
\mathrm{For}\; i = 0, 1, \dots\\
\left\lfloor
\begin{array}{l}
\displaystyle\gamma_i \in \bracks*{\epsilon, \frac{1-\epsilon}{\theta}}\\\\
\parens[\Big]{\widehat{x}^{[i]},\widehat{\zeta}^{[i]}} = \parens[\Big]{x^{[i]},\zeta^{[i]}} - \gamma_i \parens[\Big]{A^\top \nabla g_1(A x^{[i]}) + F^\top\Lambda^\top v^{[i]}, \nu^{[i]}}\\
\parens[\Big]{p^{[i]},\rho^{[i]}} = P_{C_2 \times V_1}\parens[\Big]{\widehat{x}^{[i]},\widehat{\zeta}^{[i]}}\\
\parens[\Big]{\widehat{v}^{[i]},\widehat{\nu}^{[i]}} = \parens[\Big]{v^{[i]},\nu^{[i]}} + \gamma_i \parens[\Big]{{\Lambda} F \, x^{[i]},\zeta^{[i]}}\\
\parens[\Big]{a^{[i]},\alpha^{[i]}} = \parens[\Big]{\widehat{v}^{[i]},\widehat{\nu}^{[i]}} - \gamma_i P_{E_1}\parens[\Big]{\widehat{v}^{[i]}/\gamma_i,\widehat{\nu}^{[i]}/\gamma_i}\\
\parens[\Big]{v^{[i+1]},\nu^{[i+1]}} = \parens[\Big]{a^{[i]},\alpha^{[i]}} + \gamma_i \parens[\Big]{\Lambda F (p^{[i]}-x^{[i]}),  \rho^{[i]}-\zeta^{[i]}}\\
\parens[\Big]{\widetilde{x}^{[i]},\widetilde{\zeta}^{[i]}} = \parens[\Big]{p^{[i]},\rho^{[i]}} - \gamma_i \parens[\Big]{A^\top \nabla g_1(A p^{[i]}) + F^\top\Lambda^\top a^{[i]},\alpha^{[i]}}\\
\parens[\Big]{x^{[i+1]},\zeta^{[i+1]}} = \parens[\Big]{x^{[i]}- \widehat{x}^{[i]}+ \widetilde{x}^{[i]},\zeta^{[i]} - \widehat{\zeta}^{[i]} + \widetilde{\zeta}^{[i]}}\\
\end{array}
\right.
\end{array}
\]
}
\end{algorithm}

\item \textsl{Direct method} -- For completeness, we also consider the projection onto $C_1$ with the algorithm in \cite{VanDenBerg_E_2008_j-siam-sci-comp_pro_pfb} when \mbox{$p=2$},$^*$ \footnote[0]{$^*$Code available at \texttt{\scriptsize{www.cs.ubc.ca/$\sim$mpf/spgl1}}} or the iterative algorithm in \cite{Quattoni_A_2009_p-icml_efficient_plr} when $p=\pinf$.\footnote{Code available at \texttt{\scriptsize{www.lsi.upc.edu/$\sim$aquattoni}}} In this case, proximal methods (such as M+LFBF and SDMM) can be used to solve directly Problem \eqref{eq:problem_old}. 
\end{itemize}
According to the general results in \cite[Theorem~4.2]{Combettes_P_2011_j-svva_pri_dsa} and \cite{Combettes_P_2010_inbook_proximal_smsp}, the sequence $\big(x^{[i]}\big)_{i\in \NN}$ generated by M+LFBF or SDMM is guaranteed to converge to a (global) minimizer of Problem~\eqref{eq:problem_old}.

\subsubsection{Smoothness constraint}
In the context of image restoration, the quality of the results obtained through a variational approach strongly depends on the ability to model the regularity present in images. Since natural images are often piecewise smooth, popular regularization models tend to penalize the image gradient. In this regard, \textit{total variation} (TV) \cite{Rudin_L_1992_tv_atvmaopiip} has emerged as a simple, yet successful, convex optimization tool. However, TV fails to preserve textures, details and fine structures, because they are hardly distinguishable from noise. To improve this behaviour, the TV model has been extended by using a non-locality principle \cite{Buades_A_2005_j-siam-mms_review_idawno,Gilboa_G_2009_j-siam-mms_nonlocal_oai}. Another approach to overcome these limitations is to replace the gradient operator with a frame representation which yields a more suitable sparse representation of the image \cite{Mallat_S_1999_ap_wavelet_awtosp}. The connections between these two different approaches have been studied in \cite{Cai_JF_2012_restoration_TV_frame}. It is still unclear which approach leads to the best results. However, there are some evidences that \textit{non-local} (NL) TV may perform better in some image restoration tasks \cite{Zhang_X_2010_j-siam-is_bregmanized_nrd,Peyre_G_2011_NL_reg_inv_prob,Peyre_G_2011_review_adapt_image}. We thus focus our attention on NLTV-based constraints, although our proposed algorithm is quite 
general and it can also be adapted to frame-based approaches.

By appropriately selecting the operators $F$, $B_\ell$ and $\Omega_\ell$ in Problem~\eqref{eq:problem_old}, we can integrate the NLTV measures in a constrained convex optimization approach. In our experiments, we propose to evaluate the performances of two NLTV constraints that constitute particular cases of the one considered in \eqref{eq:problem_old} when $L = \overline{N}$. They are described for 2D data in the following.
\begin{itemize}

\item \textsl{$\bell_2$-NLTV} -- This constraint has the form
\begin{equation}
\sum_{\ell =1}^{\overline{N}} \parens*{\sum_{n \in \mathcal{N}_\ell \subset \mathcal{W}_\ell}\omega_{\ell,n}\parens{x^{(\ell)} - x^{(n)}}^2}^{1/2} \le \eta
,
\end{equation}
where $\mathcal{N}_\ell$ is the \textit{neighbourhood support} at position $\ell$ and $\mathcal{W}_\ell$ is the set of positions $n \in \braces{1,\dots,\overline{N}}\setminus\{\ell\}$ located into a $Q \times Q$ window centered at $\ell$, where $Q \in \NN$ is odd. This constraint is a particular case of the one considered in \eqref{eq:problem_old} where
$K = (Q^2-1) \overline{N}$ and $F$ is a concatenation of discrete difference operators $F_{q_1,q_2}$ with $(q_1,q_2) \in \{-(Q-1)/2,\ldots,(Q-1)/2\}^2\setminus \{(0,0)\}$. More precisely, for every $(q_1,q_2)$, $F_{q_1,q_2}$ is a 2D filter with 
impulse response: for every $(n_1,n_2)\in \ZZ^2$,
\begin{equation}
f_{q_1,q_2}^{(n_1,n_2)} = \begin{cases}
1, & \mbox{if $n_1 = n_2 = 0$},\\
-1, & \mbox{if $n_1 = q_1$ and $n_2 = q_2$},\\
0, & \mbox{otherwise.}
\end{cases}
\end{equation}
In addition, for every $\ell \in \{1,\ldots,\overline{N}\}$, \mbox{$M^{(\ell)} \le Q^2-1$}, $B_\ell$ selects the components of $Fx$ corresponding to differences $(x^{(\ell)}-x^{(n)})_{ n \in \mathcal{N}_\ell}$, and the positive weights $(\omega_{\ell,n})_{n \in \mathcal{N}_\ell}$ are gathered in the diagonal matrix $\Omega_\ell$. 

\item \textsl{$\bell_\infty$-NLTV}-- We consider the following constraint
\begin{equation}
\sum_{\ell =1}^{\overline{N}} \max_{n \in \mathcal{N}_\ell}\big\{\omega_{\ell,n} \, \abs{x^{(\ell)} - x^{(n)}}\big\} \le \eta.
\end{equation}
We proceed similarly to the previous constraint, except that the $\bell_\infty$-norm is now substituted for the $\bell_2$-norm.
\end{itemize}
Note that the classical isotropic TV constraint (designated by $\bell_2$-TV in the following) constitutes a particular case of the $\bell_2$-NLTV one, where 
each neighbourhood $\mathcal{N}_\ell$ only contains the horizontal/vertical neighbouring pixels ($M^{(\ell)} \equiv 2$) and the weights are $\omega_{\ell,n} \equiv 1$. Similarly, the $\bell_\infty$-TV constraint is a special case of the $\bell_\infty$-NLTV one.

\subsubsection{Weight estimation and neighbourhood choice}
To set the weights, we got inspired from the Non-Local Means approach originally described in \cite{Buades_A_2005_j-siam-mms_review_idawno}. Here, for every $\ell \in \{1,\ldots,\overline{N}\}$ and $n\in \mathcal{N}_\ell$, the weight $\omega_{\ell, n}$ depends on the similarity between patches built around the pixels $\ell$ and $n$ of the image. Since our degradation process involves some missing data, a two-step approach has been adopted. In the first step, the $\bell_2$-TV approach is used in order to obtain an estimate $\widetilde{x}$ of the target image. This estimate is subsequently used in the second step to compute the weights through a \textit{self-similarity} measure, yielding
\begin{equation}\label{eq:weight}
\omega_{\ell,n} = \widetilde{\omega}_\ell\exp\left( - \delta^{-2} \; \norm{\widetilde{B}_\ell \widetilde{F}_\ell \widetilde{x} - \widetilde{B}_n \widetilde{F}_n \widetilde{x}}^2\right),
\end{equation}
where $\delta \in \RR\setminus \{0\}$, $\widetilde{\omega}_\ell \in \RPPbis$, $\widetilde{B}_\ell$ (resp. $\widetilde{B}_n$) selects a $\widetilde{Q} \times \widetilde{Q} $ patch centered at position $\ell$ (resp. $n$) and $\widetilde{F}_\ell$ (resp. $\widetilde{F}_n$) is a linear processing of the image depending on the position $\ell$ (resp. $n$). The constant $\widetilde{\omega}_\ell$ is set so as to normalize the weights (i.e. $\sum_{n \in \mathcal{N}_\ell} \omega_{\ell,n} = 1$).

The measure in \eqref{eq:weight} generalizes the one proposed in \cite{Buades_A_2005_j-siam-mms_review_idawno}, which corresponds to the case when $\widetilde{F}_\ell$ (resp. $\widetilde{F}_n$) 
reduces to a Gaussian function with mean $\ell$ (resp. $n$). In the present work, we consider the \textit{foveated self-similarity} measure recently introduced in \cite{Foi_A_2012_p-spie_foveated_ssnif}, due to its better performance in denoising. This approach can be derived from \eqref{eq:weight} by setting $\widetilde{F}_\ell$ (resp. $\widetilde{F}_n$) to a set of low-pass Gaussian filters whose variances increase as the spatial distance from the patch center $\ell$ (resp. $n$) grows.
  
For every $\ell \in \{1,\ldots,\overline{N}\}$, the neighbourhood $\mathcal{N}_\ell$ is built according to the procedure described in \cite{Gilboa_G_2007_j-siam-mms_nonlocal_irss}. In practice, we limit the size of the neighbourhood, so that $M^{(\ell)} \le \overline{M}$ (a possible choice of $\overline{M}$ is given in the next paragraph).

\subsubsection{Numerical results -- analysis of convergence times}
In this section, the execution time of the proposed epigraphical technique is evaluated w.r.t.\ the direct method involving standard numerical solutions. 

In the following experiments, if not specified otherwise, the degradation matrix $A$ is a convolution which consists of a $3 \times 3$ uniform blur and the decimation matrix $D$ randomly removes $60\%$ of the pixels ($N = 0.4 \times \overline{N}$). The standard deviation of the additive white Gaussian noise is equal to $\sigma = 10$. Since we deal with natural images, the data range bounds are $\underline{\mu} = 0$ and $\overline{\mu} = 255$. For the smoothness constraint, we set $Q = 11$, $\widetilde{Q} = 5$, $\delta = 35$ and $\overline{M} = 14$.

We present the results obtained with the image \emph{boat} cropped at $256\times256$ ($\overline{N} = 256^2$), since a similar behaviour was observed for other images. The stopping criterion is set to $\norm{x^{[i+1]}-x^{[i]}} \le 10^{-4} \norm{x^{[i]}}$. For the $\bell_{1,p}$-ball projectors needed by the direct method, we used the software publicly available on-line \cite{VanDenBerg_E_2008_j-siam-sci-comp_pro_pfb, Quattoni_A_2009_p-icml_efficient_plr}. 

Note that SDMM requires to invert the matrix $Q = A^\top A + I + F^\top\Lambda \Lambda F$, which we address by resorting to the solution proposed in \cite{Peyre_G_2011_p-eusipco_gro_sop, Briceno_L_2011_j-math-imaging-vis_pro_ami}. In order to make the operators $A$ and $F$ diagonalizable 
in the DFT domain, a periodic extension of the image is performed.

In practice, the constraint bound $\eta$ may not be known precisely.
Although it is out of the scope of this paper to devise an optimal strategy to set this bound, it is important to evaluate the impact of its choice on our method 
performance. 
In the following, we compare the epigraphical approach with the direct computation of the projections (via standard iterative solutions)
for different choices of regularization constraints 
and values of $\eta$.
\begin{itemize}

\item \textsl{Total Variation} -- Tables \ref{tab:tv_eta} and \ref{tab:tv_eta_inf} report a comparison between the direct and epigraphical methods 
for $\bell_{2}$-TV and $\bell_{\infty}$-TV, respectively. For more readability, the values of $\eta$ are expressed as a multiplicative factor of the $\bell_p$-TV-semi-norm 
of the original image. The convergence times 
indicate that the epigraphical approach yields a faster convergence than the direct approach for SDMM and M+LFBF. Moreover, the numerical results show that errors within $\pm 20\%$ 
from the optimal value for $\eta$ lead to SNR variations within $2\%$.

Figs.\ \ref{fig:tv_prof:time} and \ref{fig:tv_prof_inf:time}  show the relative error $\norm{x^{[i]}-x^{[\infty]}}/\norm{x^{[\infty]}}$ as a function of the computational time, where $x^{[\infty]}$ denotes the solution computed after a large number of iterations (typically, 5000 iterations). The dashed line presents the results for the direct method while the solid line refers to the epigraphical one. These plots show that the epigraphical approach is faster despite it requires more iterations in order to converge. This can be explained by the computational cost of the subiterations required by the direct projections onto the $\bell_{1,p}$-ball.

\begin{table*}[p]
  \centering%
  \caption{Results for the $\bell_2$-TV constraint and different values of $\eta$}
  {\scriptsize
  \begin{tabular}{+c@{\quad}^c @{\quad} ^c@{\;}^c ^c@{\;}^c @{}^c @{\;}^c@{}^c @{\qquad} ^c@{\;}^c ^c@{\;}^c @{}^c @{\;} ^c}
    \toprule
    \multirow{4}{*}{$\eta$} & \multirow{4}{*}{SNR (dB) -- SSIM} & \multicolumn{6}{c}{SDMM} && \multicolumn{6}{c}{M+LFBF} \\
    \cmidrule{3-8}\cmidrule(r){10-15}
    & & \multicolumn{2}{c}{direct} & \multicolumn{2}{c}{epigraphical} && \multirow{2}{*}{speed up} && \multicolumn{2}{c}{direct} & \multicolumn{2}{c}{epigraphical} && \multirow{2}{*}{speed up}\\
    \cmidrule(r){3-4}\cmidrule(lr){5-6}\cmidrule(r){10-11}\cmidrule(lr){12-13}
                            &                           & \# iter. & {sec.}    & \# iter. & {sec.}      && && \# iter. & {sec.}    & \# iter. & {sec.}      &&\\
    \midrule
		0.45 & 19.90 -- 0.733 & 107 &  6.07 & 174 &  2.03 && 2.99 && 113 &  6.15 & 182 &  3.49 && 1.76 \\
		0.50 & 20.18 -- 0.745 & 117 &  6.95 & 159 &  1.95 && 3.57 && 116 &  6.97 & 168 &  3.44 && 2.03 \\
\brow	0.56 & 20.23 -- 0.745 & 129 &  8.36 & 153 &  1.90 && 4.41 && 124 &  8.17 & 159 &  3.01 && 2.72 \\
		0.62 & 20.16 -- 0.737 & 141 &  9.44 & 155 &  1.83 && 5.16 && 131 &  8.62 & 159 &  3.26 && 2.65 \\
		0.67 & 20.00 -- 0.724 & 154 & 10.20 & 162 &  2.17 && 4.71 && 140 & 10.00 & 164 &  2.84 && 3.52 \\
    \bottomrule
  \end{tabular}
  }
  \label{tab:tv_eta}
\end{table*}

\begin{table*}[p]
  \centering%
  \caption{Results for the $\bell_\infty$-TV constraint and different values of $\eta$}
  {\scriptsize
  \begin{tabular}{+c@{\quad}^c @{\quad} ^c@{\;}^c ^c@{\;}^c @{}^c @{\;}^c@{}^c @{\qquad} ^c@{\;}^c ^c@{\;}^c @{}^c @{\;} ^c}
    \toprule
    \multirow{4}{*}{$\eta$} & \multirow{4}{*}{SNR (dB) -- SSIM} & \multicolumn{6}{c}{SDMM} && \multicolumn{6}{c}{M+LFBF} \\
    \cmidrule{3-8}\cmidrule(r){10-15}
    & & \multicolumn{2}{c}{direct} & \multicolumn{2}{c}{epigraphical} && \multirow{2}{*}{speed up} && \multicolumn{2}{c}{direct} & \multicolumn{2}{c}{epigraphical} && \multirow{2}{*}{speed up}\\
    \cmidrule(r){3-4}\cmidrule(lr){5-6}\cmidrule(r){10-11}\cmidrule(lr){12-13}
                            &                           & \# iter. & {sec.}    & \# iter. & {sec.}      && && \# iter. & {sec.}    & \# iter. & {sec.}      &&\\
    \midrule
		0.45 & 19.52 -- 0.726 & 160 & 312.55 & 231 &  3.89 &&  80.43 && 183 & 347.10 & 252 &  6.43 && 53.96\\
		0.50 & 19.71 -- 0.734 & 168 & 342.01 & 215 &  3.75 &&  91.31 && 185 & 368.24 & 236 &  5.83 && 63.17\\
\brow	0.56 & 19.71 -- 0.728 & 180 & 373.60 & 211 &  3.49 && 106.93 && 189 & 386.29 & 229 &  5.53 && 69.91\\
		0.62 & 19.59 -- 0.715 & 196 & 412.68 & 216 &  3.67 && 112.50 && 198 & 411.04 & 229 &  5.86 && 70.15\\
		0.67 & 19.39 -- 0.698 & 211 & 448.77 & 223 &  3.76 && 119.27 && 207 & 437.66 & 234 &  5.76 && 75.96\\
    \bottomrule
  \end{tabular}
  }
  \label{tab:tv_eta_inf}
\end{table*}
  
\begin{table*}[p]
	\centering%
	\caption{Results for the $\bell_2$-NLTV constraint and some values of $\eta$ and $Q$} 
	{\scriptsize
	\begin{tabular}{+c@{\quad}^c @{\quad} ^c@{\;}^c ^c@{\;}^c @{}^c @{\;}^c@{}^c @{\qquad} ^c@{\;}^c ^c@{\;}^c @{}^c @{\;} ^c}
		\toprule
		\multirow{4}{*}{$\eta$} & \multirow{4}{*}{SNR (dB) -- SSIM} & \multicolumn{6}{c}{SDMM} && \multicolumn{6}{c}{M+LFBF} \\
		\cmidrule{3-8}\cmidrule(r){10-15}
		& & \multicolumn{2}{c}{direct} & \multicolumn{2}{c}{epigraphical} && \multirow{2}{*}{speed up} && \multicolumn{2}{c}{direct} & \multicolumn{2}{c}{epigraphical} && \multirow{2}{*}{speed up}\\
		\cmidrule(r){3-4}\cmidrule(lr){5-6}\cmidrule(r){10-11}\cmidrule(lr){12-13}
		                   &                           & \# iter. & {sec.}    & \# iter. & {sec.}      && && \# iter. & {sec.}    & \# iter. & {sec.}      &&\\
		\midrule
		\multicolumn{15}{c}{\textsl{Neighbourhood size: $Q = 3$}}\\
		0.43 & 20.82 -- 0.757 & 208 & 20.67 & 211 & 10.93 && 1.89 && 82 &  6.95 & 93 &  3.76 && 1.85\\
		0.49 & 20.97 -- 0.765 & 167 & 16.84 & 177 &  9.01 && 1.87 && 75 &  6.61 & 83 &  3.47 && 1.91\\
		0.54 & 21.02 -- 0.767 & 147 & 15.31 & 157 &  7.93 && 1.93 && 71 &  6.45 & 77 &  3.15 && 2.04\\
		0.59 & 20.98 -- 0.764 & 134 & 14.44 & 148 &  7.67 && 1.88 && 72 &  6.58 & 77 &  3.24 && 2.03\\
		0.65 & 20.88 -- 0.757 & 133 & 14.82 & 136 &  7.11 && 2.08 && 76 &  7.53 & 80 &  3.27 && 2.30\\

		\multicolumn{15}{c}{\textsl{Neighbourhood size: $Q = 5$}}\\
		0.43 & 21.00 -- 0.766 & 301 & 56.03 & 343 & 45.18 && 1.24 && 82 &  8.51 & 90 &  5.43 && 1.57\\
		0.49 & 21.15 -- 0.773 & 260 & 49.03 & 302 & 39.64 && 1.24 && 75 &  7.90 & 81 &  4.90 && 1.61\\
\brow   0.54 & 21.20 -- 0.775 & 242 & 46.31 & 283 & 37.72 && 1.23 && 71 &  8.26 & 75 &  4.47 && 1.85\\
		0.59 & 21.17 -- 0.773 & 231 & 46.20 & 268 & 36.56 && 1.26 && 70 &  7.94 & 74 &  4.49 && 1.77\\
		0.65 & 21.08 -- 0.767 & 220 & 44.64 & 252 & 34.46 && 1.30 && 73 &  8.40 & 76 &  4.59 && 1.83\\
		\bottomrule
	\end{tabular}
	}
	\label{tab:nltv3_2}
\end{table*}
	
\begin{table*}[p]
	\centering%
	\caption{Results for the $\bell_\infty$-NLTV constraint and some values of \scriptsize $\eta$ and $Q$}
	{\scriptsize
	\begin{tabular}{+c@{\quad}^c @{\quad} ^c@{\;}^c ^c@{\;}^c @{}^c @{\;}^c@{}^c @{\qquad} ^c@{\;}^c ^c@{\;}^c @{}^c @{\;} ^c}
		\toprule
		\multirow{4}{*}{$\eta$} & \multirow{4}{*}{SNR (dB) -- SSIM} & \multicolumn{6}{c}{SDMM} && \multicolumn{6}{c}{M+LFBF} \\
		\cmidrule{3-8}\cmidrule(r){10-15}
		& & \multicolumn{2}{c}{direct} & \multicolumn{2}{c}{epigraphical} && \multirow{2}{*}{speed up} && \multicolumn{2}{c}{direct} & \multicolumn{2}{c}{epigraphical} && \multirow{2}{*}{speed up}\\
		\cmidrule(r){3-4}\cmidrule(lr){5-6}\cmidrule(r){10-11}\cmidrule(lr){12-13}
		                   &                           & \# iter. & {sec.}    & \# iter. & {sec.}      && && \# iter. & {sec.}    & \# iter. & {sec.}      &&\\
		\midrule
		\multicolumn{15}{c}{\textsl{Neighbourhood size: $Q = 3$}}\\
		0.43 & 20.78 -- 0.762 & 434 & 1470.46 & 449 & 25.03 && 58.76 && 225 & 730.26 & 244 & 12.35 && 59.15\\
		0.49 & 20.86 -- 0.764 & 395 & 1319.64 & 413 & 22.86 && 57.72 && 221 & 692.25 & 237 & 11.92 && 58.08\\
		0.54 & 20.83 -- 0.760 & 363 & 1193.61 & 382 & 21.46 && 55.62 && 217 & 667.50 & 233 & 11.46 && 58.22\\
		0.59 & 20.73 -- 0.752 & 340 & 1093.26 & 354 & 19.77 && 55.30 && 216 & 653.79 & 230 & 11.67 && 56.01\\
		0.65 & 20.58 -- 0.740 & 322 & 1007.55 & 336 & 18.64 && 54.06 && 216 & 643.00 & 229 & 11.45 && 56.18\\
		
		\multicolumn{15}{c}{\textsl{Neighbourhood size: $Q = 5$}}\\
		0.43 & 20.91 -- 0.769 & 384 & 2069.62 & 452 & 64.42 && 32.13 && 233 & 863.01 & 252 & 18.47 && 46.73\\
		0.49 & 20.98 -- 0.771 & 326 & 1700.34 & 412 & 58.66 && 28.99 && 231 & 822.06 & 247 & 18.36 && 44.77\\
\brow 	0.54 & 20.97 -- 0.767 & 290 & 1476.98 & 389 & 55.35 && 26.69 && 229 & 787.61 & 245 & 17.90 && 43.99\\
		0.59 & 20.88 -- 0.759 & 276 & 1336.16 & 374 & 52.64 && 25.38 && 230 & 772.42 & 245 & 17.57 && 43.96\\
		0.65 & 20.75 -- 0.749 & 268 & 1220.14 & 362 & 51.45 && 23.72 && 231 & 760.86 & 245 & 17.81 && 42.72\\
		\bottomrule
	\end{tabular}
	}
	\label{tab:nltv3_inf}
\end{table*}

\item \textsl{$\bell_2$-NLTV} -- Table \ref{tab:nltv3_2} collects the results of $\bell_2$-NLTV for different values of neighbourhood size $Q$. 
To set the weights, the first TV estimate is computed with $\eta = 0.56$. 
The convergence times show that the epigraphical approach is faster than the direct one for both considered algorithms. 
Moreover, it can be noticed that errors within $\pm 20\%$ from the optimal bound value lead to SNR variations within $1\%$. 
In \figurename~\ref{fig:nltv2_prof:time}, a plot similar to those in Figs.\ \ref{fig:tv_prof:time} and \ref{fig:tv_prof_inf:time} show the convergence profile. The epigraphical method requires about the same number of iterations as the direct one in order to converge. This results in a time reduction, as a single iteration of the epigraphical method is faster than one iteration of the direct method.
  
\item \textsl{$\bell_\infty$-NLTV} -- Table~\ref{tab:nltv3_inf} and \figurename~\ref{fig:nltv_inf_prof:time} show the results obtained with the $\bell_\infty$-NLTV constraint. Similarly to $\bell_\infty$-TV, the epigraphical approach greatly speeds up the convergence times.
\end{itemize}

\begin{figure*}[]%
   	\centering%
   	\subfloat[$\bell_2$-TV.]{\includegraphics[width=0.35\textwidth]{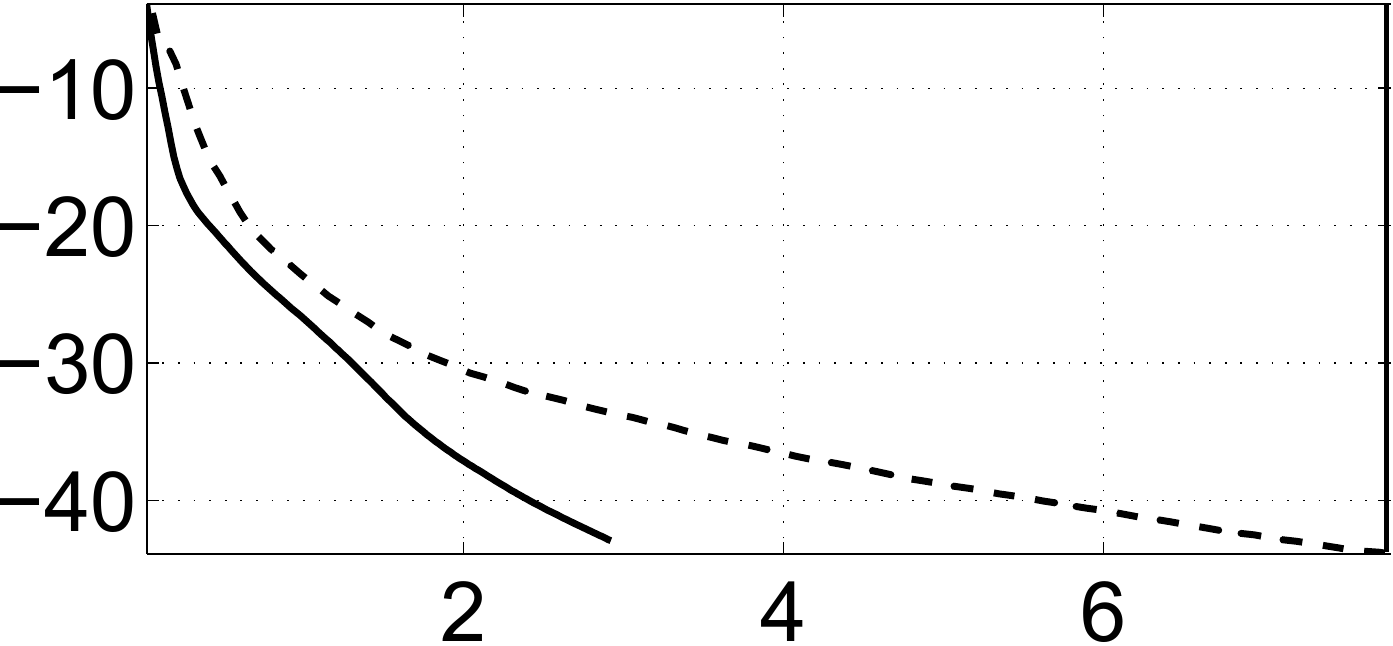}\label{fig:tv_prof:time}}
   	\qquad
   	\subfloat[$\bell_\infty$-TV.]{\includegraphics[width=0.35\textwidth]{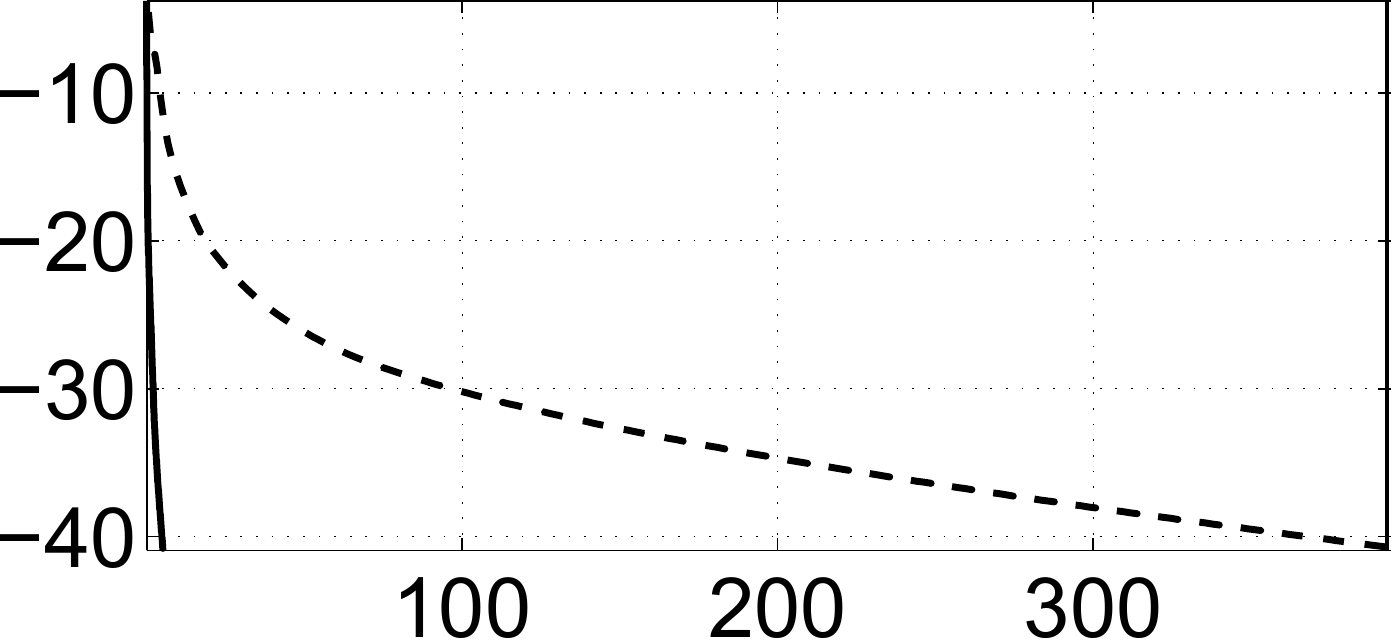}\label{fig:tv_prof_inf:time}}
  	
   	\subfloat[$\bell_2$-NLTV.]{\includegraphics[width=0.35\textwidth]{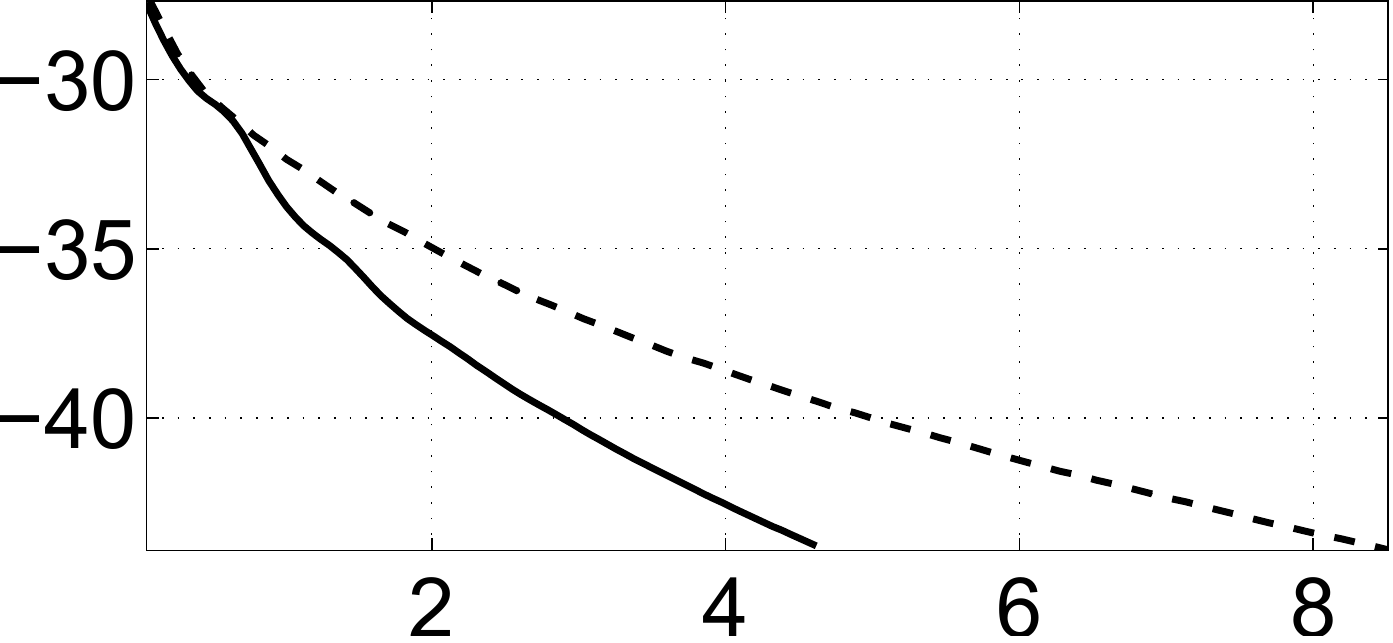}\label{fig:nltv2_prof:time}} 
   	\qquad
   	\subfloat[$\bell_\infty$-NLTV.]{\includegraphics[width=0.35\textwidth]{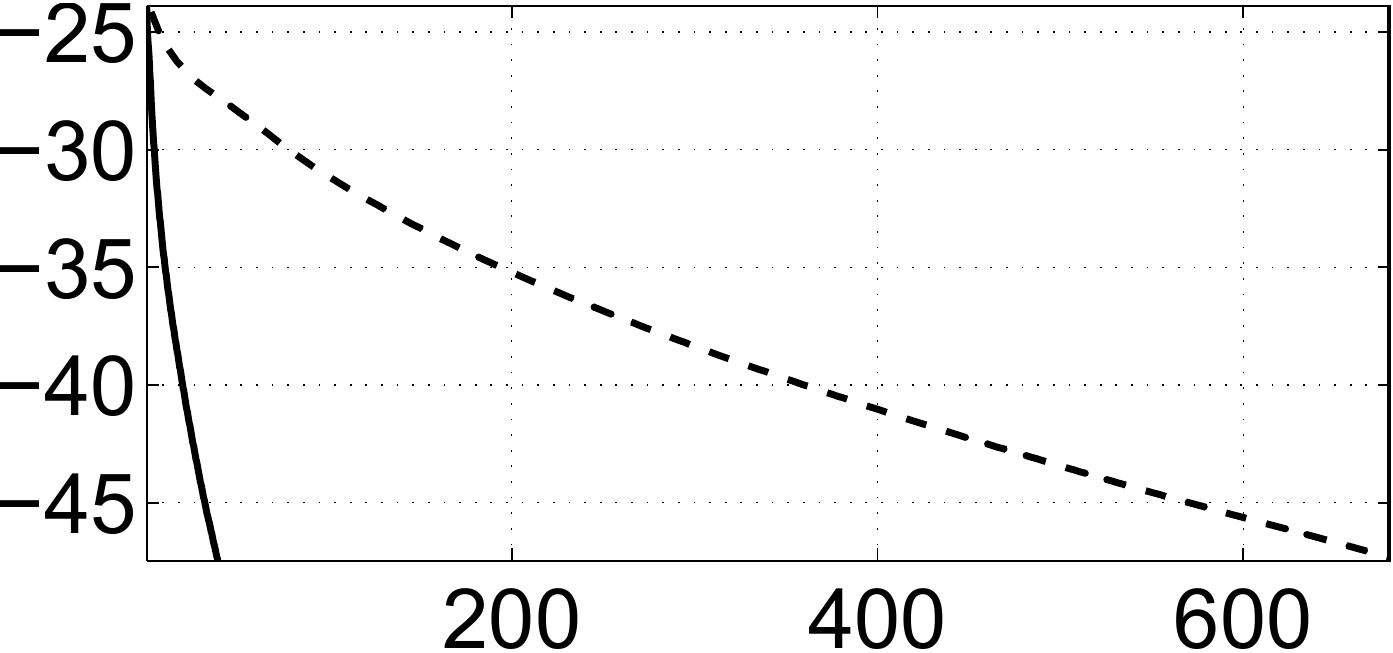}\label{fig:nltv_inf_prof:time}} 
   	 \caption{Comparison between epigraphical method (\textit{solid line}) and direct method (\textit{dashed line}): $\frac{\norm{x^{[i]}-x^{[\infty]}}}{\norm{x^{[\infty]}}}$ in dB vs time}%
   	\label{fig:tv_prof}%
\end{figure*}

\subsubsection{Numerical results -- restoration performance}
In this section, the quality of images reconstructed with our variational approach is evaluated for different choices of regularization constraints and comparisons are made with a state-of-the-art method. 
%
%
Extensive tests have been carried out on several standard images of different sizes. The SNR and SSIM \cite{Wang_Z_2009_spm_MES_lioli} results obtained by using the various previously introduced TV-like constraints are collected in Table~\ref{tab:all}. In addition, a comparison is performed between our method using an M+LFBF implementation
and the Gradient Projection for Sparse Reconstruction (GPSR) method \cite{Figueiredo_M_2007_j-ieee-sel-topics-sp_gra_psr}, which also relies on a variational approach. The constraint bound for both methods was hand-tuned in order to achieve the best SNR values. The best results are highlighted in bold.  A visual comparison is made in \figurename~\ref{fig:tv_images}, where two representative images are displayed. These results demonstrate the interest of considering non-local smoothness measures. Indeed, NLTV with $\bell_{1,2}$-norm proves to be the most effective constraint with gains in SNR and SSIM (up to 1.82 dB and 0.042) with respect to $\bell_2$-TV, which in turn outperforms GPSR. The better performance of NLTV seems to be related to its ability to better preserve edges and thin structures present in images. In terms of computational time, GPSR is about twice faster than $\bell_2$-NLTV. Our codes were developed in MATLAB$^*$ \footnote[0]{$^*$R2011b version on an Intel Xeon CPU at 2.80 GHz and 8 GB of RAM.}, the operators $F$ and $F^\top$ being implemented in C using mex files.

In order to complete the analysis, we report in \figurename~\ref{fig:noise} SNR/SSIM comparisons between $\bell_2$-NLTV and $\bell_2$-TV for different blur and noise configurations. These plots show that $\bell_2$-NLTV provides better results regardless of the degradation conditions.

\begin{table*}
  \centering%
  \caption{$\mathrm{SNR_{dB}}$ and SSIM results of our method and GPSR (noise parameters: blur = $3 \times 3$, $\sigma = 10$, decimation = $60\%$)}
  {\scriptsize
  \begin{tabular}{lc @{\qquad} rrrr c @{\quad} r@{ -- }r}
    \toprule
    SNR (dB) -- SSIM & $\overline{N}$ & \multicolumn{1}{c}{$\bell_2$-TV} & \multicolumn{1}{c}{$\bell_\infty$-TV} & \multicolumn{1}{c}{$\bell_2$-NLTV} & \multicolumn{1}{c}{$\bell_\infty$-NLTV} && \multicolumn{2}{c}{\textsc{gpsr}} \\
    \midrule
\textsc{Culicoidae}	& $256^2$	& 20.80 -- 0.855 & 20.25 -- 0.853 & \textbf{22.62} -- \textbf{0.897} & 22.38 -- 0.897 && 17.03 & 0.738\\
\textsc{Lena}		& $256^2$	& 23.18 -- 0.783 & 22.77 -- 0.769 & \textbf{24.18} -- \textbf{0.812} & 24.14 -- 0.812 && 20.26 & 0.678\\
\textsc{Boat}		& $256^2$	& 20.25 -- 0.739 & 19.74 -- 0.718 & \textbf{21.13} -- \textbf{0.770} & 20.77 -- 0.741 && 18.06 & 0.649\\
\textsc{Cameraman}	& $256^2$	& 20.06 -- 0.774 & 19.68 -- 0.755 & \textbf{20.71} -- \textbf{0.801} & 20.17 -- 0.743 && 17.92 & 0.673\\
\textsc{House}		& $256^2$	& 25.47 -- 0.823 & 24.70 -- 0.808 & \textbf{26.31} -- \textbf{0.836} & 25.87 -- 0.823 && 22.14 & 0.734\\
\textsc{Man}		& $256^2$	& 19.24 -- 0.725 & 18.96 -- 0.714 & \textbf{19.66} -- \textbf{0.741} & 19.51 -- 0.736 && 17.11 & 0.629\\
\textsc{Peppers}    & $512^2$	& 23.69 -- 0.801 & 23.25 -- 0.786 & \textbf{24.80} -- \textbf{0.829} & 24.45 -- 0.813 && 21.94 & 0.709\\
\textsc{Barbara}    & $512^2$	& 16.74 -- 0.653 & 16.64 -- 0.642 & \textbf{17.02} -- \textbf{0.673} & 16.99 -- 0.652 && 15.97 & 0.562\\
\textsc{Hill} 		& $512^2$	& 22.18 -- 0.723 & 21.89 -- 0.715 & \textbf{22.55} -- \textbf{0.735} & 22.43 -- 0.733 && 20.21 & 0.637\\
\textsc{Culicoidae} & $1024^2$	& 20.84 -- 0.855 & 20.49 -- 0.812 & \textbf{23.25} -- \textbf{0.885} & 22.57 -- 0.810 && 17.19 & 0.725\\
    \bottomrule
  \end{tabular}
  }
  \label{tab:all}
\end{table*}

\makeatletter
\define@key{Gin}{crop1}[true]{%
    \edef\@tempa{{Gin}{trim=20mm 20mm 20mm 20mm,clip}}%
    \expandafter\setkeys\@tempa
}
\define@key{Gin}{crop2}[true]{%
    \edef\@tempa{{Gin}{trim=20mm 20mm 20mm 20mm,clip}}%
    \expandafter\setkeys\@tempa
}
\makeatother
\newcommand{\mywidth}{0.2\textwidth}
\begin{figure*}%
	\centering%
	\subfloat[Culicoidae.]{\includegraphics[width=\mywidth]{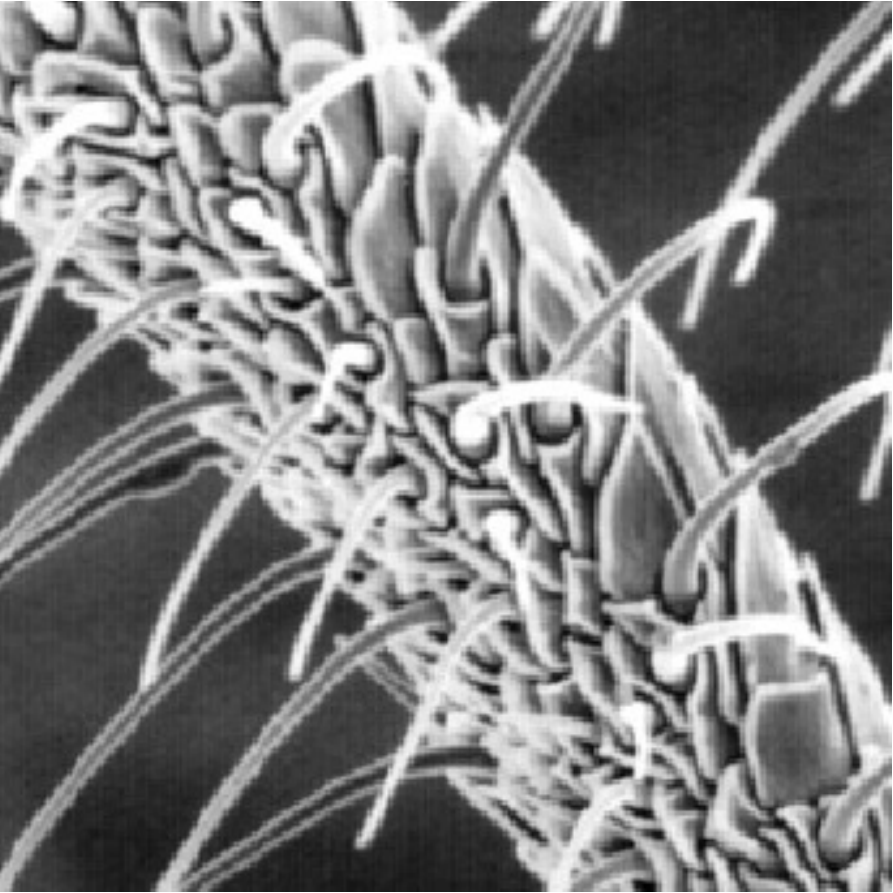}\label{fig:tv_images:orig1}}
	\hfill
	\subfloat[Degraded.]{\includegraphics[width=\mywidth]{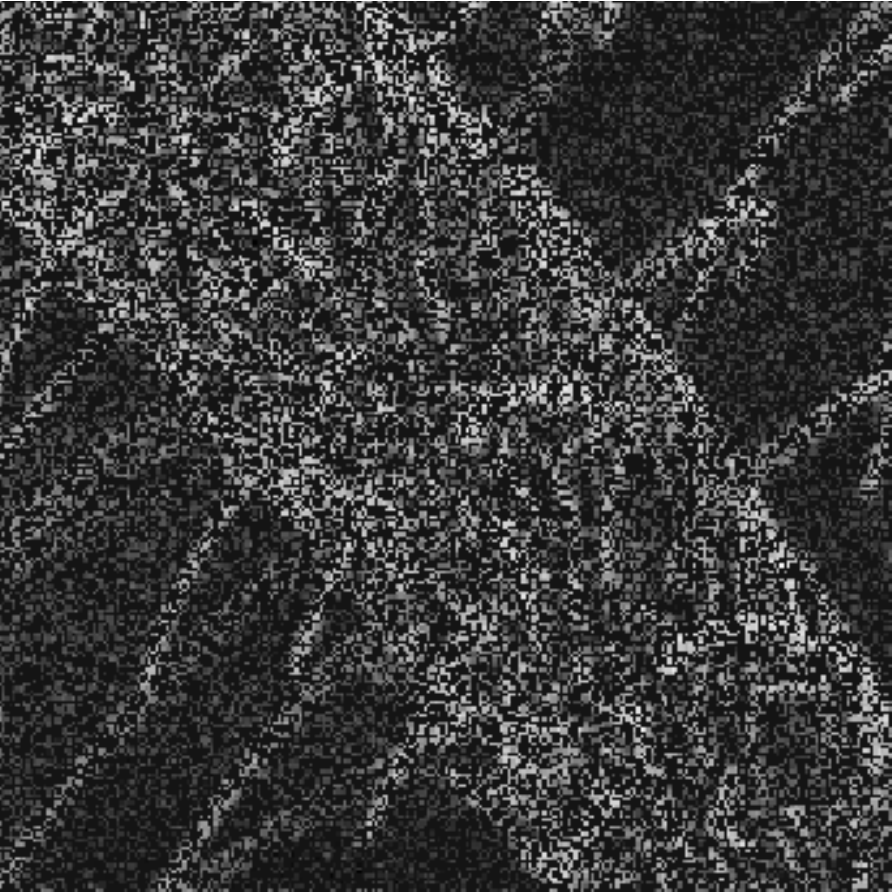}\label{fig:tv_images:degr1}}
	\hfill
	\subfloat[Zoom.]{\includegraphics[crop1,width=\mywidth]{culi_original}}
	\hfill
	\subfloat[GPSR, SNR:~17.03~dB.]{\includegraphics[crop1,width=\mywidth]{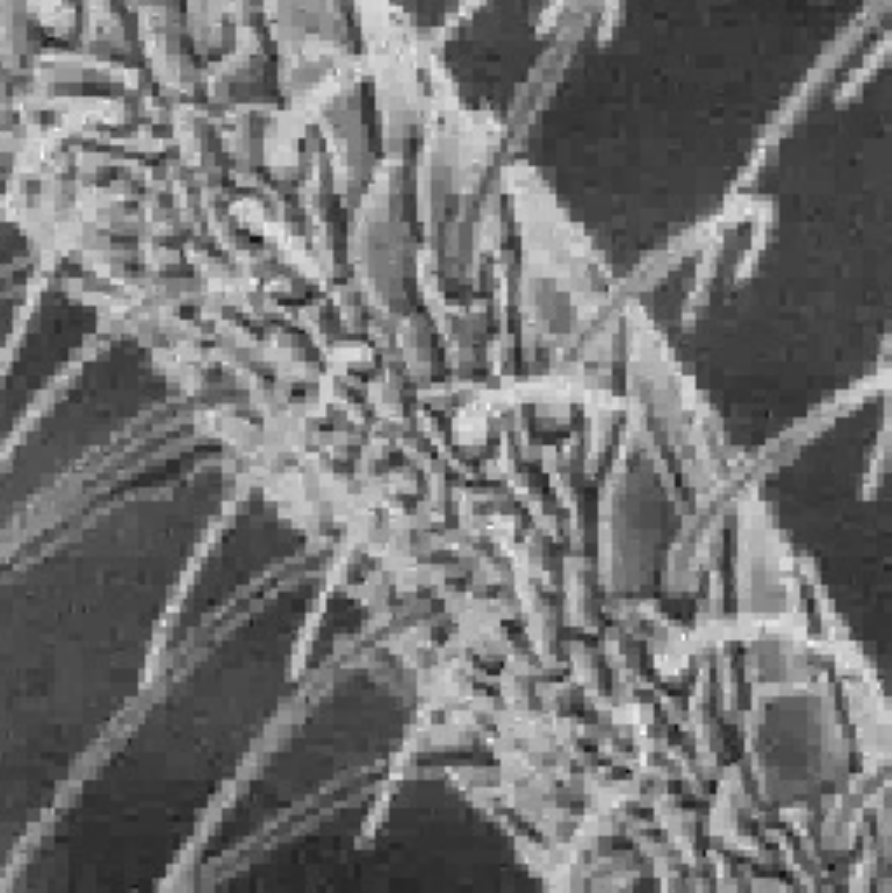}}

	\subfloat[$\bell_2$-TV, SNR:~20.80~dB.]{\includegraphics[crop1,width=\mywidth]{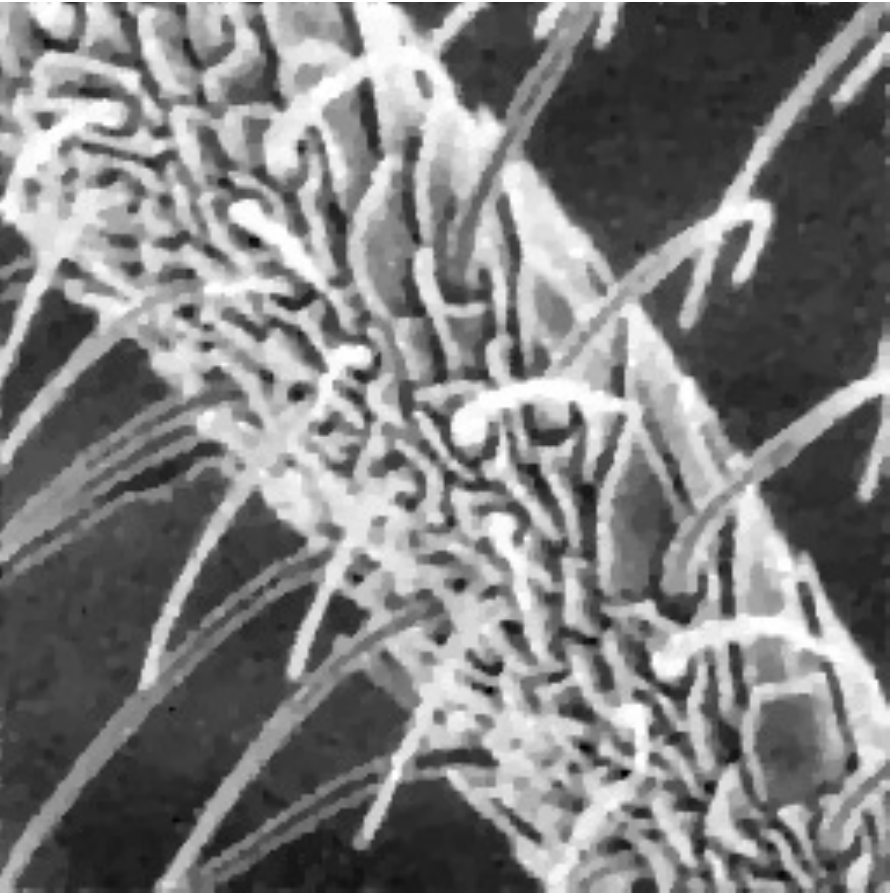}\label{fig:tv_images:rest1a}}
	\hfill
	\subfloat[$\bell_\infty$-TV, SNR:~20.25~dB.]{\includegraphics[crop1,width=\mywidth]{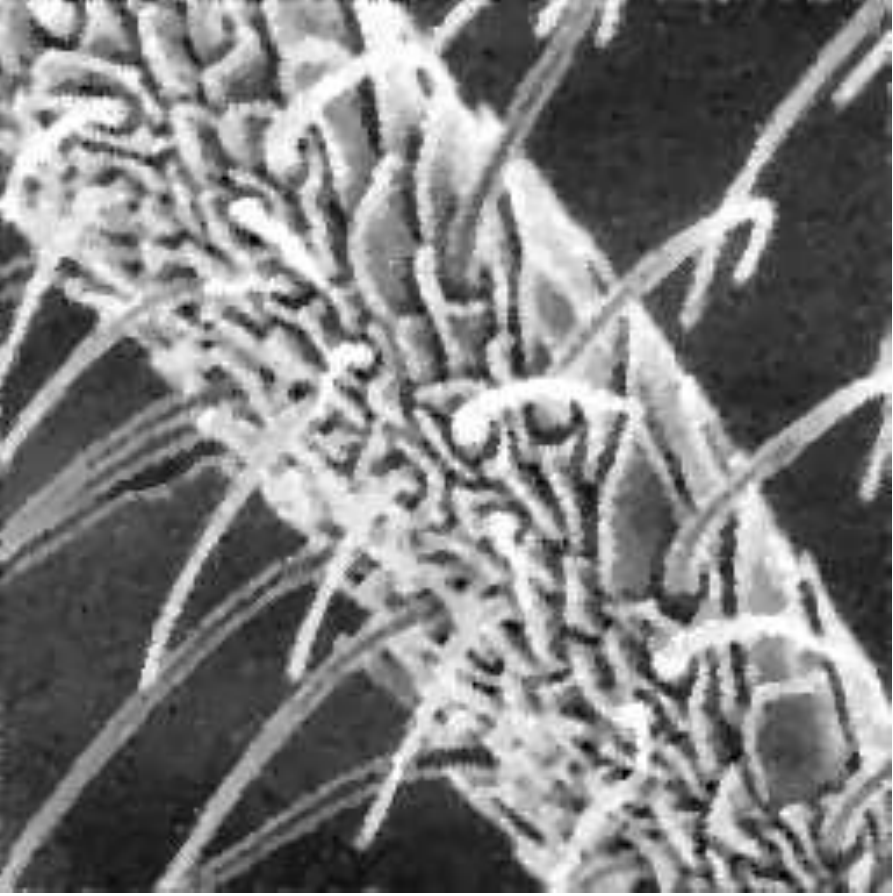}\label{fig:tv_images:rest1b}}
	\hfill	
	\subfloat[$\bell_2$-NLTV, SNR:~\textbf{22.62~dB}.]{\includegraphics[crop1,width=\mywidth]{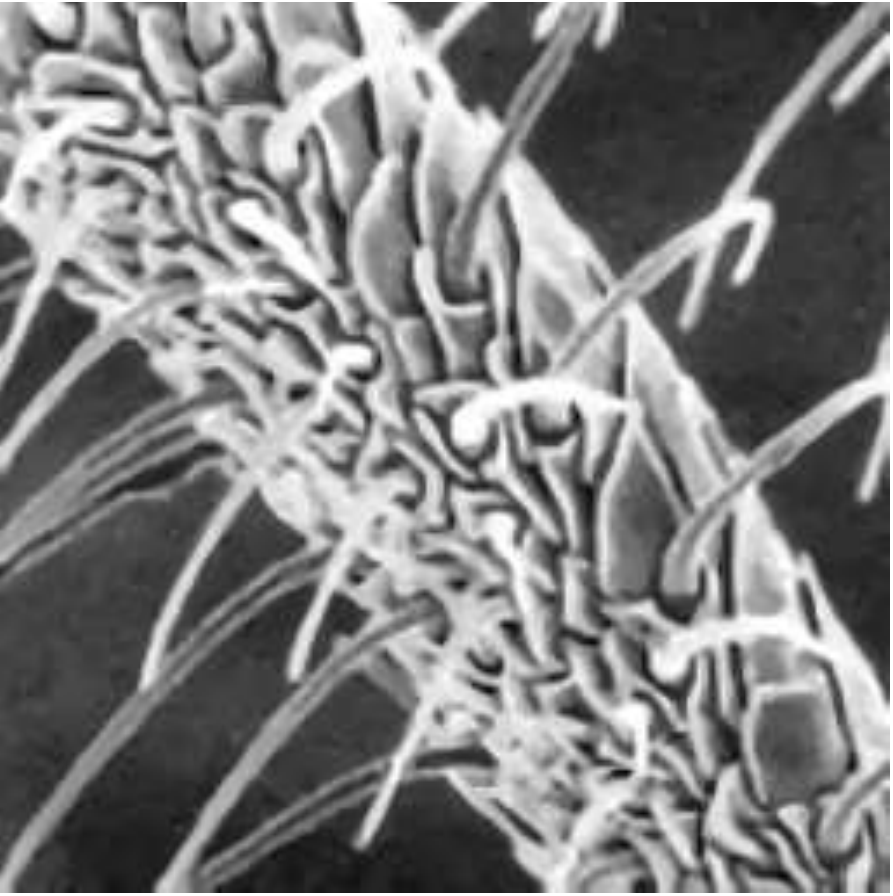}\label{fig:tv_images:rest1c}}
	\hfill
	\subfloat[$\bell_\infty$-NLTV, SNR:~22.38~dB.]{\includegraphics[crop1,width=\mywidth]{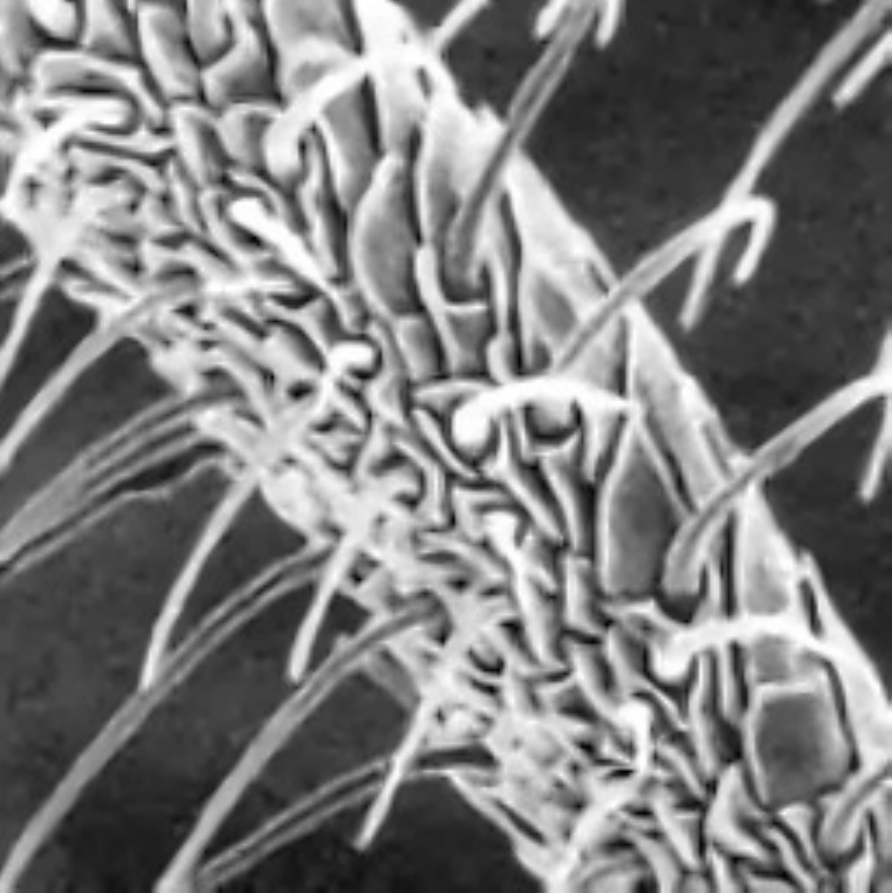}\label{fig:tv_images:rest1d}}

	\subfloat[Lena.]{\includegraphics[width=\mywidth]{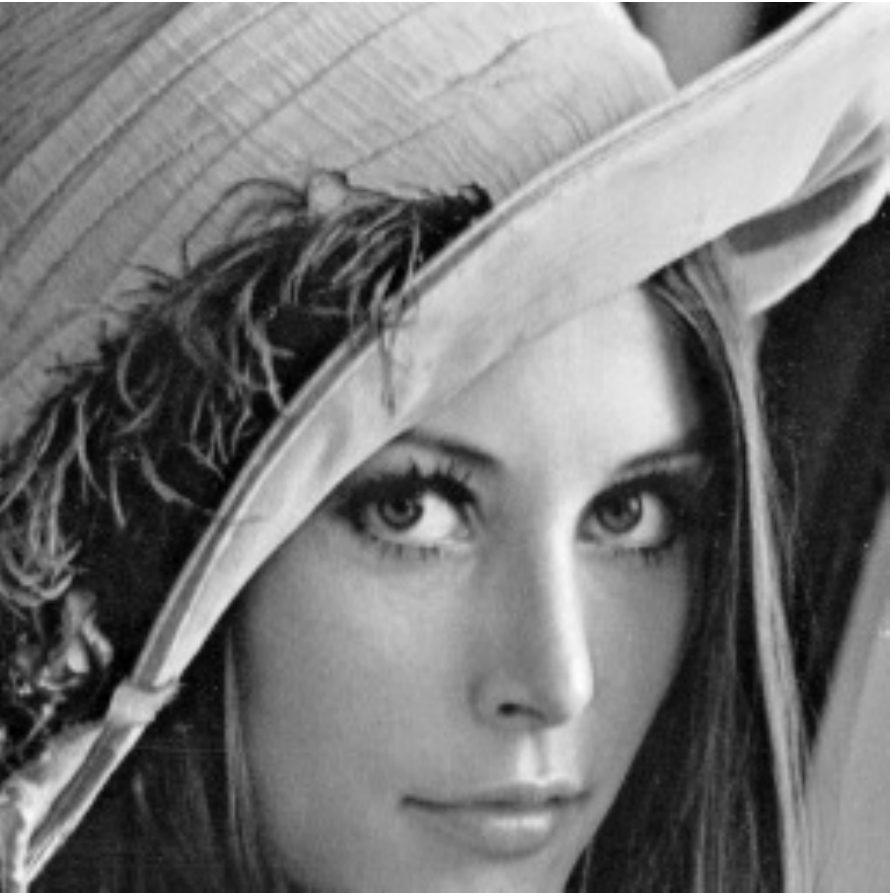}\label{fig:tv_images:orig2}}
	\hfill
	\subfloat[Degraded.]{\includegraphics[width=\mywidth]{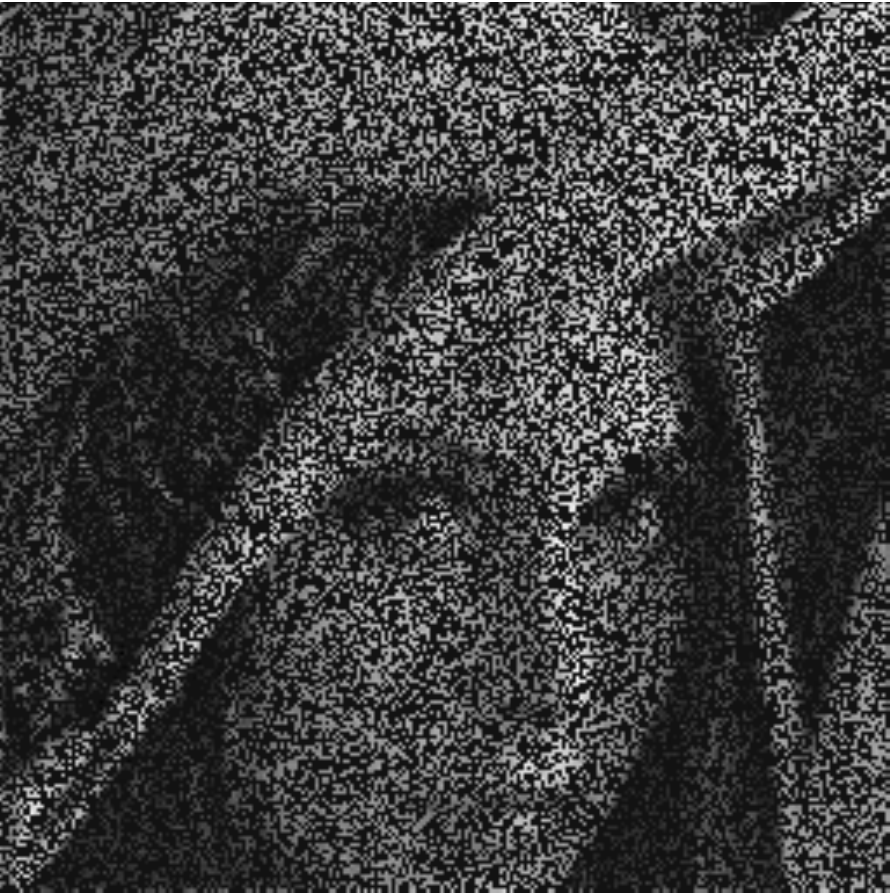}\label{fig:tv_images:degr2}}
	\hfill
	\subfloat[Zoom.]{\includegraphics[crop2,width=\mywidth]{lena_orig}}
	\hfill
	\subfloat[GPSR, SNR:~20.26~dB.]{\includegraphics[crop2,width=\mywidth]{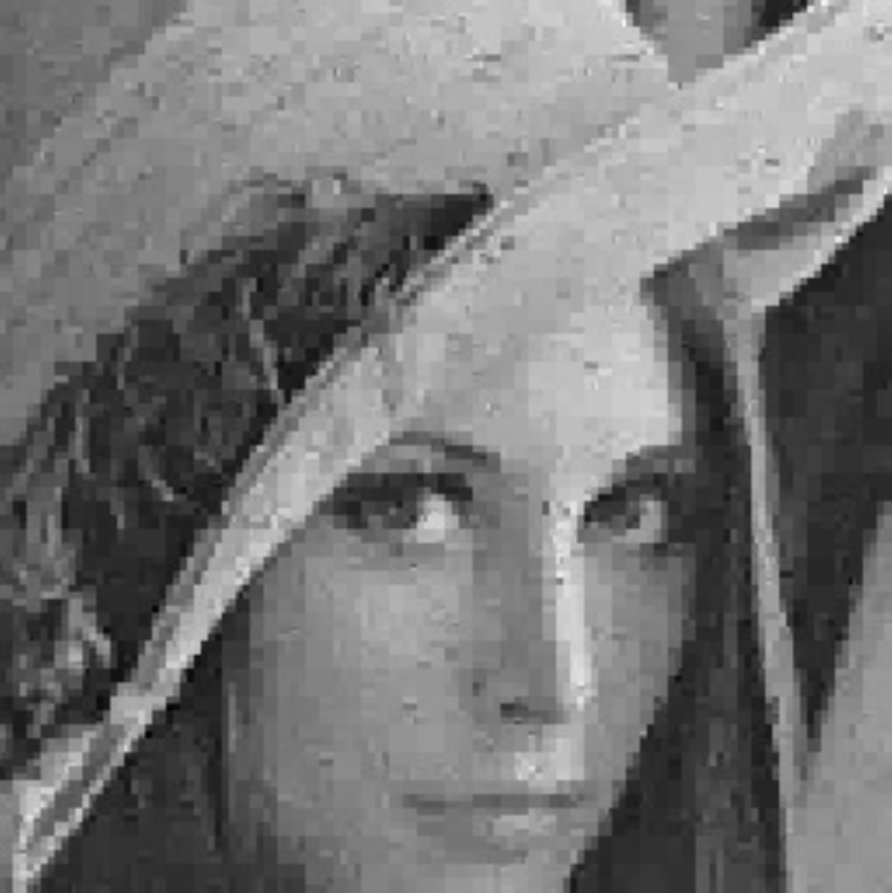}}

	\subfloat[$\bell_2$-TV, SNR:~23.18~dB.]{\includegraphics[crop2,width=\mywidth]{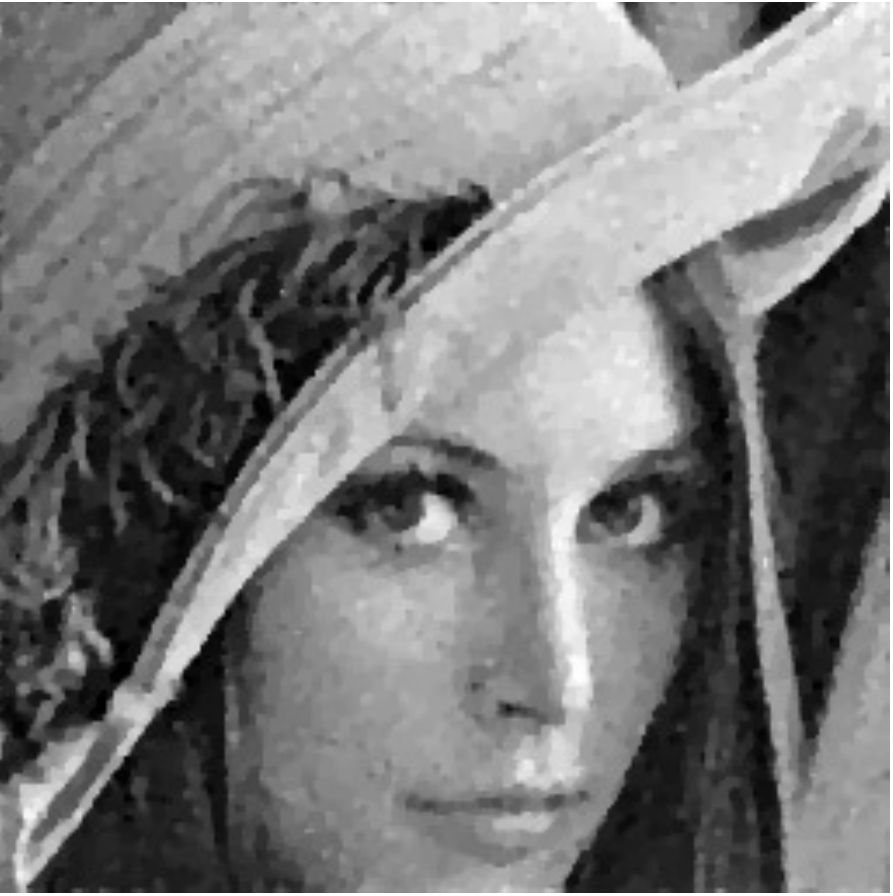}\label{fig:tv_images:rest3a}}
	\hfill
	\subfloat[$\bell_\infty$-TV, SNR:~22.77~dB.]{\includegraphics[crop2,width=\mywidth]{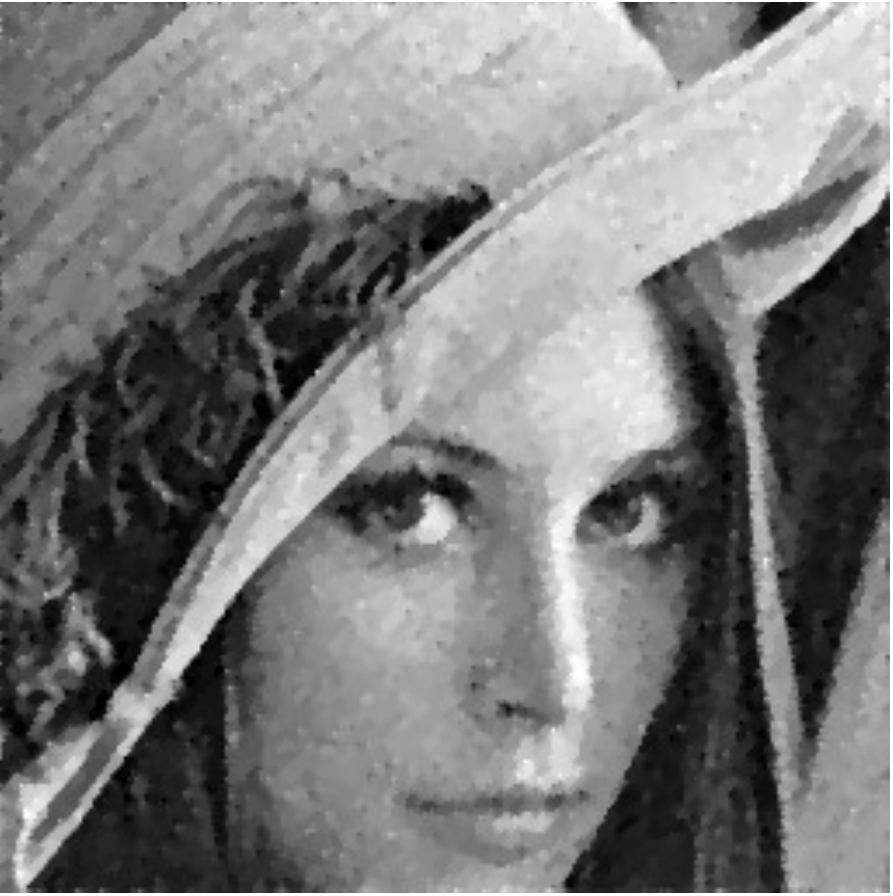}\label{fig:tv_images:rest3b}}
	\hfill
	\subfloat[$\bell_2$-NLTV, SNR:~\textbf{24.18~dB}.]{\includegraphics[crop2,width=\mywidth]{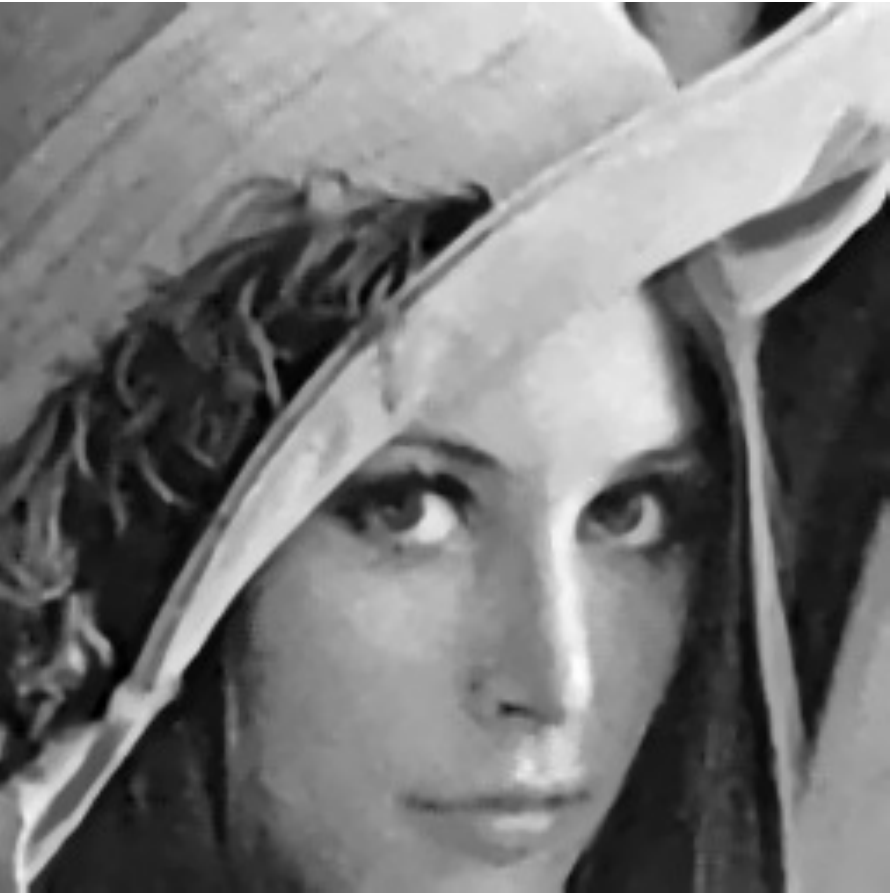}\label{fig:tv_images:rest3c}}
	\hfill
	\subfloat[$\bell_\infty$-NLTV, SNR:~24.14~dB.]{\includegraphics[crop2,width=\mywidth]{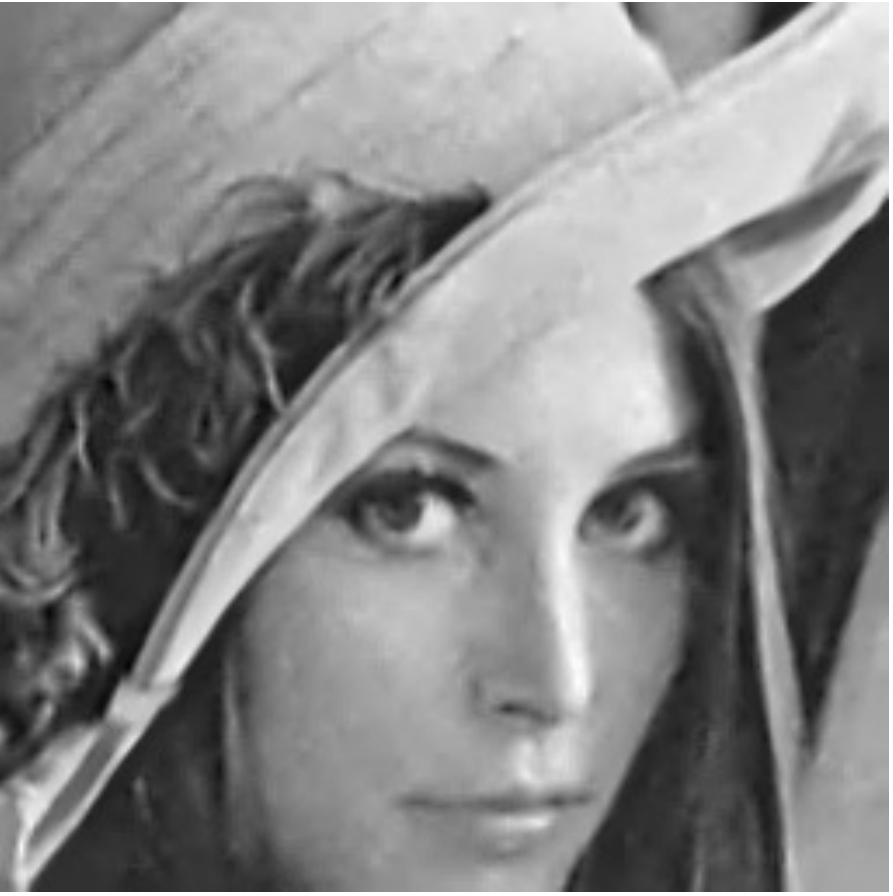}\label{fig:tv_images:rest3d}}
	
	\caption{Image restoration examples (noise parameters: blur = $3 \times 3$, $\sigma = 10$, decimation = $60\%$)}%
	\label{fig:tv_images}%
\end{figure*}
 
 \begin{figure*}%
 	\centering%
 	\subfloat[SNR comparison ($3\times3$ blur).]{\includegraphics[width=0.24\textwidth, height=2.5cm]{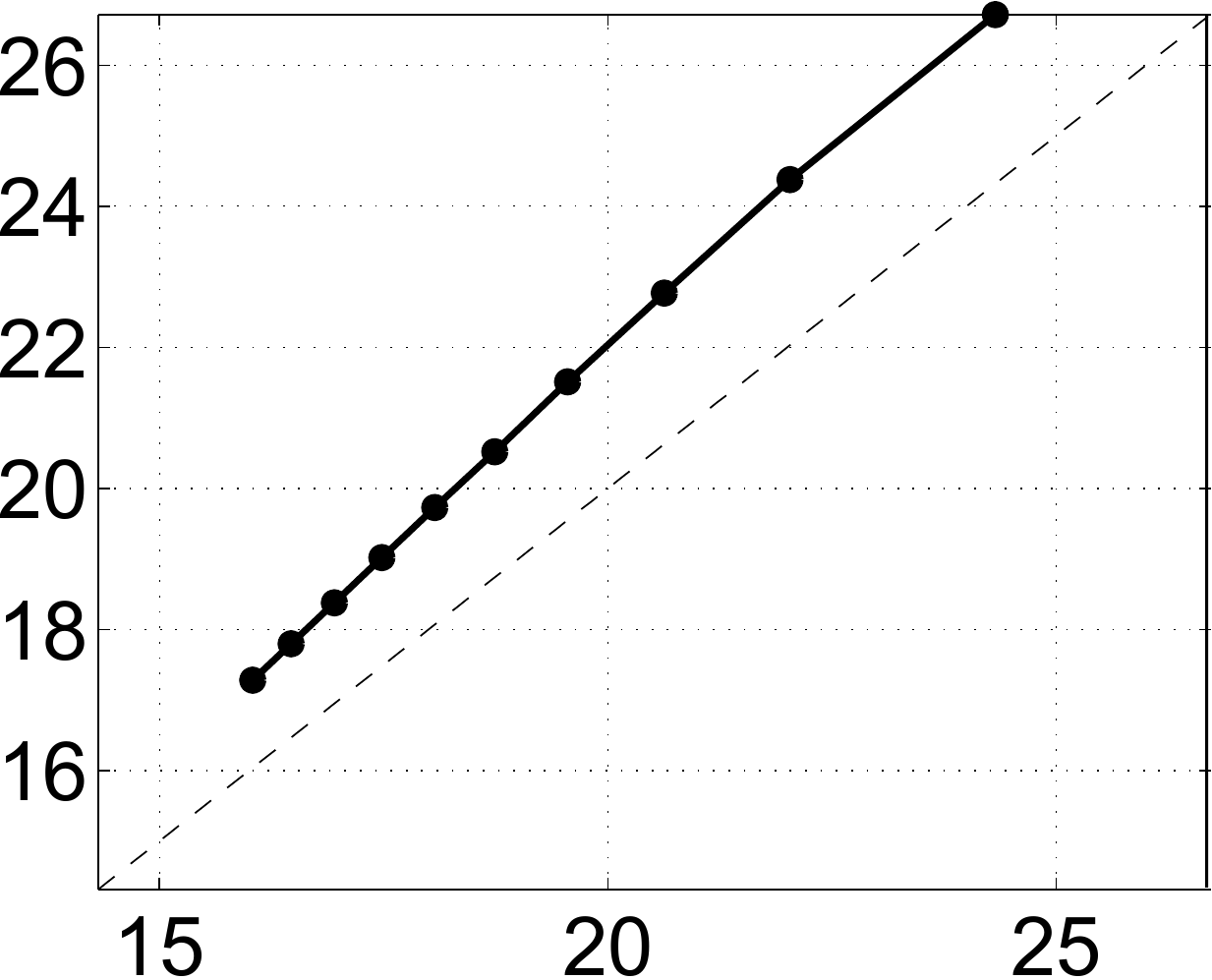}}
 	\hfill
 	\subfloat[SNR comparison ($5\times5$ blur).]{\includegraphics[width=0.24\textwidth, height=2.5cm]{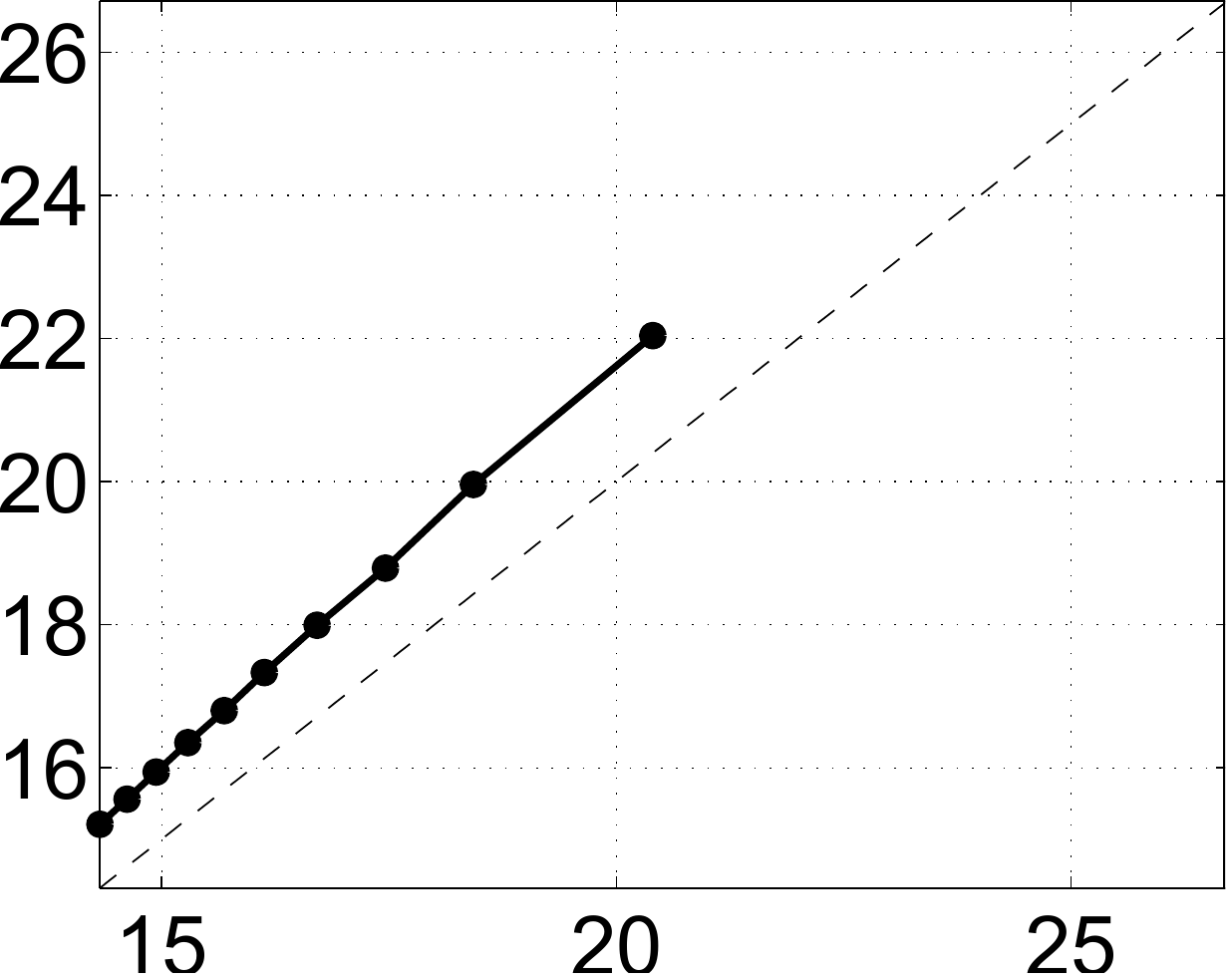}}
 	\hfill
 	\subfloat[SSIM comparison ($3\times3$ blur).]{\includegraphics[width=0.24\textwidth, height=2.5cm]{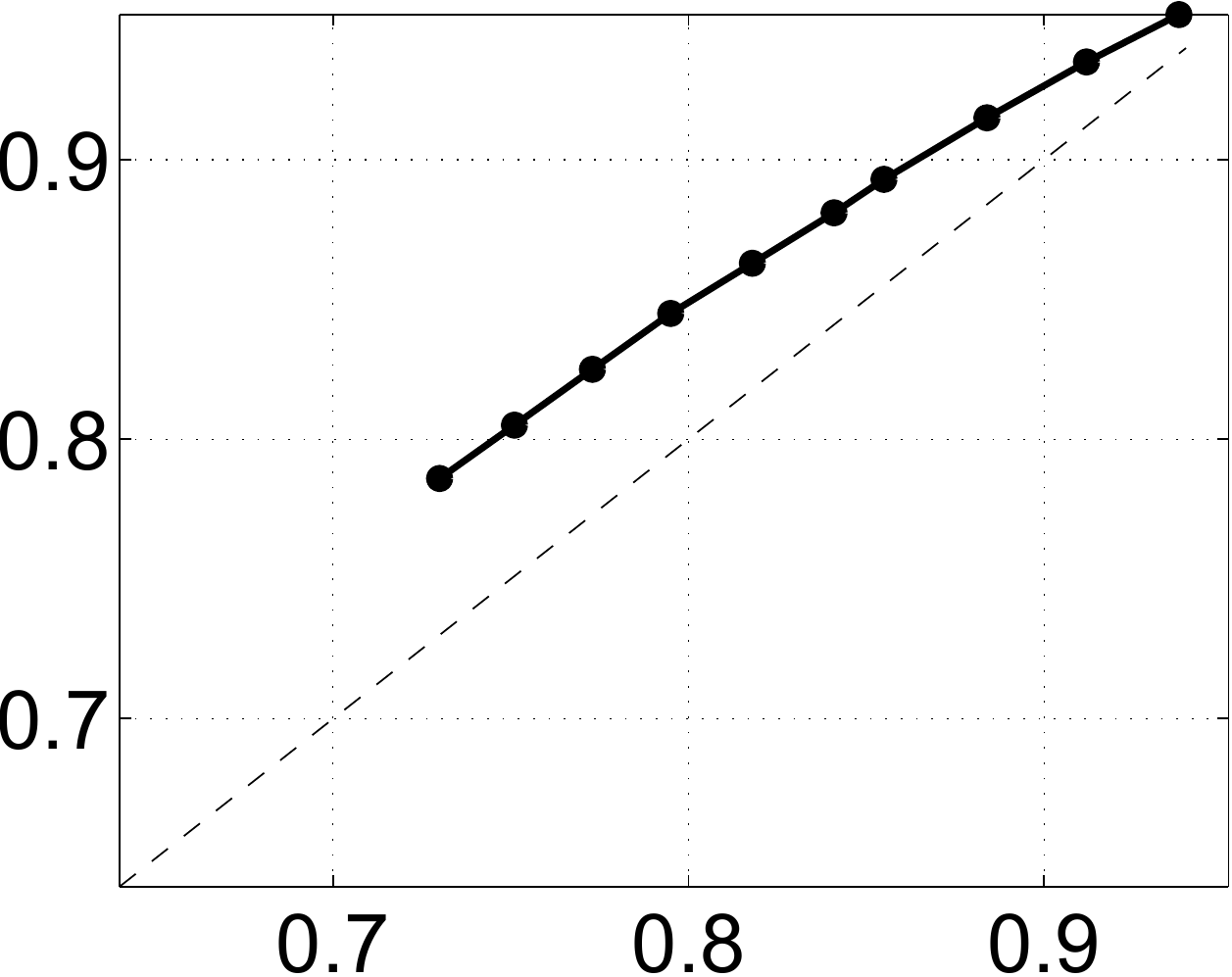}}
 	\hfill
 	\subfloat[SSIM comparison ($5\times5$ blur).]{\includegraphics[width=0.24\textwidth, height=2.5cm]{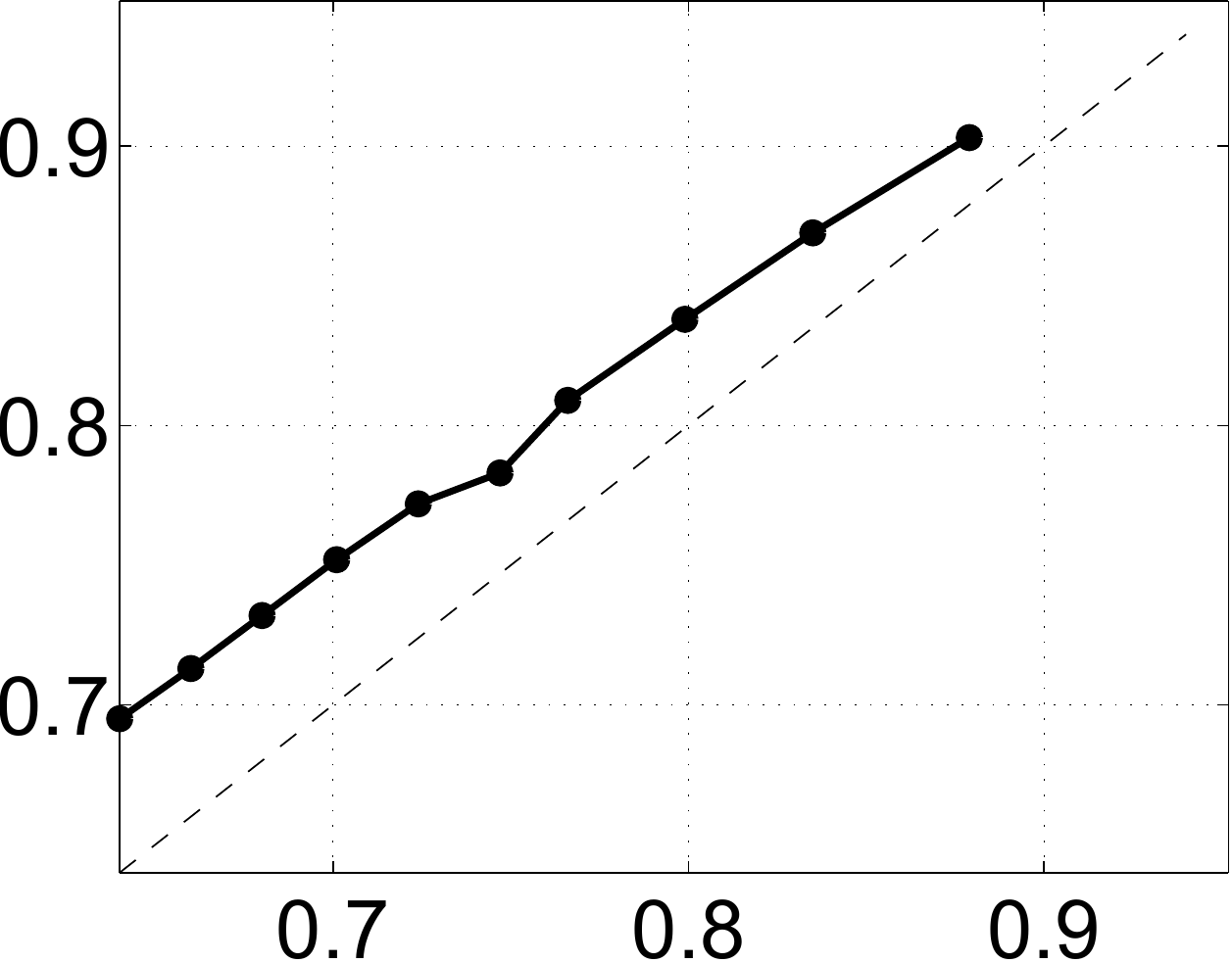}}
 	
 	\caption{SNR and SSIM values for $\bell_2$-NLTV (vertical axes) and $\bell_2$-TV (horizontal axes), for the \textit{culicoidae} image. The plots show the results obtained for $\sigma \in \{5,10,\dots,50\}$, where lower SNR or SSIM values correspond to higher $\sigma$ values. No decimation is applied in this experiment.}%
 	\label{fig:noise}%
 \end{figure*}

\subsection{Pulse shape design}

\subsubsection{Problem}
We consider a pulse shape design problem for digital communications. This problem has been previously addressed in terms of constrained optimization in 
\cite{Nobakht_R_1995_j-ieee-tcom_opt_psd,Combettes_P_1999_j-ieee-tsp_hard_cis,Combettes_PL_2008_j-ip_proximal_apdmfscvip}. 
Based on the epigraphical approach, we are able to revisit this problem by relaxing some of the involved constraints.

Five constraints arise from engineering specifications \cite{Combettes_PL_2008_j-ip_proximal_apdmfscvip}. We denote by 
$x = \big(x^{(k)}\big)_{0\leq k \leq N-1} \in \RR^N$ the pulse and by $\chi = \big(\chi^{(k)}\big)_{0\leq k \leq N-1}$ its discrete Fourier transform. The underlying sampling rate is 2560 Hz and the involved constraints are:
\begin{enumerate}
\item Bound on the modulus of $\chi$. The modulus of the Fourier transform should not exceed a prescribed bound $\gamma = 10^{-3/2}$ beyond 300 Hz. This leads to
\begin{equation}
(\forall k\in \DD_1)\qquad C_1^{(k)} = \{x \in \RR^N  \, \vert \, \vert \chi^{(k)} \vert \leq \gamma\} 
\end{equation}
where $\DD_1$ represents frequencies beyond 300 Hz.
\item Vanishing frequencies of $\chi$ at the zero frequency and at integer multiples of 50 Hz:
\begin{equation}
C_2 = \{x \in \RR^N \, \vert \, (\forall k\in \DD_2)\,\chi^{(k)} = 0\} 
\end{equation}
where $\DD_2$ denotes frequencies where the Fourier transform $\chi$ vanishes.
\item Pulse energy should not exceed a bound $\mu > 0$ (in order to avoid interference with other systems):
\begin{equation}
C_3 = \{x \in \RR^N \, \vert \, \Vert x \Vert \leq \mu\}.
\end{equation}
\item Symmetry of the pulse and its mid-point value should be equal to 1:
\begin{align}
C_4 = \{&x \in \RR^N \, \vert \, x^{(N/2)} = 1 \, \nonumber\\
&\mbox{and}\, (\forall k\in \{0,\ldots,N/2\})\,x^{(k)} = x^{(N-1-k)}\}.
\end{align}
\item Pulse duration should be 50 ms and it should have periodic zero crossings every
3.125 ms:
\begin{equation}
C_5 = \{x \in \RR^N \, \vert \, (\forall k\in \DD_3)\, x^{(k)}  = 0\} 
\end{equation}
where $\DD_3$ is the set of time indices in the zero areas.
\end{enumerate}
Some of the above constraints (e.g. the second and the last ones)  are incompatible and, in order to make the problem feasible,
we propose to replace the constraint sets $(C_1^{(k)})_{1 \le k \le N}$ with 
\begin{equation}
C_1 = \{x \in \RR^N  \, \vert \, \sum_{k\in \DD_1} d^\beta_{C_1^{(k)}}(\chi^{(k)} ) \le \varepsilon\} 
\end{equation}
where $\varepsilon>0$ and $\beta\in [1,+\infty [$.

\subsubsection{Algorithmic solution}
The resulting optimization problem reads
\begin{equation}
\underset{x\in \RR^N}{\mbox{minimize}}  \; \Vert x \Vert^2 +\sum_{s=1}^5 \iota_{C_s}(x).
\label{eq:cdf}
\end{equation}
A squared $\ell_2$-norm has been included in the criterion in order to ensure the uniqueness of the solution. Algorithms usually employed for solving such a kind of problem \cite{Combettes_PL_2008_j-ip_proximal_apdmfscvip} require
closed form expressions of the projection associated with each involved constraint.
The projections onto $C_2$, $C_4$, and $C_5$ are standard projections onto vector subspaces. The projection onto $C_3$ is the projection onto an $\ell_2$-ball which also has a closed form. The difficulty stems from the computation of the projection onto the convex set $C_1$. We propose to split it into two constraints as proposed in Section~\ref{sec:spl}. More specifically, Problem \eqref{eq:cdf} can be equivalently written as
\begin{equation}\label{eq:cdf_epi}
\minimize{(x,\zeta_1) \in \RR^N\times V_1} \; \Vert x \Vert^2 + \iota_{E_1}(x,\zeta_1)+\sum_{s=2}^5 \iota_{C_s}(x)
\end{equation}
where  
\begin{equation}\label{e:defV2}
V_1 = \menge{\zeta\in \RR^{N}}{{\sf 1}_{\DD_1}^\top\zeta \leq \varepsilon}
\end{equation}
and the closed convex set $E$ is
\begin{equation}\label{e:defE2}
E_1 = \big\{ {(x,\zeta)\in \RR^N\times \RR^N}~\big|~{(\forall k \in \DD_1)} \;d^\beta_{C_1^{(k)}}(x) \leq \zeta^{(k)}
\big\}.
\end{equation}
The projection onto $V_1$ is well-known \cite{Hiriart_Urruty_1996_book_convex_amaIf}, while the projection onto $E_1$ follows from Proposition \ref{ex:epidistl}.

\subsubsection{Numerical results}
Several experiments are performed in order to compare state-of-the-art solutions with the proposed constrained formulation.
The results are summarized in Figs~\ref{fi:res_fil_des_1} and \ref{fi:res_fil_des_3}. 
Fig.~\ref{fi:res_fil_des_1} presents state-of-the-art results (from \cite{Combettes_PL_2008_j-ip_proximal_apdmfscvip}). Fig.~\ref{fi:res_fil_des_2} shows the results obtained with the proposed solution for $\beta=1$ and different values of $\varepsilon$ leading to admissible solutions while Fig.~\ref{fi:res_fil_des_3} presents admissible solutions for $\beta=2$   and different values of $\varepsilon$. Note that for large values of $\varepsilon$, the solutions for $\beta=1$ or $\beta=2$ converge to a solution of the unconstrained (without imposing $C_1$) problem (cf. Fig~\ref{fi:res_fil_des_1}-right). For $\beta=2$, it is also interesting to experimentally observe that the estimated pulse for the smallest value of $\varepsilon$ leading an admissible solution (cf. Fig~\ref{fi:res_fil_des_3}-left) is similar to the solution proposed in \cite{Combettes_PL_2008_j-ip_proximal_apdmfscvip} (cf. Fig.~\ref{fi:res_fil_des_1} - middle). 

As illustrated by these experiments, the proposed approach allows us to gain more design flexibility at the expense of 
a small additional computational cost.

\begin{figure*}
\centering
 \begin{tabular}{ccc}
\includegraphics[width=4.8cm]{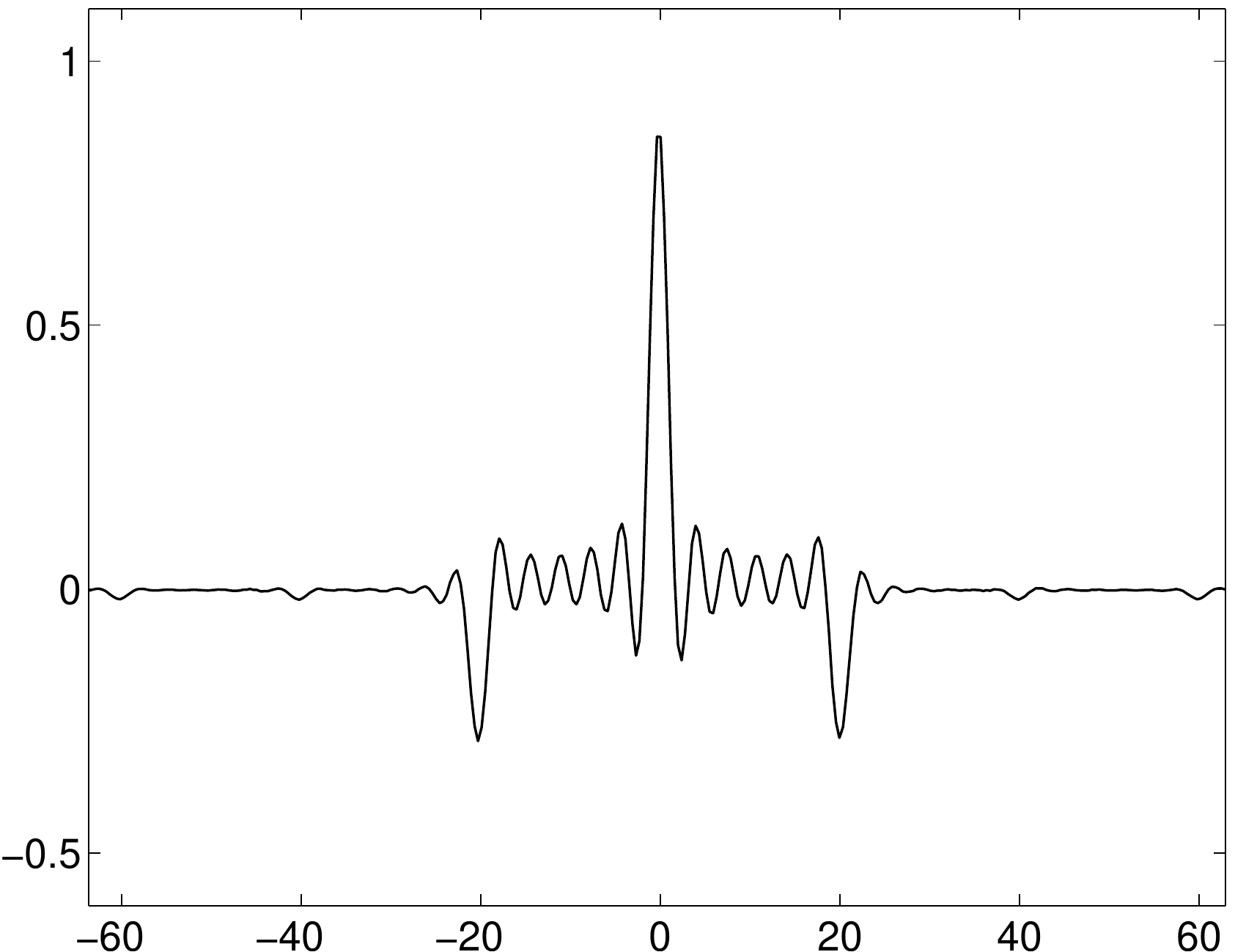}&\includegraphics[width=4.8cm]{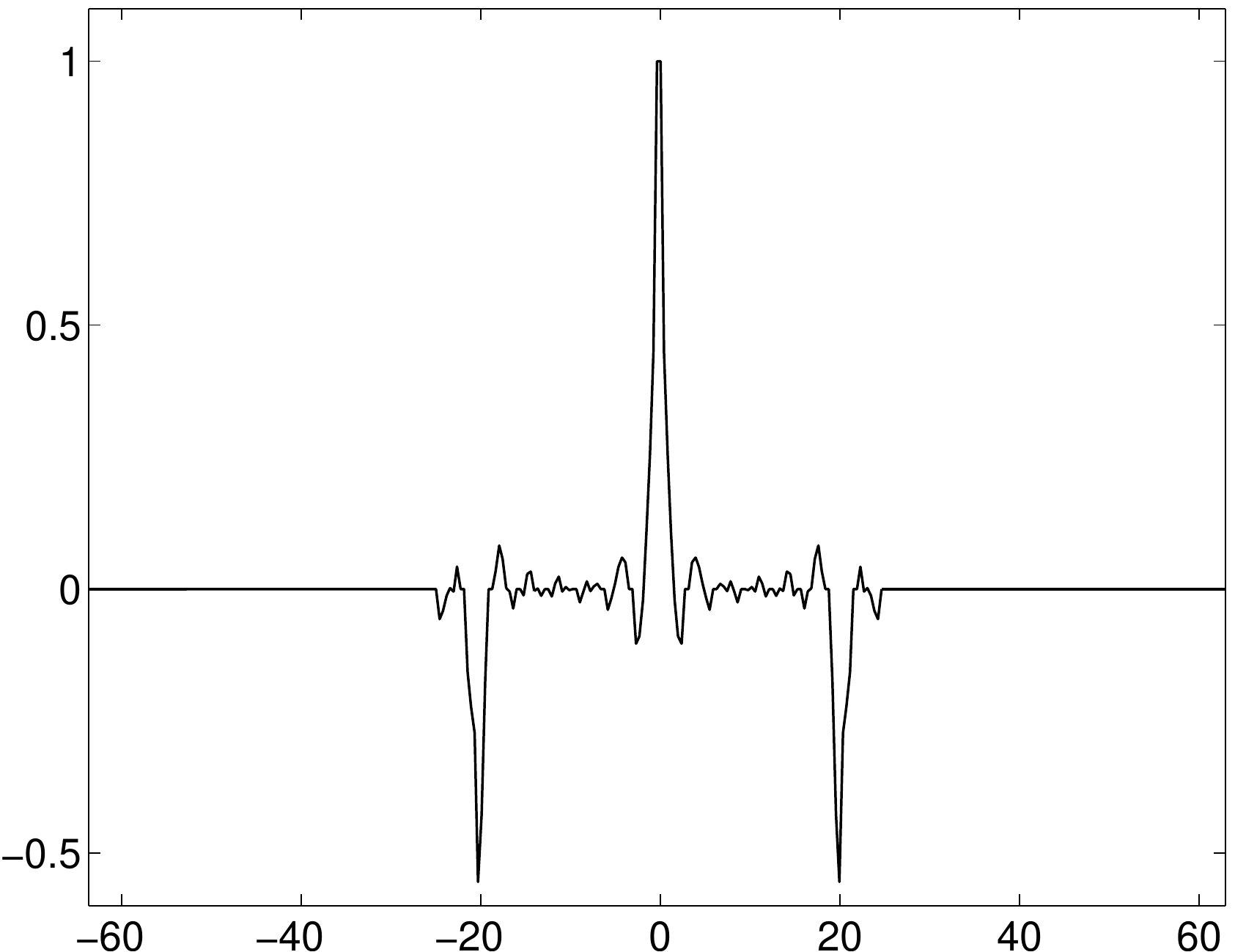}&\includegraphics[width=4.8cm]{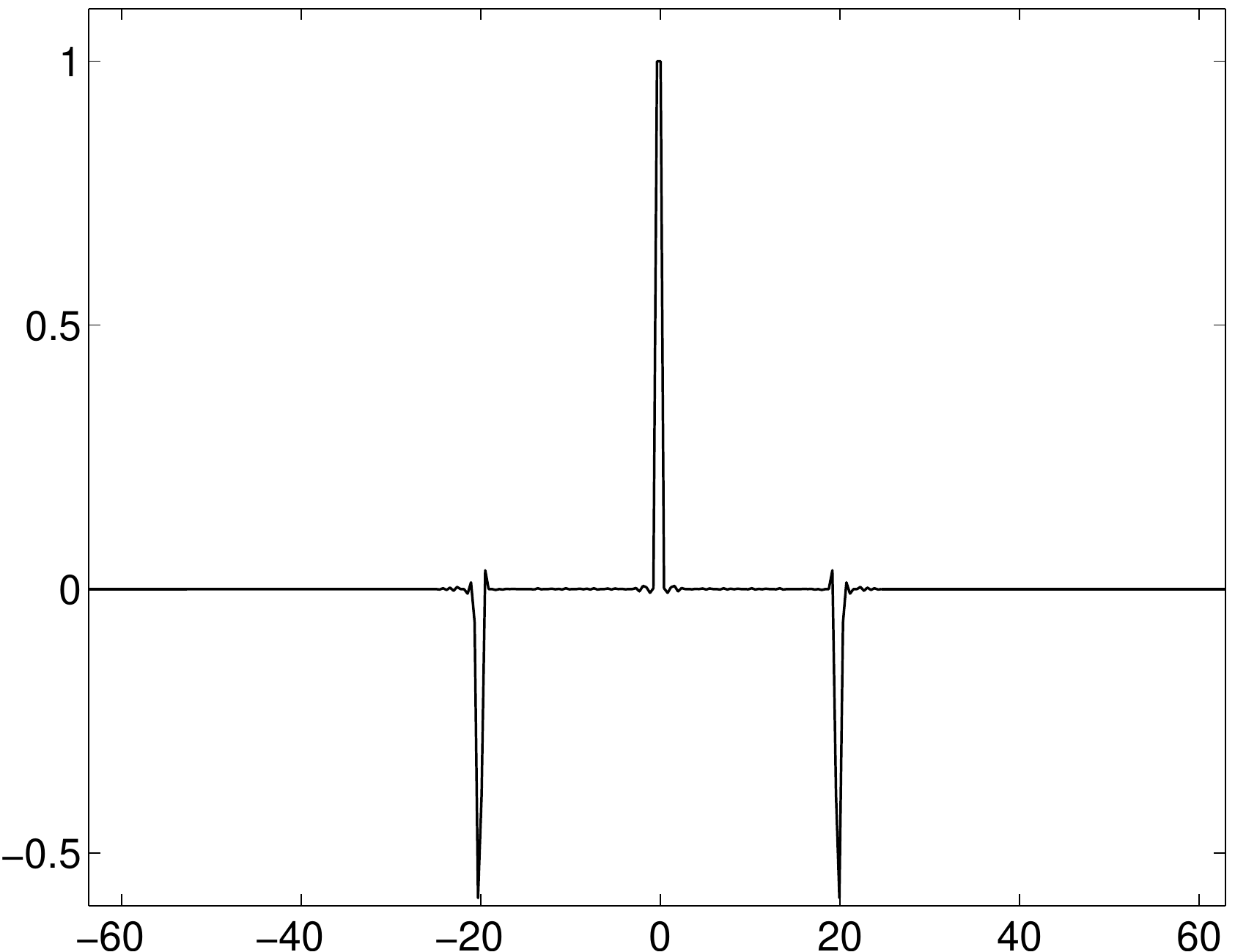}\\
\includegraphics[width=4.8cm]{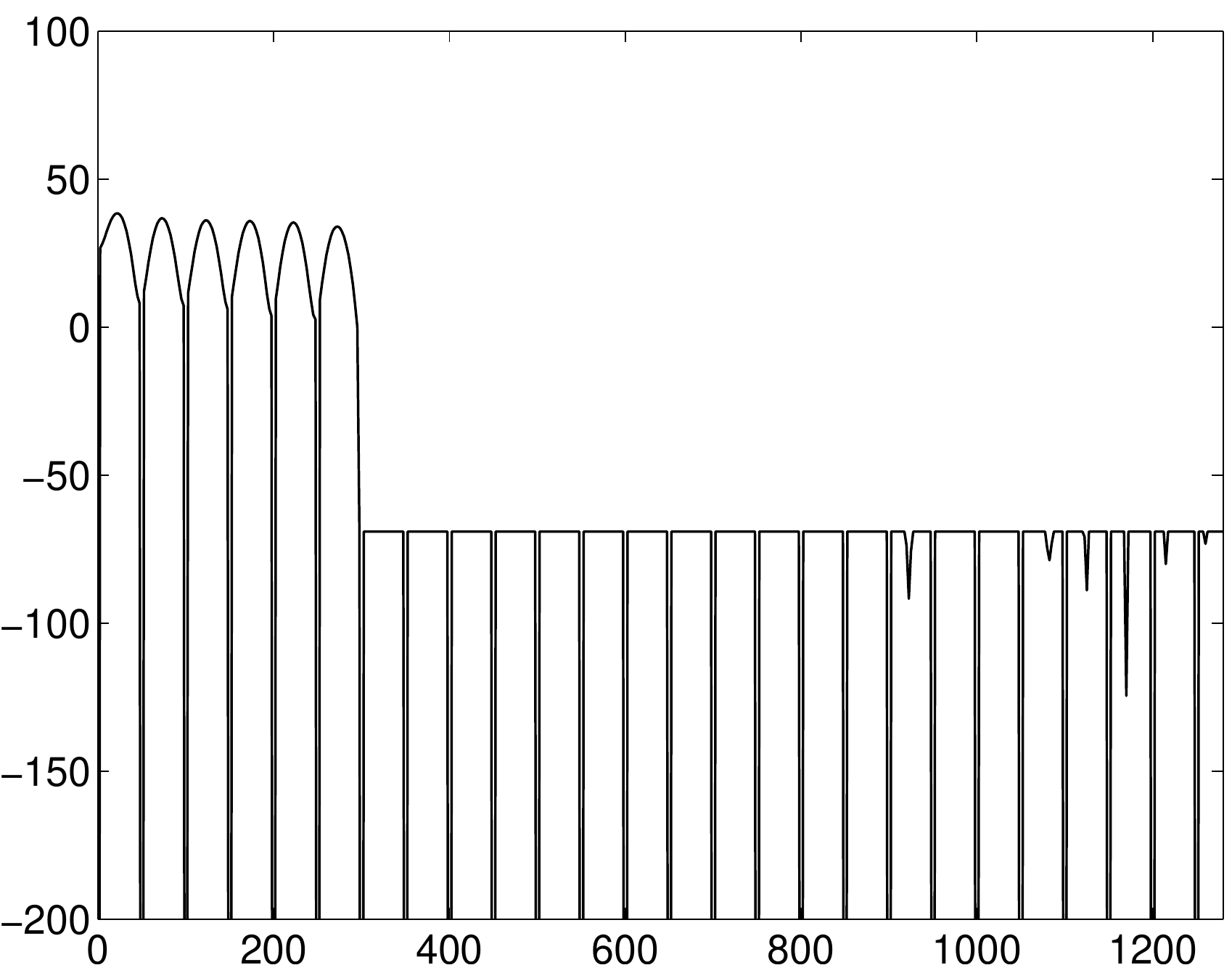} &\includegraphics[width=4.8cm]{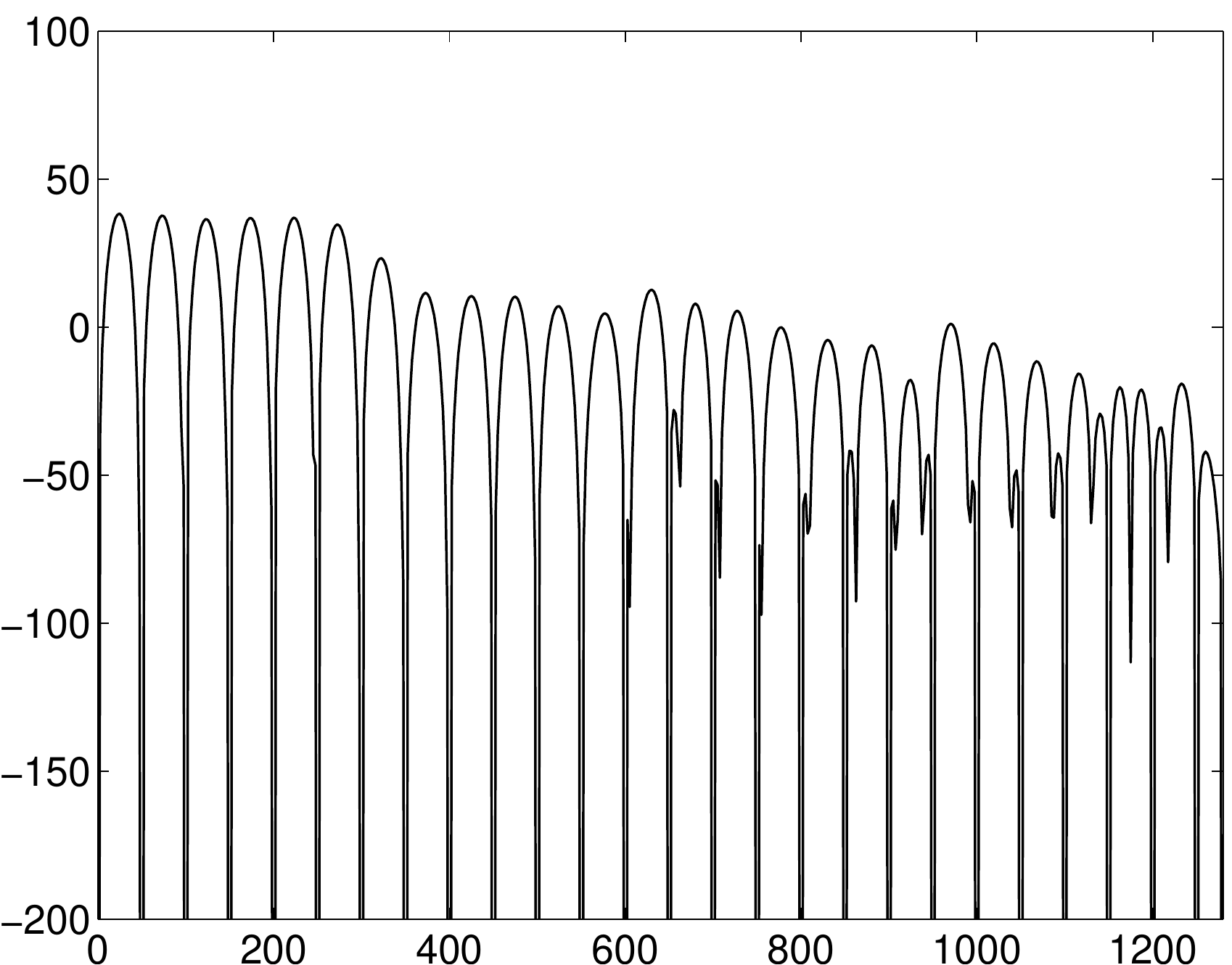} &\includegraphics[width=4.8cm]{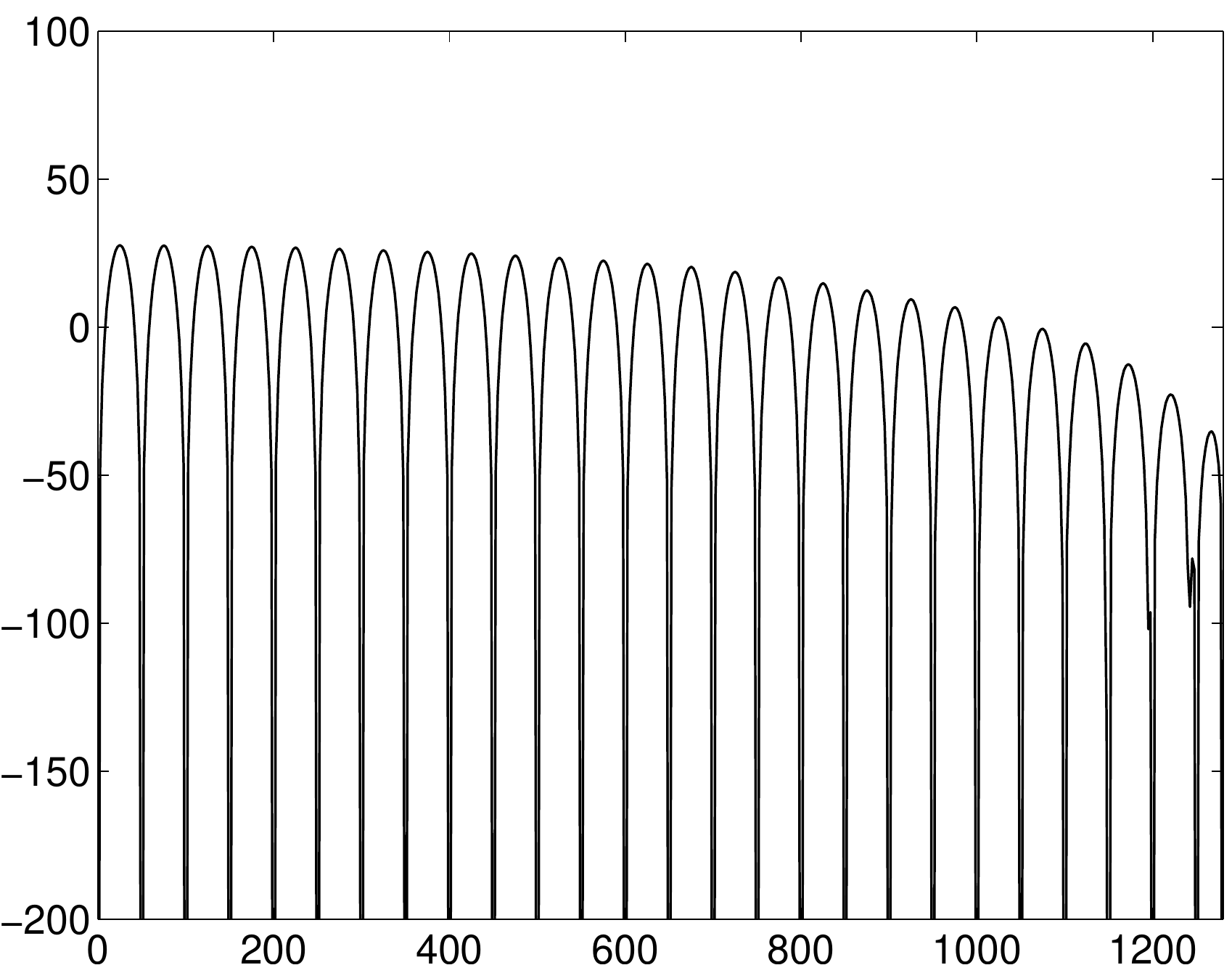} \\
 \end{tabular}
\caption{State-of-the-art results. 
First row: pulse $x$, second row: Fourier transform $\chi$ of the pulse. 
The results in the first column are obtained by minimizing $d_{C_4}(x) + d_{C_5}(x)$ subject to $x \in \Big(\bigcap_{k\in \DD_1}C_1^{k}\Big)\cap (C_2 \cap C_3)$, the results in the second one are obtained by minimizing $\sum_{k\in \DD_1 }d_{C^k_1}^2(x)$ subject to $x \in C_2 \cap C_3 \cap C_4 \cap C_5$, and the last one presents a solution in $C_2 \cap C_3 \cap C_4 \cap C_5$. \label{fi:res_fil_des_1}}
\end{figure*}

\begin{figure*}
\centering
 \begin{tabular}{ccc}
\includegraphics[width=5cm]{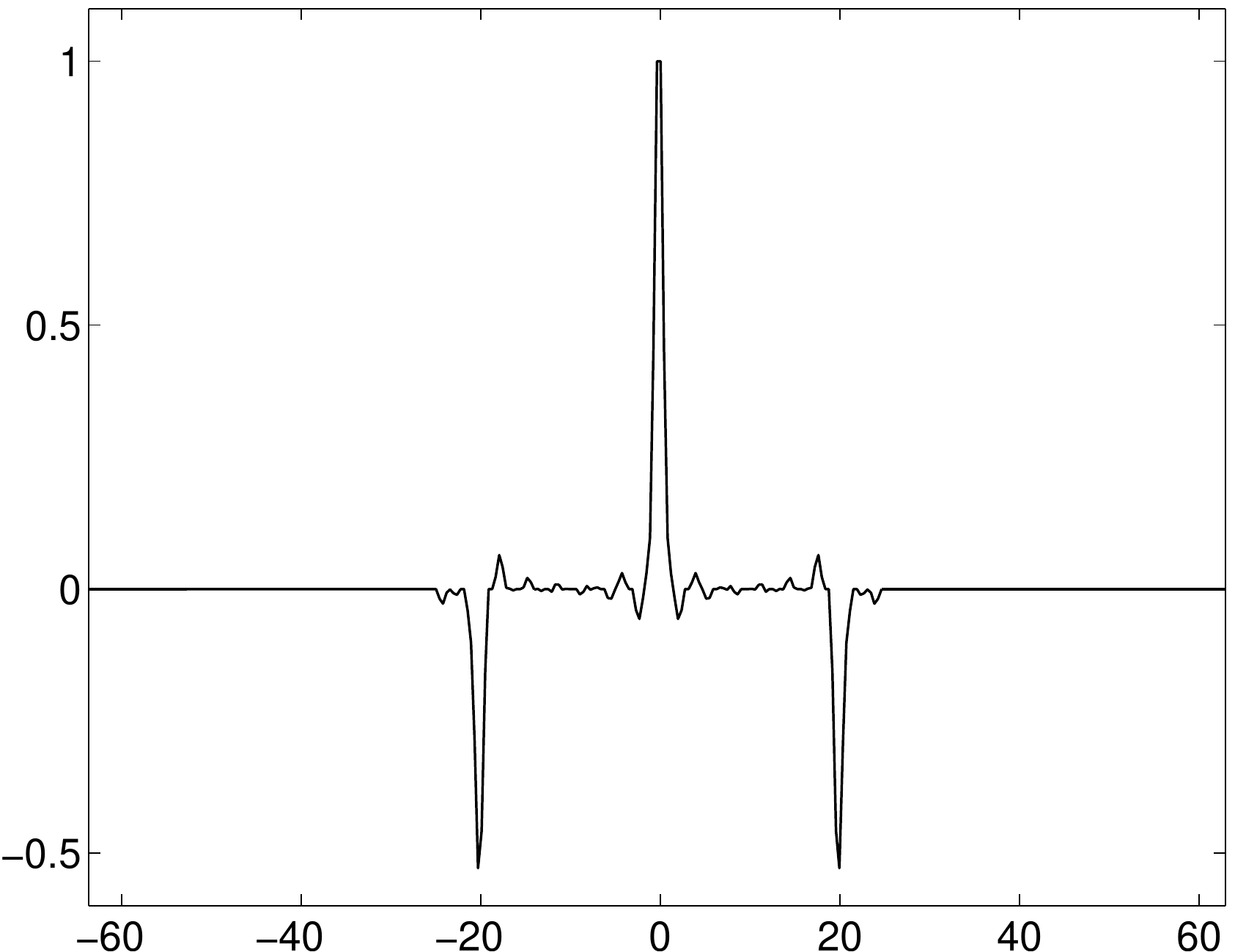}&\includegraphics[width=5cm]{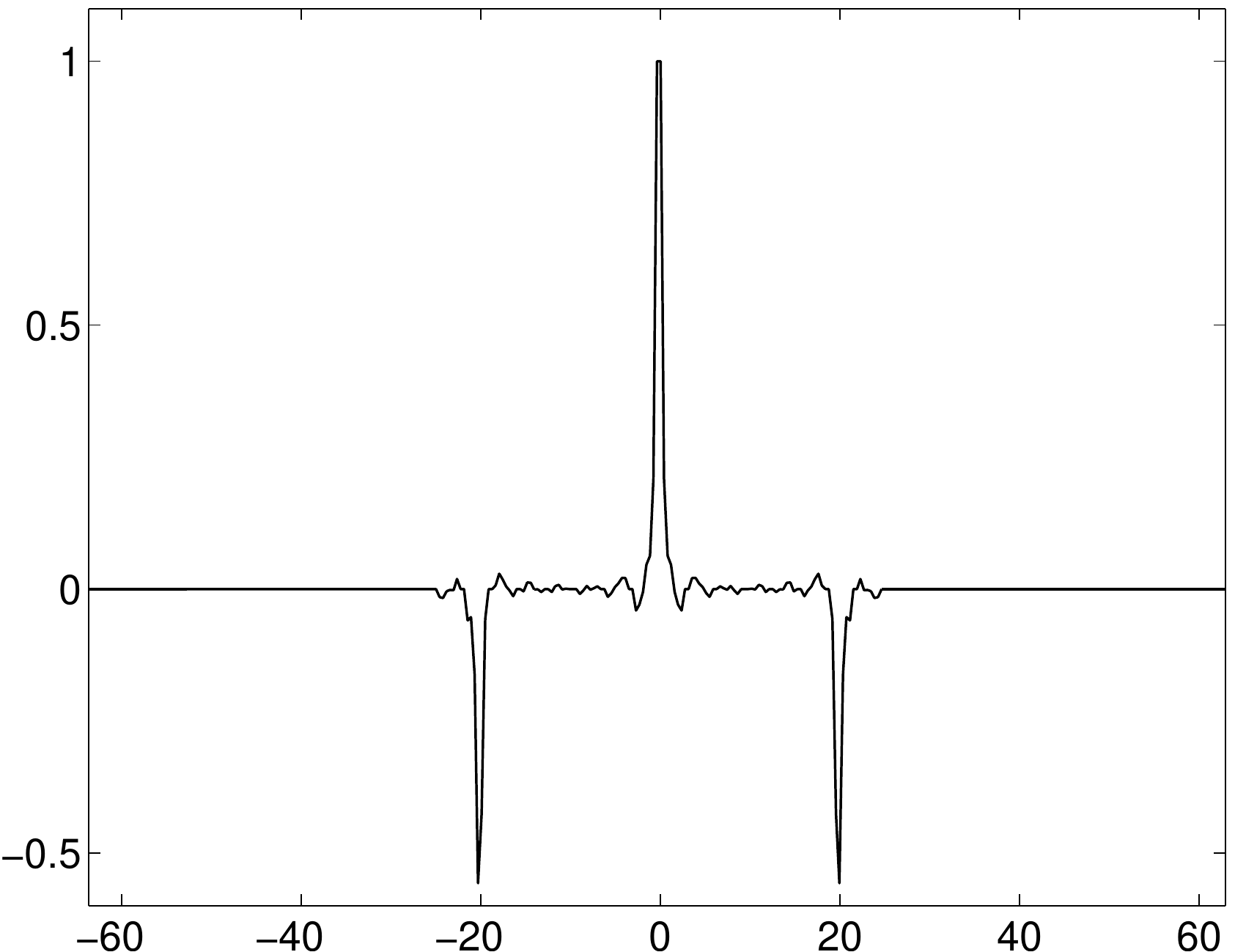}&\includegraphics[width=5cm]{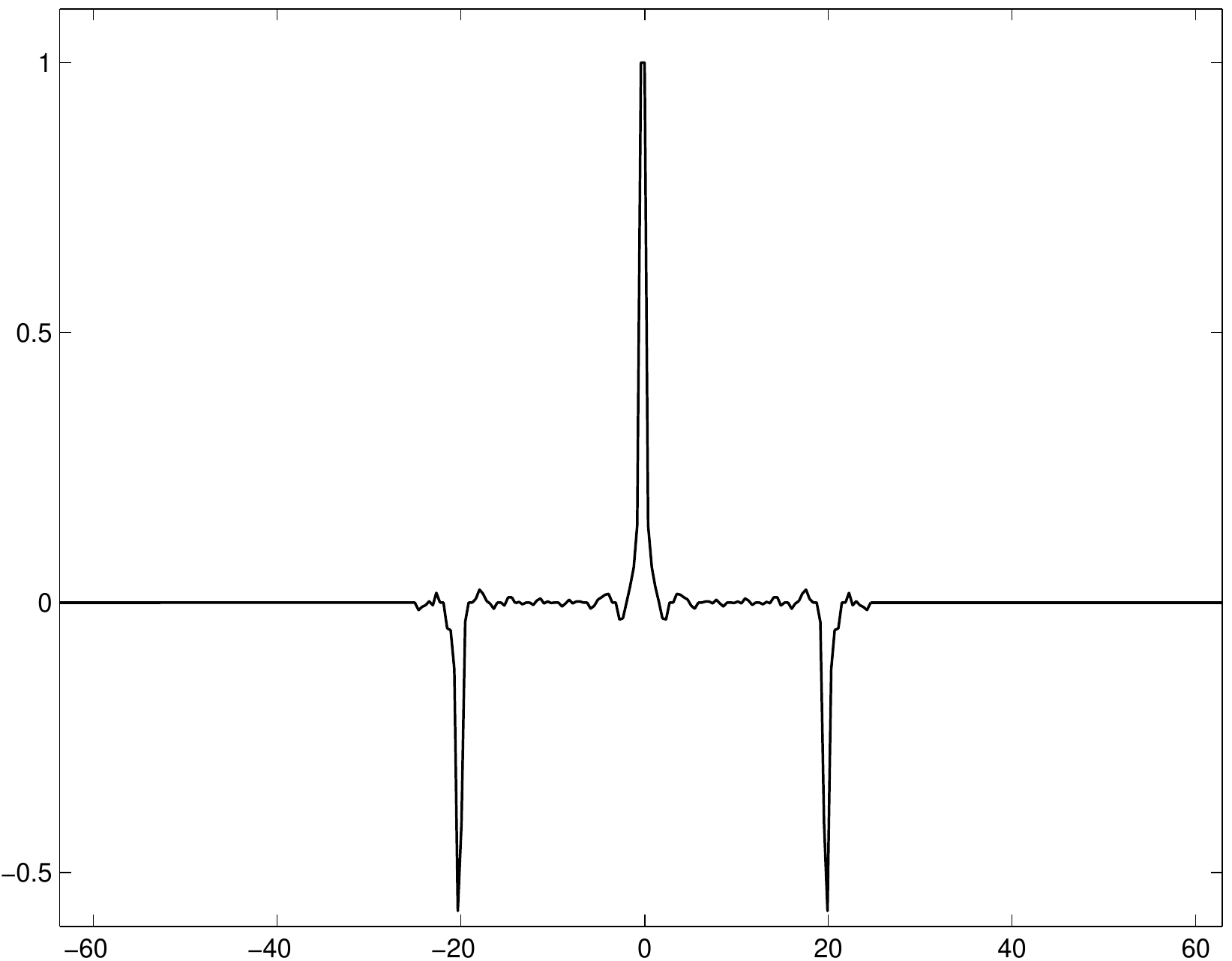}\\
\includegraphics[width=5cm]{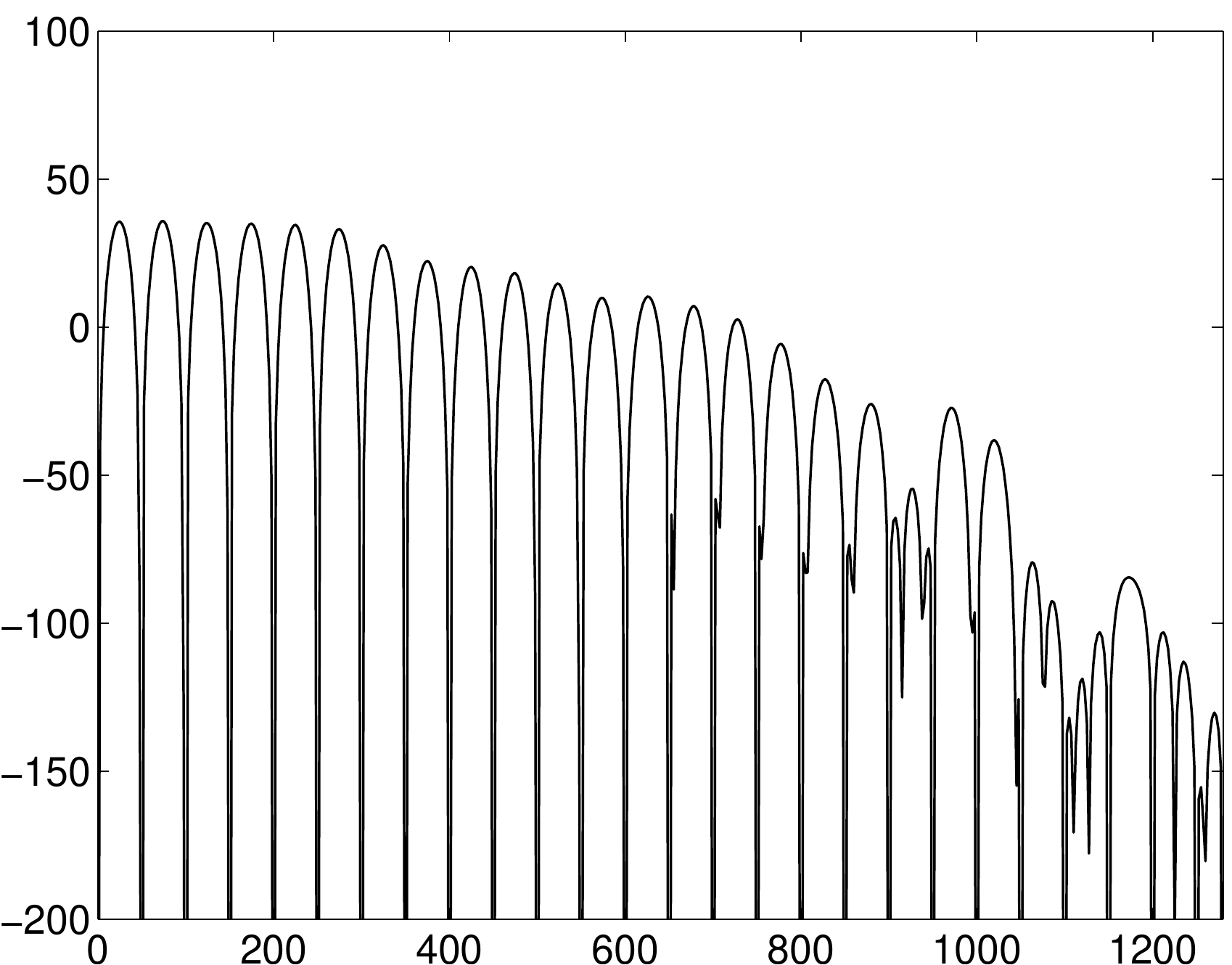} &\includegraphics[width=5cm]{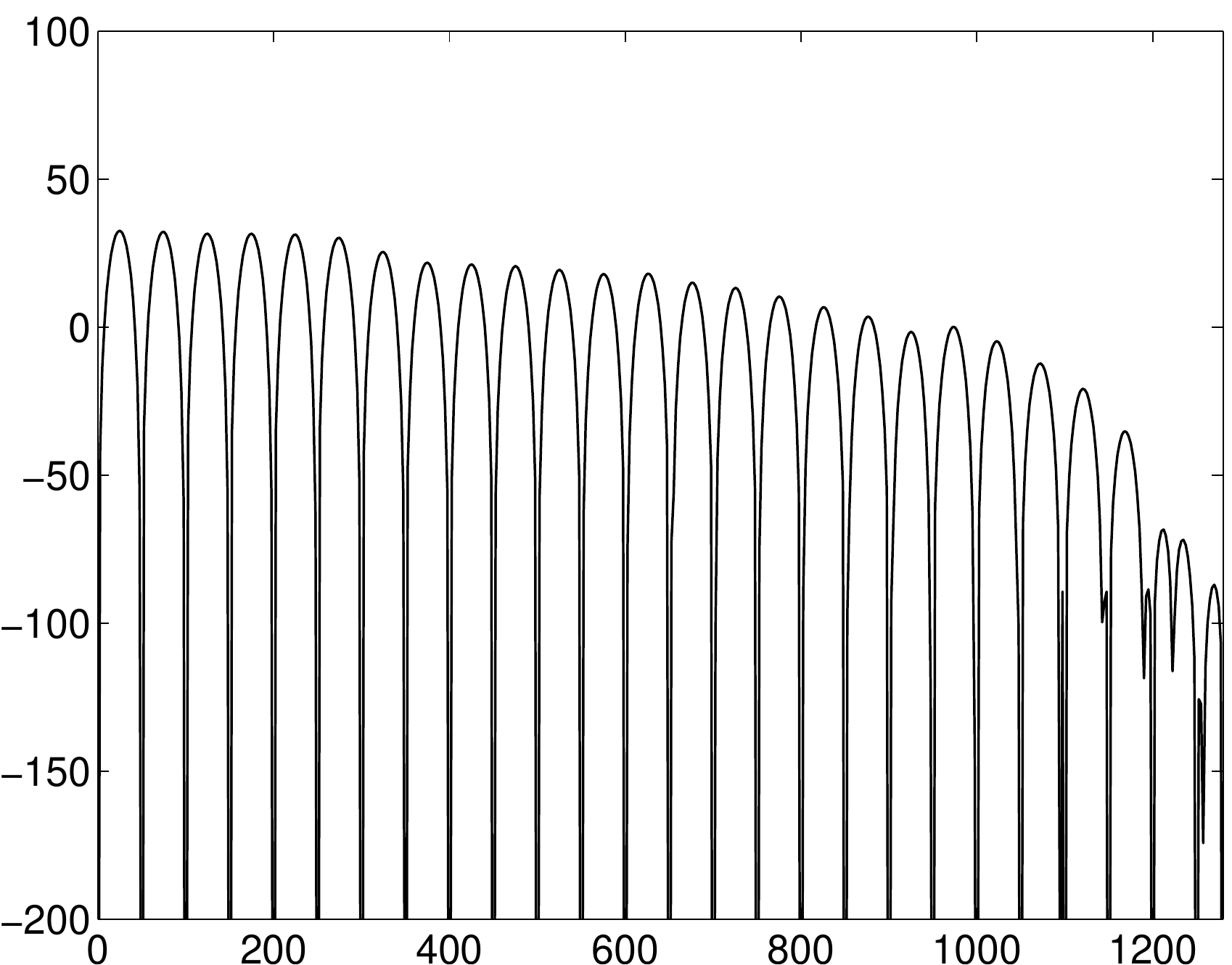} &\includegraphics[width=5cm]{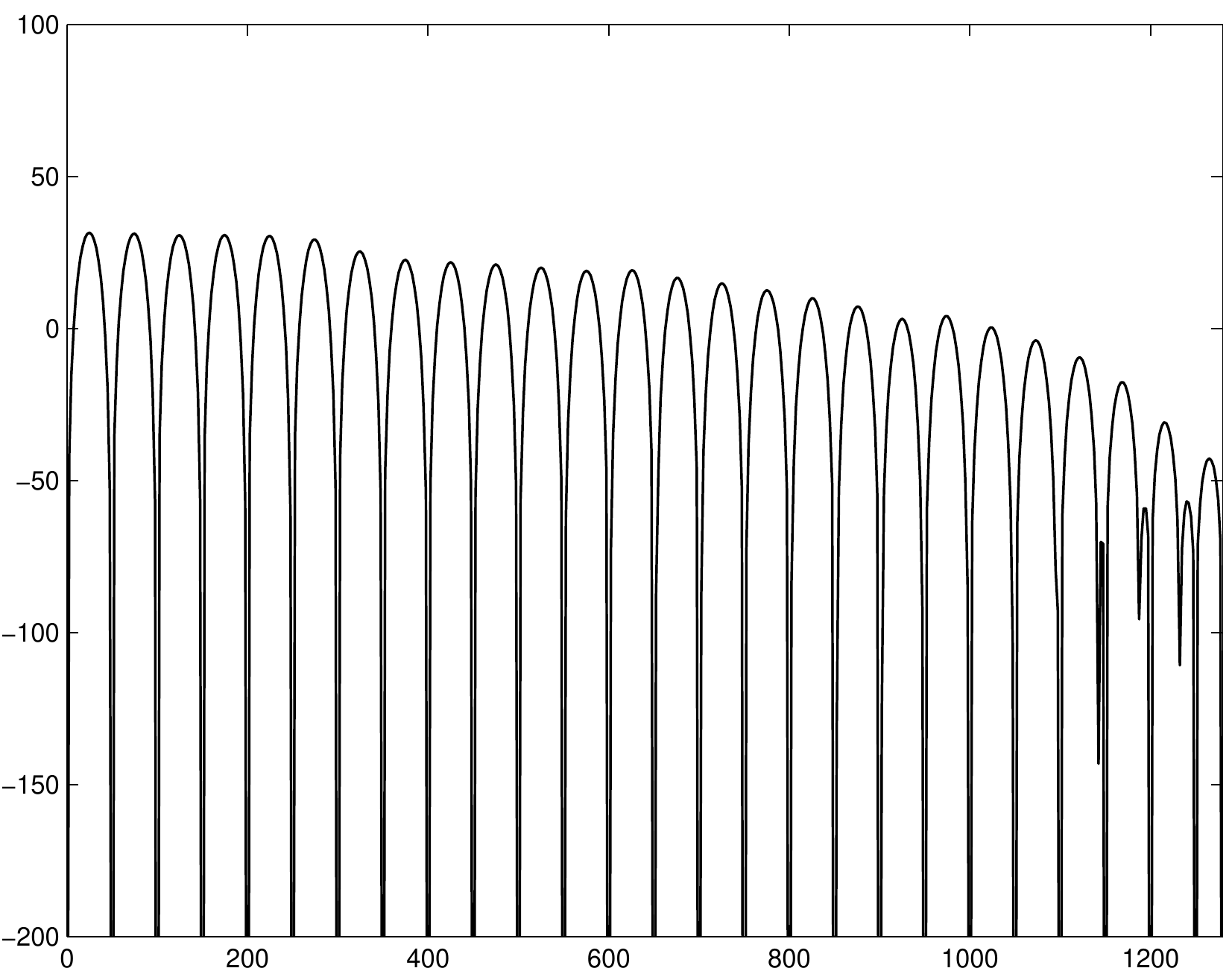}\\
 \end{tabular}
  \caption{Results obtained from \eqref{eq:cdf_epi} for $\beta=1$ and
  $\varepsilon =800 $ (first column), $\varepsilon = 1000$ (second
  column),  and $\varepsilon = 2000$ (third column). First row: pulse $x$,
  second row: Fourier transform $\chi$ of the pulse.
  \label{fi:res_fil_des_2}}
\end{figure*}
%

\begin{figure*}
\centering
 \begin{tabular}{ccc}
\includegraphics[width=5cm]{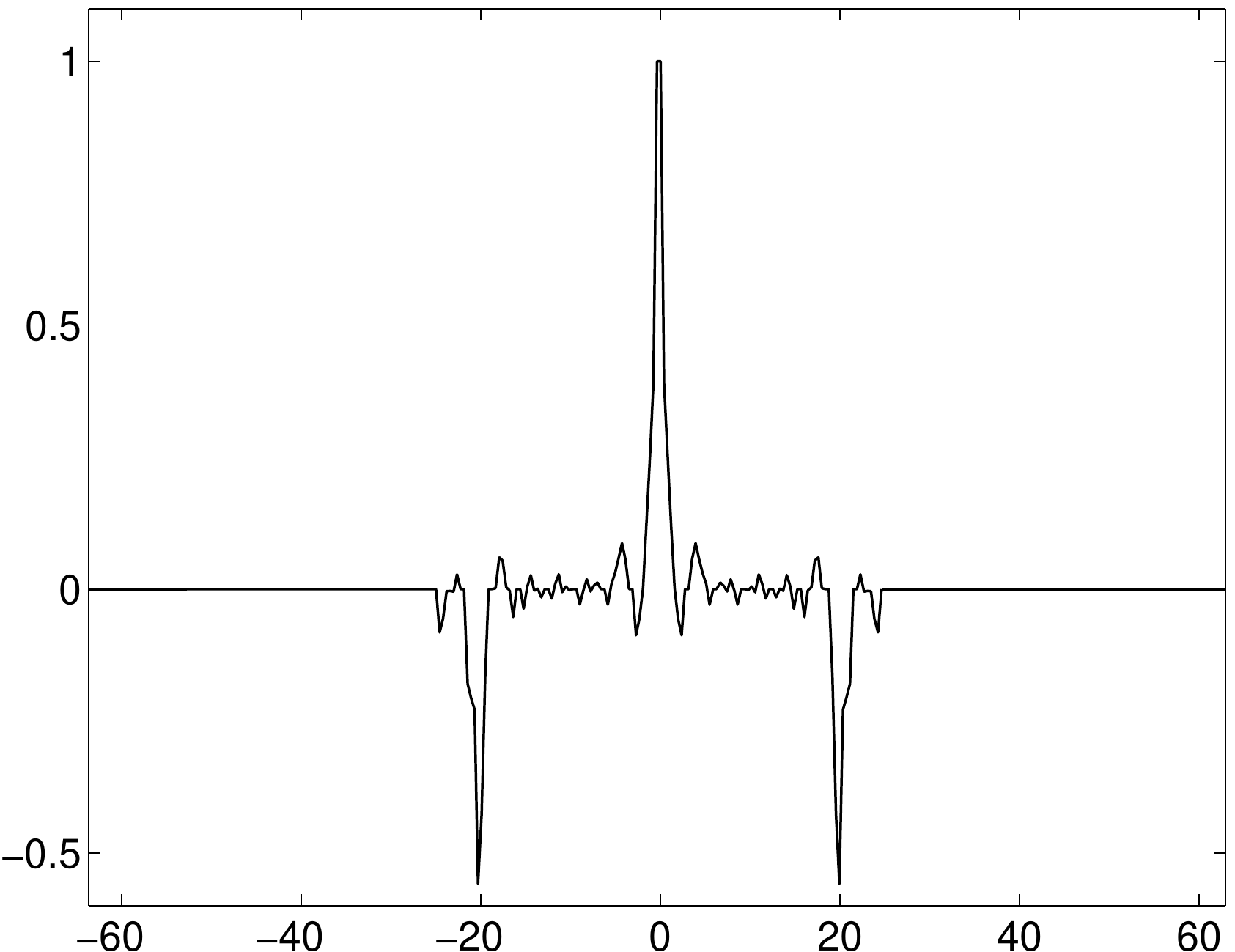}&\includegraphics[width=5cm]{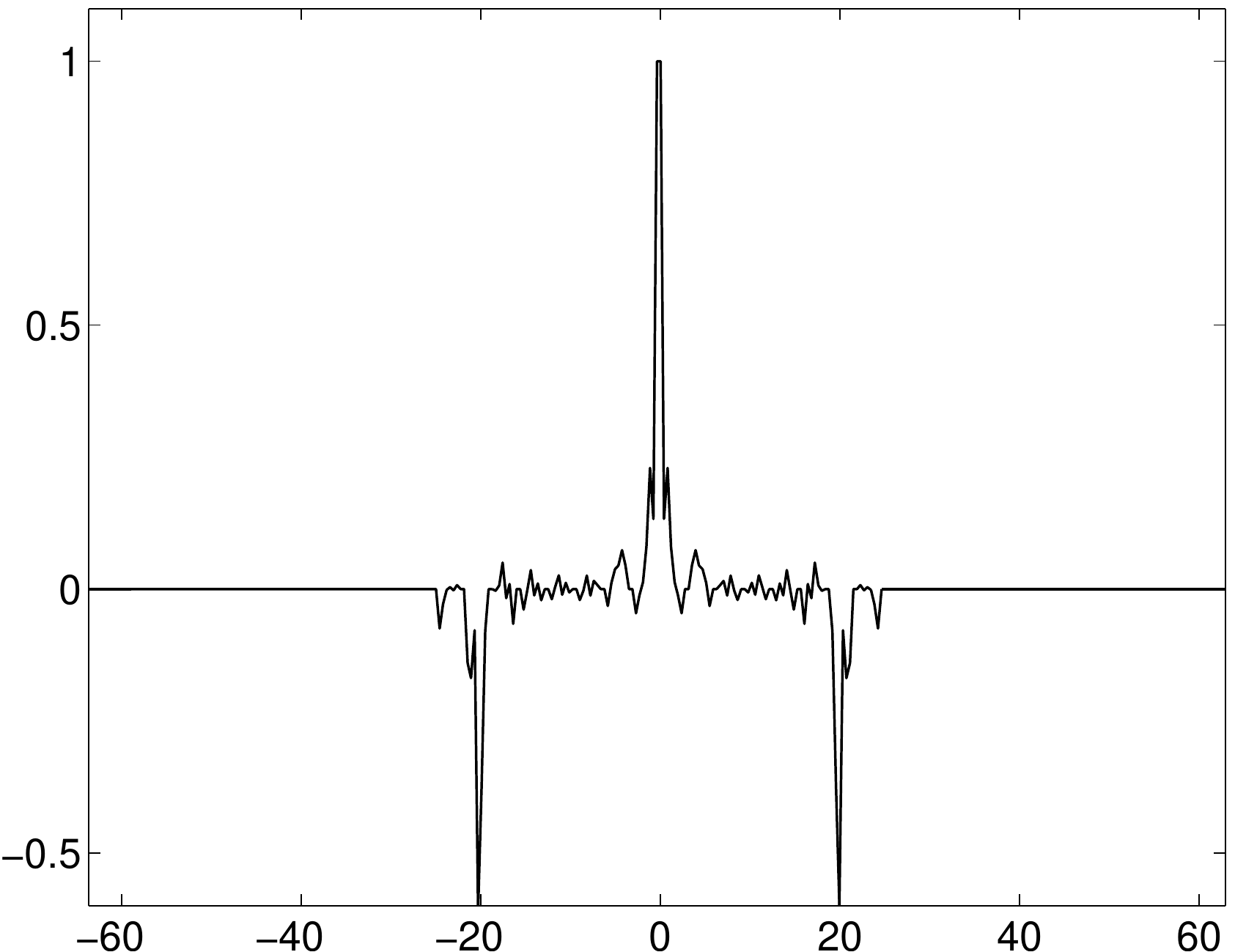}&\includegraphics[width=5cm]{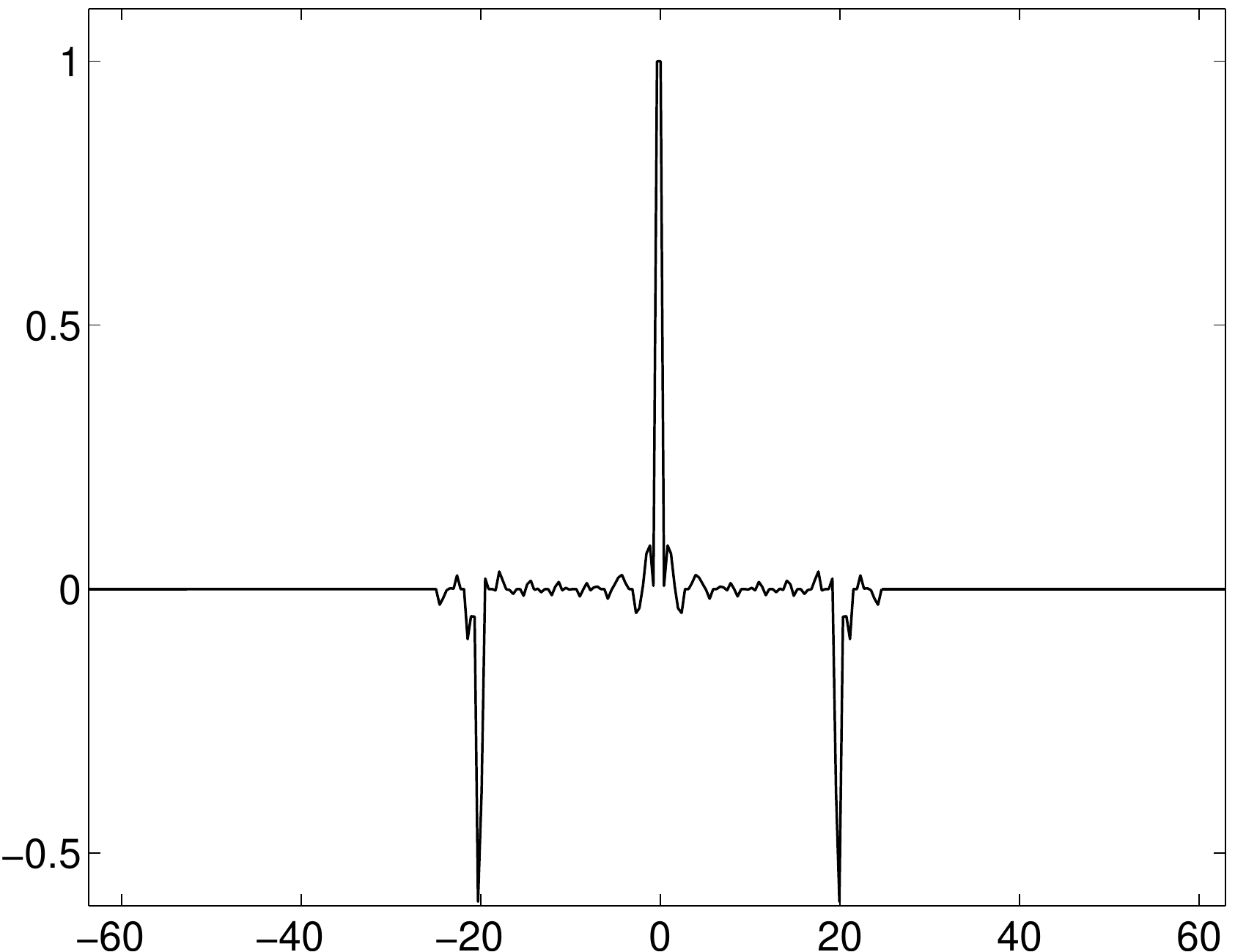}\\
\includegraphics[width=5cm]{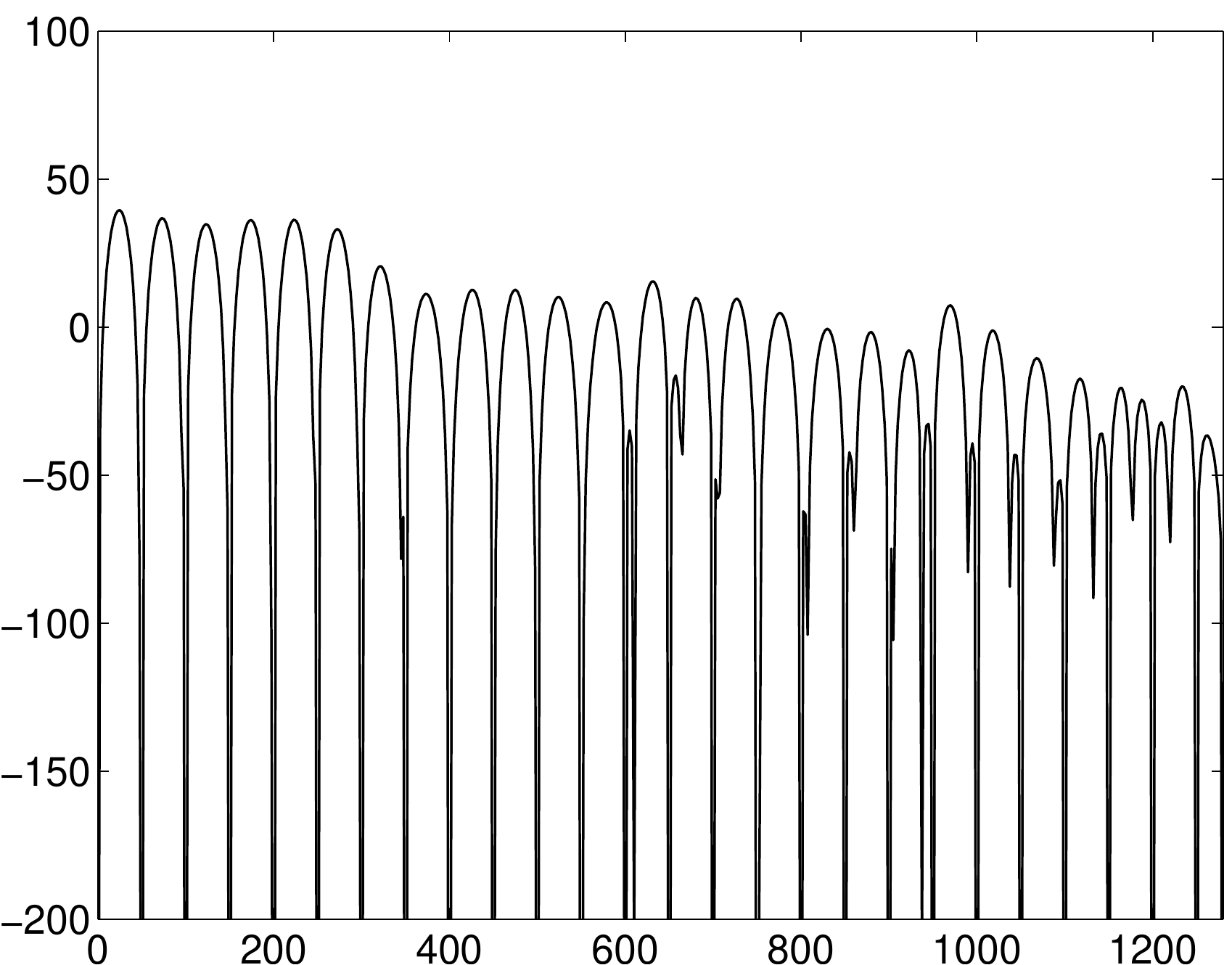} &\includegraphics[width=5cm]{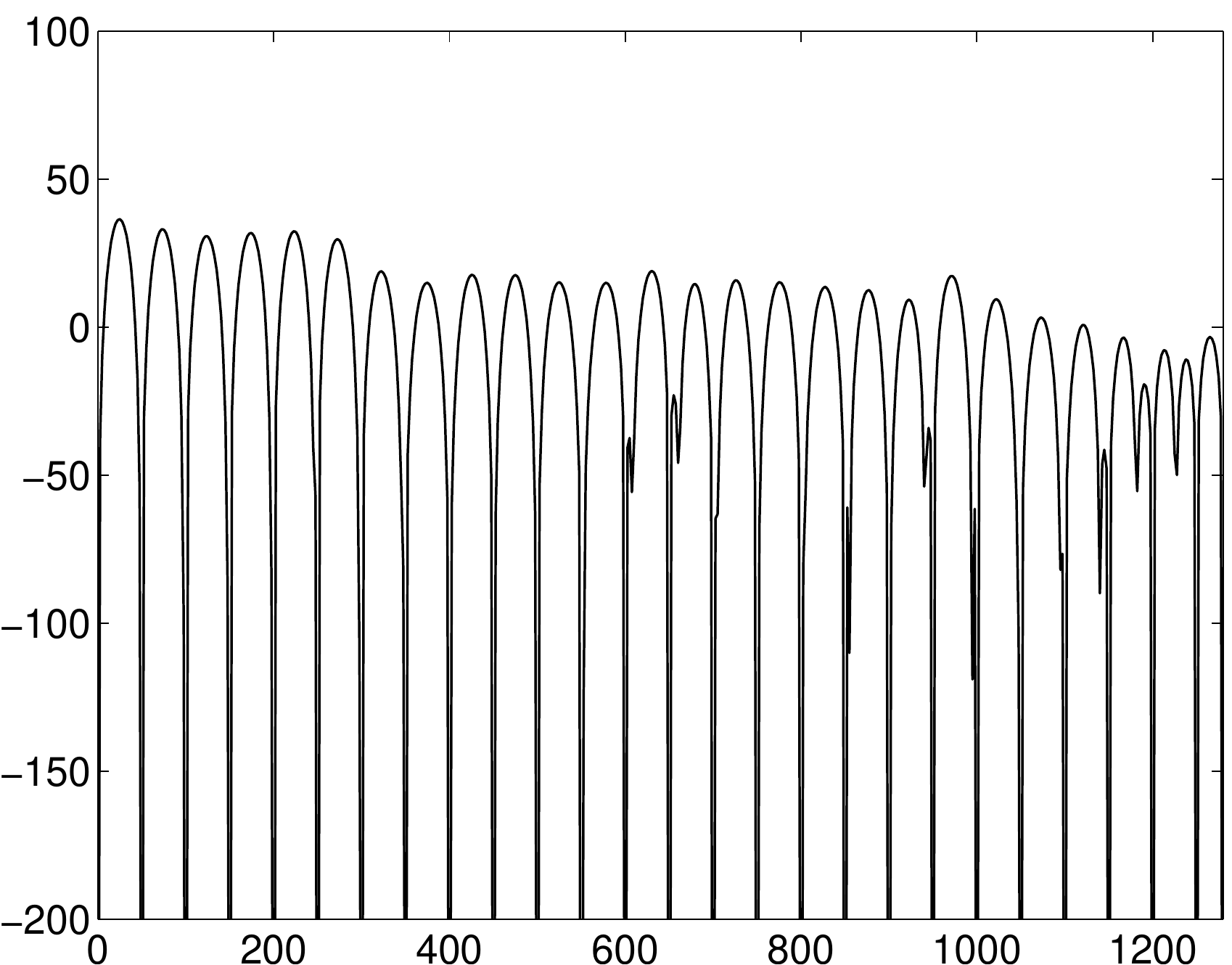} &\includegraphics[width=5cm]{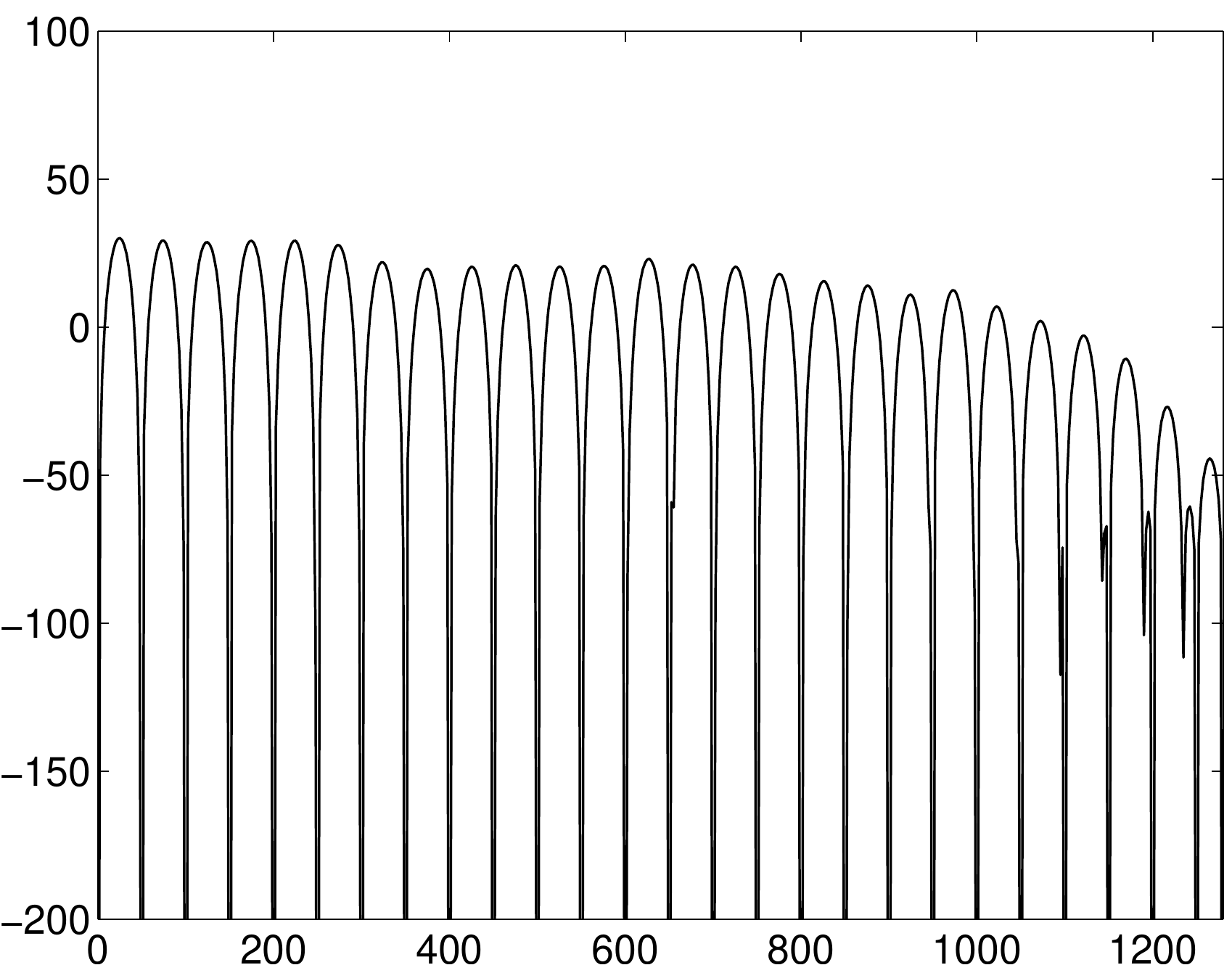}\\
 \end{tabular}
  \caption{Results obtained from \eqref{eq:cdf_epi} for $\beta=2$ and
  $\varepsilon =800 $ (first column), $\varepsilon = 5000$ (second
  column), $\varepsilon = 10000$ (third column). First row: pulse $x$,
  second row: Fourier transform $\chi$ of the pulse.
  \label{fi:res_fil_des_3}}
\end{figure*}


\section{Conclusions}\label{sec:con}
We have proposed a new epigraphical technique to deal with constrained convex variational formulations of inverse problems
with the help of proximal algorithms. In this paper, our attention has been turned to constraints based on distance functions and weighted $\bell_{1,p}$-norms with $p\in \{2,\pinf\}$. In the context of 1D signals,
we have shown that constraints based on distance functions are useful for pulse shape design. In the context of images, we have used $\bell_{1,p}$-norm constraints to promote block-sparsity of analysis representations. The obtained results demonstrate the better performance of non-local measures in terms of image quality. Our results also show that the $\bell_{1,2}$-norm has to be preferred over the $\bell_{1,\infty}$-norm for image recovery problems. However, it would be interesting to consider alternative applications of $\bell_{1,\infty}$-norms such as regression problems \cite{Tibshirani_R_1996_j-r-stat-s-b_regression_ss,Yuan_M_2006_j-r-stat-s-b_model_ser}. Furthermore, the experimental part indicates that the epigraphical method converges faster than the approach based on the direct computation of the projections via standard iterative solutions. Parallelization of our codes should even allow us to accelerate them \cite{Gaetano2012}. Note that, although the considered application involves two constraint sets, the proposed approach can handle an arbitrary number of convex constraints. The epigraphical approach could also be used to develop approximation methods for addressing more general convex constraints.

\section*{References}

\end{document}